\documentclass[10pt, reqno]{amsart}
\usepackage{a4wide}
\usepackage{wrapfig}
\usepackage{amsmath}
\usepackage{amsthm}
\usepackage{mathrsfs}
\usepackage{color}
\usepackage{xcolor}
\usepackage{graphicx}
\usepackage{amssymb}
\usepackage{url}
\usepackage{multicol}
\usepackage{tikz}
\usepackage{amsmath}
\usepackage{fancyhdr}
\usetikzlibrary{matrix}
\usetikzlibrary{arrows}
\usepackage{subfig}
\usepackage{hyperref}
\usepackage{listings}
\usepackage{bbm}
\usepackage{enumitem}
\usepackage{dsfont}
\usepackage{algpseudocode}
\usepackage{algorithm}
\usepackage{pseudocode}
\usepackage{varwidth}
\usepackage{stmaryrd}
\usepackage{dsfont}
\allowdisplaybreaks

\title[Stability and Asymptotic stability]{Nonlinear stability of 2-solitons of the Sine-Gordon equation in the energy space}%Stability and asymptotic stability of sine-Gordon 2-solitons in the energy space}

\author[C. Mu\~noz]{Claudio Mu\~noz$^*$}
\address{CNRS and Departamento de Ingenier\'{\i}a Matem\'atica and Centro
de Modelamiento Matem\'atico (UMI 2807 CNRS), Universidad de Chile, Casilla
170 Correo 3, Santiago, Chile.}
\email{cmunoz@dim.uchile.cl}
\thanks{$^{*}$ C. M. work was funded in part by Chilean research grants FONDECYT  1150202, Fondo Basal CMM-Chile, MathAmSud EEQUADD and Millennium Nucleus Center for Analysis of PDE NC130017. }

\author[J. M. Palacios]{Jos\'e M. Palacios$^{**}$}
\thanks{$^{**}$ J. M. P.  was partially supported by Chilean research grants FONDECYT  1150202, and Millennium Nucleus Center for Analysis of PDE NC130017.}
\address{Jos\'e M. Palacios, Departamento de Ingeniería Matem\'atica DIM, Universidad de Chile}
\email{jpalacios@dim.uchile.cl}
\thanks{Part of this work was carried out while the authors were part of the Focus Program on Nonlinear Dispersive Partial Differential Equations and Inverse Scattering (August 2017) held at Fields Institute, Canada. They would like to thank the Institute and the organizers for their warming support.}

\date{\today}

\newcommand{\be}{\begin{equation}}
\newcommand{\ee}{\end{equation}}
\newcommand{\bp}{\begin{proof}}
\newcommand{\ep}{\end{proof}}
\newcommand{\bel}{\begin{equation}\label}
\newcommand{\eeq}{\end{equation}}
\newcommand{\bea}{\begin{eqnarray}}
\newcommand{\eea}{\end{eqnarray}}
\newcommand{\bee}{\begin{eqnarray*}}
\newcommand{\eee}{\end{eqnarray*}}
\newcommand{\ben}{\begin{enumerate}}
\newcommand{\een}{\end{enumerate}}
\newcommand{\nonu}{\nonumber}
\providecommand{\abs}[1]{\lvert#1 \rvert}
\newcommand{\ve}{\varepsilon}

\newcommand{\R}{\mathbb{R}}
\newcommand{\K}{\mathbb{K}}
\newcommand{\N}{\mathbb{N}}
\newcommand{\Z}{\mathbb{Z}}

\newcommand{\Com}{\mathbb{C}}

\newcommand{\al}{\alpha}
\newcommand{\bt}{\beta}
\newcommand{\ga}{\gamma}

\newcommand{\sech}{\operatorname{sech}}

\newcommand{\re}{\operatorname{Re}}
\newcommand{\ima}{\operatorname{Im}}

\newtheorem{thm}{Theorem}[section]
\newtheorem{cor}[thm]{Corollary}
\newtheorem{lem}[thm]{Lemma}
\newtheorem{prop}[thm]{Proposition}

\newtheorem{defn}[thm]{Definition}

\theoremstyle{remark}
\newtheorem{rem}{Remark}[section]

\definecolor{codegreen}{rgb}{0,0.6,0}
\definecolor{codegray}{rgb}{0.5,0.5,0.5}
\definecolor{codepurple}{rgb}{0.58,0,0.82}
\definecolor{backcolour}{rgb}{0.95,0.95,0.92}

\lstdefinestyle{mystyle}{
	backgroundcolor=\color{backcolour},   
	commentstyle=\color{codegreen},
	keywordstyle=\color{magenta},
	numberstyle=\tiny\color{codegray},
	stringstyle=\color{codepurple},
	basicstyle=\footnotesize,
	breakatwhitespace=false,         
	breaklines=true,                 
	captionpos=b,                    
	keepspaces=true,                 
	numbers=left,                    
	numbersep=5pt,                  
	showspaces=false,                
	showstringspaces=false,
	showtabs=false,                  
	tabsize=2
}
\numberwithin{equation}{section}

\usepackage{pgfplots}
\pgfplotsset{compat=newest}



\theoremstyle{definition}

\numberwithin{ej}{section}

\begin{document}
%\begin{sf}
%==============PORTADA=====================

%==============CAMBIO DE MARGENES=====================
%\newgeometry{top=2.5cm, bottom=2.5cm, right=2.0cm, left=1.8cm}

%==============ENCABEZA/PIE DE PAGINA=====================
%\newpage
%\pagestyle{fancy}
%\fancyhf{}

%Encabezado
%\fancyhead[L]{\rightmark}
%\fancyhead[L]{\small \rm \textit{\curso}} %Izquierda
%\fancyhead[R]{\small \rm \textit{}} %Derecha

%\fancyfoot[L]{\small \rm \textit{\titulo.}} %Izquierda
%\fancyfoot[R]{\small \rm \textbf{\thepage}} %Derecha
%\fancyfoot[C]{\thepage} %Centro

\renewcommand{\sectionmark}[1]{\markright{\thesection.\ #1}}
\renewcommand{\headrulewidth}{0.5pt}
\renewcommand{\footrulewidth}{0.5pt}

%===========================================

\begin{abstract}
In this article we prove that 2-soliton solutions of the sine-Gordon equation (SG) are orbitally stable in the natural energy space of the problem. The solutions that we study are the \emph{2-kink, kink-antikink and breather} of SG. In order to prove this result, we will use B\"acklund transformations implemented by the Implicit Function Theorem. These transformations will allow us to reduce the stability of the three solutions to the case of the vacuum solution, in the spirit of previous results by Alejo and the first author \cite{AM1}, which was done for the case of the scalar modified Korteweg-de Vries equation. However, we will see that SG presents several difficulties because of its vector valued character. Our results improve those in \cite{AMP1}, and give a first rigorous proof of the stability in the energy space of SG 2-solitons. %We show this result without using inverse scattering theory, which requires additional decay assumptions.
\end{abstract}
\maketitle \markboth{Stability of SG 2-solitons} {C. Mu\~noz and J. M. Palacios}
%\tableofcontents

\section{Introduction and Main results}

\subsection{The model} This article considers the sine-Gordon (SG) equation in physical coordinates for a scalar field $\phi$:
\begin{align}\label{sg1}
\phi_{tt}-\phi_{xx}+\sin \phi=0. 
\end{align} 
Here, $\phi =\phi(t,x)$ is a real or complex-valued function, and $(t,x)\in \R^2$. SG has been extensively studied in differential geometry  (constant negative curvature surfaces), as well as relativistic field theory and soliton integrable systems. The interested reader may consult the monograph by Lamb \cite[Section 5.2]{Lamb}, and for more details about the Physics of SG, see e.g. Dauxois and Peyrard \cite{DP}. 

\medskip

Using the standard notation $\vec\phi :=(\phi,\phi_t)$, corresponding to a wave-like dynamics, and given data $\vec \phi(t=0)$, a natural energy space for \eqref{sg1} is $(H^1\times L^2)(\R;\K) $ ($\K=\R$ or $\Com$), as it is revealed by the conservation laws  Energy and Momentum, respectively:
\be\label{Energy}
E[\vec\phi](t)=\dfrac{1}{2}\int_\mathbb{R}(\phi_x^2+\phi_t^2)(t,x)dx+\int_{\mathbb{R}}(1-\cos \phi(t,x))dx=E[\vec\phi](0), 
\ee
and
\be\label{Momentum}
P[\vec\phi](t)=\dfrac{1}{2}\int_\mathbb{R}\phi_t(t,x)\phi_x(t,x)dx=P[\vec\phi](0),
\ee
although spaces slightly different may be considered, using the fact that $\vec\phi$ need not be necessarily zero at infinity for $E$ and $P$ being well-defined. However, real-valued solutions of \eqref{sg1} that initially are in $H^1\times L^2$ are preserved for all time. Additionally, they are globally well-defined thanks to standard Strichartz estimates and the fact that $\sin(\cdot)$ is a smooth bounded function. In what follows, we will assume that we have a real-valued solution of \eqref{sg1} (in vector form) $\vec\phi \in C(\R; H^1\times L^2)$, although complex-valued solutions, or solutions with nonzero values at infinity will be also considered in some places of this paper.

\medskip

Solutions of \eqref{sg1} are known for satisfying several symmetry properties: shifts in space and time, as well as \emph{Lorentz boosts}: for each $\bt\in (-1,1)$, given $\vec\phi(t,x)=(\phi,\phi_t)(t,x)$ solution, then 
\be\label{Lorentz}
(\phi,\phi_t)_\bt(t,x):= (\phi,\phi_t)\big(\ga(t-\bt x),\ga(x-\bt t) \big), \quad \ga:=(1-\bt^2)^{-1/2},
\ee
is another solution of \eqref{sg1}. The parameter $\ga$ is called Lorentz scaling factor, having an important role in what follows.

\subsection{2-soliton solutions} In this article we will show stability of a certain class of particular solutions of 2-soliton type for \eqref{sg1}. In order to explain better the 2-solitons forms that we will study, first we need to understand the notion of 1-soliton. This is an exact solution of \eqref{sg1} usually referred as the \emph{kink} \cite{Lamb}:
\[
Q(x) := 4\arctan (e^{x+x_0}), \quad x_0\in \R.
\]
Thanks to \eqref{Lorentz}, it is possible to define a kink of arbitrary speed $\bt\in (-1,1)$. From the integrability of SG, interactions between kinks are elastic, i.e. they are solitons \cite{Lamb}. Also, $-Q(x)$ is another stationary solution of SG, usually called \emph{anti-kink}. It is well-known that $(Q,0)$ is stable under small perturbations in the energy space $(H^1\times L^2)(\R)$, see Henry-Perez-Wresinski \cite{HPW}. 

\medskip

These kinks are also locally asymptotically stable in the energy space under odd perturbations, a property that follows from the proofs in  \cite{KMM}, as well as some of the methods exposed in this article.

\medskip

A 2-soliton is formally a solution that behaves as the elastic interaction between two forms of 1-soliton, and under different scalings (or speeds, real or complex-valued). This structure remains valid for all time. The 2-solitons considered in this paper are the following (see Lamb \cite[pp. 145--149]{Lamb}):

\medskip

\emph{Notation:} Let $x_1,x_2\in \R$ be shift parameters, $\bt\in (-1,1)$ be a scaling parameter, and $\ga=(1-\bt^2)^{-1/2}$ be the Lorentz factor. We will study

\smallskip 

\ben
\item First of all, the SG \emph{breather} $B=B(t,x) =B(t,x;\bt,x_1,x_2)$ given by 
\be\label{breather0000}
\qquad B(t,x;\bt,x_1,x_2)= 4\arctan \left(\frac{\bt}{\al} \frac{\sin(\al (t+x_1))}{\cosh(\bt (x+x_2))} \right), \quad \al=\sqrt{1-\bt^2}, \quad \beta\neq 0,
\ee
which represents a solution (even in $x+x_2$) which is localized in space and oscillatory in time because of the parameter $\al$. This solution can be made arbitrarily small provided $\bt$ is small, and has energy  $E[B,B_t]= 16\bt$, see \cite{Lamb,AMP1}. Additionally, $B$ is a counterexample to the asymptotic stability property of the vacuum solution under small perturbations (except if perturbations are  odd), as was discussed in \cite{KMM2} (see Fig. \ref{br1}). Similarly, in \cite{AMP1} it was conjectured, thanks to numerical evidence, that this solution is stable.

\smallskip

\item Second, the stability of the \emph{2-kink} $R=R(t,x)$, given by
\be\label{R_0}
R(t,x;\bt,x_1,x_2)= 4\arctan \left( \bt \frac{\sinh(\ga (x+x_2))}{\cosh(\ga (t+x_1))} \right), \quad  \bt \neq 0,
\ee
which represents the interaction of two SG kinks with speeds $\pm\bt$, with limits as $x\to \pm \infty$ equal to $-2\pi$ and $2\pi$ respectively\footnote{Note that the classic $2$-kink should connect the states 0 and $4\pi$, but the subtraction of $2\pi$ to a solution of SG is still a solution.} (i.e., $R$ do not decay to zero). Note that $R$ is \emph{odd wrt  the axis $x=-x_2$}.  See Fig. \ref{2K_fig} for more details.

\smallskip

\item Finally, we shall consider the \emph{kink-antikink} $A=A(t,x)$:
\be\label{A0}
A(t,x;\bt,x_1,x_2)=4\arctan \left(\frac{1}{\bt} \frac{\sinh(\ga (t+x_1))}{\cosh(\ga (x+x_2))}\right), \quad  \bt \neq 0,
\ee
which represents the elastic collision between a SG kink and an anti-kink, with speeds $\pm\bt$. This solution decays to zero at infinity, and it is even wrt $x+x_2$. See Fig. \ref{KaK_fig}.
\een

%\medskip

These three time depending functions are exact solutions of SG that have two modes of independent variables, in contrast with the kink $Q$ which has only one. Another type of degenerate solitons, not treated in this paper, can be found in \cite{CCF}.

\subsection{Main results} The purpose of this paper is to give a first proof of the fact that the three 2-soliton of SG are stable under perturbations well-defined in the natural \emph{energy space} associated to the problem, this without \emph{any additional decay assumption}, and no use of the Inverse Scattering methods. Consecuently, our results extends those of Henry-Perez-Wresinski \cite{HPW} to the case of SG 2-solitons, and allow possible extensions to the case of three or more soltions. Our main theorem is the following:

\begin{thm}[Stability of 2-solitons in the energy space]\label{MT1}
The 2-solitons of SG \eqref{sg1} are nonlinearly stable under perturbations in the energy space $H^1\times L^2$. More precisely, there exist $C_0>0$ and $\eta_0>0$ such that the following holds. Let $(\phi,\phi_t)$ be a solution of \eqref{sg1}, with initial data $(\phi_0,\phi_1)$ such that 
\begin{align}\label{initial_data0}
\Vert (\phi_0,\phi_1)-(D,D_t)(0,~\cdot~;\beta,0,0)\Vert_{H^1\times L^2}< \eta, 
\end{align} 
for some $0<\eta<\eta_0$ sufficiently small, and where $(D,D_t)(t,~\cdot~; \beta , 0,0)$ is a 2-soliton (breather \eqref{breather0000}, 2-kink \eqref{R_0} or kink-antikink \eqref{A0}). Then, there are shifts $x_1(t),x_2(t)\in \R$ well-defined and differentiable such that
\begin{align}\label{final_data0}
\sup_{t\in \R}\Vert (\phi(t),\phi_t(t))-(D,D_t)(t,~\cdot~;\beta,x_1(t),x_2(t))\Vert_{H^1\times L^2}< C_0\eta.
\end{align} 
Moreover, we have 
\[
\sup_{t\in \R} |x_1'(t)| +|x_2'(t)| \lesssim C_0\eta.
\]
\end{thm}

\begin{rem}
Note that in Theorem \ref{MT1} we do not specify the space where $(\phi,\phi_t)$ are posed, this because $(R,R_t)(t)$ in \eqref{R_0} does not belong to $H^1\times L^2$. However, it is possible to show local well-posedness (LWP) in each of the three cases involved in this article, such that $H^1\times L^2$ perturbations are naturally allowed. 
\end{rem}

Rigorous proofs of stability of SG 2-solitons are not known in the literature, as far as we can understand.  Formal descriptions of the dynamics can be found in \cite{EFM}, and in \cite{Sch}, under additional assumptions of rapid decay for the initial data. These last two results are strongly based on the Inverse Scattering Theory (IST), therefore the extra decay is essential. Theorem \ref{MT1} do not require this assumptions, only perturbation data in the energy space (and probably even less regular).

\medskip

A first result on conditional stability (only for the SG breather case) can be found in Alejo et. al. \cite{AMP1}. In this work it was shown that, under certain spectral conditions, breathers are stable under $H^2\times H^1$ perturbations. This result follows some of the ideas in \cite{AM,AM0}, works dealing with the modified KdV case, a simpler breather. Additionally, in the same work, the spectral conditions required in \cite{AMP1} where numerically verified in a large set of parameters for the problem. Theorem \ref{MT1} improves the results in \cite{AMP1} in two senses: first, it establishes the stability of 2-solitons for SG in a rigorous form; and second, the proof works in the energy space of the problem, without any additional assumption.

\medskip

Although 2-solitons are stable, it is known that breathers should disappear under perturbations of the equation itself. In that sense, the literature is huge, from the physical and mathematical point of view. Nonexistence results for breathers can be found in \cite{BMW,Kich,Coron,D,kruskal_segur,Vuillermot}, under different conditions on the nonlinearity. Recently, Kowalczyk, Martel and the first author \cite{KMM2} showed nonexistence of odd breathers for scalar field equations with odd nonlinearities, with no other assumptions on the nonlinearity, except being $C^1$. However, in \cite{BCLS} it was shown existence of breathers in scalar field equations with non-homogeneous coefficients. Finally, \cite{Mu4} considers in a rigorous way the stability question for Peregrine and Ma breathers, showing that they are indeed unstable, even if the equation is locally well-posed. 

\medskip

On the other hand, stability and asymptotic stability results for $N$-solitons of several dispersive nonlinear equations, are largely available in the literature. Concerning the NLS equation, see \cite{Kap,Rod_Soffer_Schlag}. We also refer to the works \cite{PW,MS,MMT,MMarma,MMnon} for the case of solitons and 2-solitons in gKdV equations. The works \cite{SW,KMM} are deeply concerned with scalar field equations, and \cite{NL} deals with the Benjamin-Ono equation and its 2-solitons. See also \cite{PS} for the study of 2-solitons in Dirac type equations. Finally, Alejo et al. \cite{AMP2} worked the case of periodic mKdV breathers.  

\medskip

In this work we extend the ideas introduced in \cite{AM1} to the SG case. More precisely, we will study the B\"acklund Transformations (BT) between two solutions $(\phi,\varphi)$ for SG, and fixed parameter $a$: 
\begin{align*} 
    \varphi_x-\phi_t \ & = \ \dfrac{1}{a}\sin\left(\dfrac{\varphi+\phi}{2}\right)+a\sin\left(\dfrac{\varphi-\phi}{2}\right), 
    \\ 
    \varphi_t-\phi_x \ & = \ \dfrac{1}{a}\sin\left(\dfrac{\varphi+\phi}{2}\right)-a\sin\left(\dfrac{\varphi-\phi}{2}\right).
    \end{align*}
These two equations allow to describe the dynamics of 2-solitons using the reduction of complexity induced by the BT. This ideas has been successfully implemented in several contexts: Hoffman and Wayne \cite{HW} used BT to show abstract results of stability. Next, Mizumachi and Pelinovsky \cite{MizPel} showed $L^2$ stability of the NLS soliton using this approach. The case in \cite{AM1} was the first where a BT was used in the case of breathers. 

\medskip

In the case of SG 2-solitons, the dynamics is more complex than usual, because, unlike mKdV in \cite{AM1}, here we will work with a system for $(\phi,\phi_t)$, and not only scalar equations. This fact makes proofs more involved, in the sense that we must work with systems at every step of the proof.   

\medskip

In order to fix ideas, let us consider the case of the SG breather \eqref{breather0000}. First of all, we will need to work with complex-valued solutions. We will introduce the kink function $(K,K_t)$:
\[
(K,K_t)(t,x):= \left(4\arctan\big(e^{\beta x+i\alpha t}\big), \frac{4i\al e^{\beta x+i\alpha t}}{1+ e^{2(\beta x+i\alpha t)}} \right).
\]
This complex-valued SG solution is connected to zero via a BT of parameter $\beta-i\alpha$. We have (Lemma \ref{back_kink}):
\be\label{caso1}
\begin{aligned}
    K_x \ & = \ \dfrac{1}{\beta-i\alpha}\sin\left(\dfrac{K}{2}\right)+(\beta-i\alpha)\sin\left(\dfrac{K}{2}\right), 
    \\ 
    K_t \ & = \ \dfrac{1}{\beta-i\alpha}\sin\left(\dfrac{K}{2}\right)-(\beta-i\alpha)\sin\left(\dfrac{K}{2}\right).
\end{aligned}
\ee
On the other hand, the complex-valued kink is a singular solution to SG, in the sense that it blows up (in $L^\infty$ norm) in a sequence of times $t_k$, without accumulation point (Remark \ref{Blow_up}).  Even under this problem, it is possible to define a dynamics for perturbations of $(K,K_t)$, for times $t\neq \tilde t_k\sim t_k$, and proving a kind of manifold stability: 

\begin{cor}[Finite codimension stability for the blow-up in $(K,K_t)$]\label{MT2}
Let $(K,K_t)(t)$ be a complex-valued kink profile such that at time $t=0$ does not blow up. For each $(u_0,s_0)\in (H^1\times L^2)(\R; \Com)$ sufficiently small, there is a unique solution of SG  
\[
(\phi,\phi_t)(t) =(\widetilde K, \widetilde K_t)(t) + (u,s)(t), \quad (u,s)(t)\in (H^1\times L^2)(\R; \Com),
\]
where $(\widetilde K, \widetilde K_t)(t) $ is a complex-valued profile suitably modified via modulations in time. This solution is well-defined for each $t\neq \tilde t_k$, a sequence of times  unbounded and without accumulation points, close to each $t_k$. Similarly, this solution blows-up at time $t=\tilde t_k$.
\end{cor}
The advantage of introducing the profiles $(K,K_t)$ in Theorem \ref{MT1} is the following: this profile is connected to the breather $(B,B_t)$ via a new BT of parameter $\bt + i \al$ (Proposition \ref{back_breather}):
\be\label{caso2}
\begin{aligned}
    B_x-K_t \ & = \ \dfrac{1}{\beta+i\alpha}\sin\left(\dfrac{B+K}{2}\right)+(\beta+i\alpha)\sin\left(\dfrac{B-K}{2}\right), 
    \\ 
    B_t-K_x \ & = \ \dfrac{1}{\beta+i\alpha}\sin\left(\dfrac{B+K}{2}\right)-(\beta+i\alpha)\sin\left(\dfrac{B-K}{2}\right). 
\end{aligned}
\ee
An important portion of this article deals with the generalization of these two identities, \eqref{caso1} and \eqref{caso2}, to the case of time-dependent perturbations of the breather $(B,B_t)$. However, this procedure presents several difficulties. First, a correct connection between neighborhoods of the breather and the zero solution. (Proposition \ref{Descenso_global}). The obtained function near zero must be real-valued, otherwise our method does not work (see Theorem \ref{MT3} below). Next, we need to come back to the original solution for any possible time. This step presents several difficulties since in general the BT are not invertible for free and we need to impose additional conditions, in order to find the correct dynamics (Proposition \ref{Ascenso_global}). Another problem comes from the fact that the method falls down when the time $t$ approaches $\tilde t_k$. We need another method for proving stability at those times, based in energy estimates (Subsection \ref{Caso_malo}). Some of these problems were already solved in \cite{AM1} for the mKdV case, however here we propose another method, more intuitive and based in the uniqueness returned by the modulation in time (Corollary \ref{Mod_Dinamica}). Through this article, we will give a rigorous meaning to the diagram of Fig. \ref{Fig:0} which describes the proof of Theorem \ref{MT1}, based in two ``descents'' and two ``ascents'' from perturbations of the breather (or any 2-soliton), to the zero solution, which is orbitally stable thanks to a respective Cauchy theory.

\begin{figure}[h!]
\begin{center}
\begin{tikzpicture}
  \matrix (m) [matrix of math nodes,row sep=3em,column sep=3em,minimum width=3em]
  {
    (B, B_t)(0)+(z_0,w_0) & & & (B,B_t)(t)+(z,w)(t) \\
     (K,K_t)(0)+(u_0,s_0) &  & & (\overline{K},\overline{K}_t)(t)+(\bar{u},\bar{s})(t) \\
    (y_0,v_0) & & & (y,v)(t) \\};
  \path[-stealth]
    
    (m-1-1) 
            edge node [above] {$t$} 
    (m-1-4)
    
    (m-1-1) 
            edge node [below] {modulation}        
    (m-1-4)
    
    (m-2-4) 
            edge node [right] {$\ \beta-i\alpha+\tilde{\delta} $} 
    (m-1-4)
        
    (m-2-1) 
            edge node [left] {$ \beta-i\alpha+\tilde{\delta}\ $}
    (m-3-1)
    
    (m-3-1) 
            edge node [above] {$t$} 
    (m-3-4)
           
    (m-3-1) 
            edge node [below] {GWP} 
    (m-3-4)        
    
    (m-3-4) 
            edge node [right] {$\ \beta+i\alpha+\delta$} 
    (m-2-4)
  
    (m-1-1) 
            edge node [left] {$ \beta+i\alpha+\delta \ $} 
    (m-2-1);
\end{tikzpicture}.
\end{center} 
\caption{Diagram of proof of Theorem \ref{MT1} in the breather case $(B,B_t)$, for times different to $\tilde t_k$. Here, $(\overline{K},\overline{K}_t)(t)$ represents the complex conjugate of the function $(K,K_t)(t)$ at time $t$.}\label{Fig:0}
\end{figure}
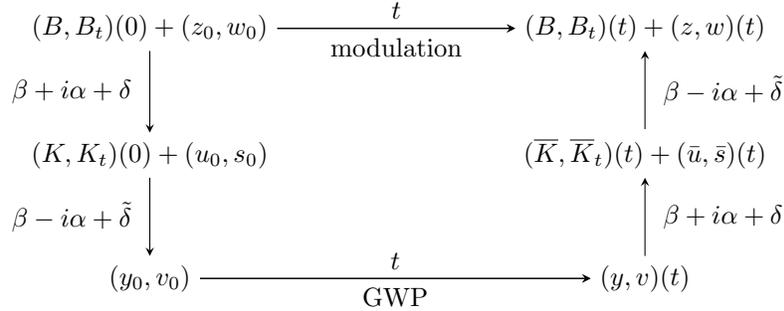

A first consequence of the (rigorous) methods associated to Fig. \ref{Fig:0} is the following:

\begin{thm}[Real-valued character of the double BT]\label{MT3}
Under hypotheses from Theorem \ref{MT1} in the breather case $(B,B_t)$, the LHS of the diagram in Fig. \ref{Fig:0} is well-defined and the functions $(y_0,v_0)\in H^1\times L^2$ obtained are necessarily real-valued, even if $(u_0,s_0)$ are not. 
\end{thm}

For more details about this result, see  Section \ref{8} and Corollary \ref{coro_reales_permutabilidad}. Another consequence of the same diagram in Fig. \ref{Fig:0} is the following method of computing the energy and momentum of each involved perturbation of a 2-soliton:

\begin{cor}[Energy and momentum identities]\label{MT4}
Under the consequences of Theorem \ref{MT1}, and according to the diagram in Fig. \ref{Fig:0}, the following identities are satisfied for each time $t\in \R$:
\begin{align}
E[B+z,B_t+w]&=E[y,v] + 8(\bt +\re\delta)\left( 1 + \dfrac{1}{1 + 2\bt\re\delta + 2\al \ima\delta + |\delta|^2}   \right), \label{EE}
\\ 
P[B+z,B_t+w]&=P[y,v] + 4(\bt +\re\delta)\left(  \dfrac{1}{1 + 2\bt\re\delta + 2\al \ima\delta + |\delta|^2} -1  \right). \label{MM}
\end{align}
Completely similar identities are satisfied by the other 2 cases: $D=A$ or $D=R$, after suitable modifications.
\end{cor}

Finally, but not least, let us mention the fundamental work by Merle and Vega \cite{MV}, who introduced the idea of the nonlinear Miura transformation for the KdV soliton, proving $L^2$ stability. See also \cite{AMV,Mu,Mu3} for other generalizations of this idea to other contexts.

\subsection*{Organization of this article}
Section \ref{2} presents preliminaries that we will need along this paper. Section \ref{3} introduces the basic notions of complex-valued kink profile, and Section \ref{4} describes in detail the 2-soliton profiles. Section \ref{5} deals with modulation of 2-solitons, and Section \ref{6} is devoted to the connection between breathers and the zero solution. In Section \ref{7} we study the corresponding inverse dynamics, while in Section \ref{8} we prove Theorem \ref{MT3}. Section \ref{9} and \ref{10} study the 2-kink and kink-antikink cases, and Section \ref{11} is devoted to the proof of Theorem \ref{MT1} and Corollary \ref{MT4}.  
%A trav'es de este art'iculo denotaremos por $\vec\phi =(\phi,\phi_t)$ una funci'on vectorial a valores en $\R^2$, no necesariamente dependiente del tiempo. Notar que la segunda coordenada $\phi_t$ representa s'olo una notaci'on (cl'asica en el 'area de ecuaciones de ondas), y {\bf no se refiere espec'ificamente} a la derivada temporal de $\phi$. 

%\medskip
%
%Otras cantidades encontradas en este art'iculo:
%
%\ben
%\item 
%\medskip
%\item $H^1\times L^2 =H^1(\R)\times L^2(\R) =(H^1\times L^2 )(\R) $, el correspondiente espacio de energ\'ia donde se toman las perturbaciones incluidas en el Theorem \ref{MT1}.
%\een
%%Para diferenciarla de la derivada parcial temporal, denotaremos esta 'ultima como $\partial_t \phi$, si es que no confusi'on aparece. Por lo tanto, si $\phi=\phi(t,x)$, en general tendremos 
%%\[
%%\phi_t (t,x) \neq (\partial_t \phi)(t,x).
%%\]

\section{Preliminaries}\label{2}

The purpose of this section is to announce a set of simple but fundamental properties that we will need through this article. Proofs are not difficult to establish or being checked in the literature. 

\subsection{B\"acklund Transformation} As a first step, let us write \eqref{sg1} in matrix form, that is $\vec\phi =(\phi,\phi_t)=(\phi_1,\phi_2)$, in such a form that \eqref{sg1} reads now
\be\label{SG}
\begin{cases}
\partial_t \phi_1 =\phi_2\\
\partial_t \phi_2 =\partial_x^2 \phi_1 -\sin\phi_1.
\end{cases}
\ee

Formally speaking, we will say that a \emph{profile} is a function of the form $(\phi_1,\phi_2)(x)$, independent of time, which under a particular time-dependent transformation, may be exact or approximate solution of \eqref{SG} described above. Although not a rigorous definition, this one will allow us to understand in a better form the concepts described below. Now we introduce the B\"acklund transformation that we will use in this article. Recall that $\dot H^1$ represents the closure of $C_0^\infty$ under the norm $\|\partial_x \cdot  \|_{L^2}$.

\begin{defn}[B\"acklund Transformation]\label{defi1}
Let $a\in \Com$ be fixed. Let $\vec\phi=(\phi,\phi_t)(x)$ be a function defined in $\dot H^1(\Com)\times L^2(\Com)$. We will say that $\vec\varphi$ in $\dot H^1(\Com)\times L^2(\Com)$ is a {\bf B\"acklund transformation} (BT) of $\vec \phi$ by the parameter $a$, denoted
\be\label{Flecha}
 \, \mathbb{B}(\vec\phi) \xrightarrow{\ a \ }{\vec\varphi},
\ee
if the triple $(\vec\phi,\vec\varphi,a)$ satisfies the following equations, for all $x\in \R$:
% Esto es, se cumple que si $(t,x)\in I\times \R$,
\begin{align} 
    \varphi_x-\phi_t \ & = \ \dfrac{1}{a}\sin\left(\dfrac{\varphi+\phi}{2}\right)+a\sin\left(\dfrac{\varphi-\phi}{2}\right), \label{b1}
    \\ 
    \varphi_t-\phi_x \ & = \ \dfrac{1}{a}\sin\left(\dfrac{\varphi+\phi}{2}\right)-a\sin\left(\dfrac{\varphi-\phi}{2}\right).\label{b2}
    \end{align}
\end{defn}

\begin{rem}
Note that if the triple $(\vec\phi,\vec\varphi,a)$ satisfies Definition \ref{defi1}, then so $(\vec\varphi,\vec\phi,-a)$ does, and we have $ \mathbb{B}(\vec\varphi) \xrightarrow{\ -a \ }{\vec\phi}.$ In that sense, the order between $\phi$ and $\varphi$ will not play an important role.
\end{rem}

\begin{rem}
Note also that we do not ask for uniqueness for $\varphi$ in Definition \ref{defi1}. However, in this articule we will construct functions $\varphi$ which are uniquely defined as BT (with fixed parameter) of a unique $\phi.$ 
\end{rem}

\begin{rem}[Different BT for SG] In \cite{Lamb} \eqref{sg1} is written in ``laboratory coordinates'' $(u,v)$ given by 
\begin{align*}
    u \ := \ \dfrac{x-t}{2} \ , \ \ \ v \ := \ \dfrac{x+t}{2} \ \ \Longleftrightarrow \ \ x \ = \ u+v \ , \ \ \  t \ = \ v-u.
\end{align*}
Under these new variables SG \eqref{sg1} reads $\sigma_{uv} =\sin \sigma$, where $\sigma(u,v):=\phi(t,x)$. It is not difficult to show that in this case, \eqref{b1}-\eqref{b2} are equivalent to the equations
\begin{align*}
     \dfrac{1}{2}\left(\sigma_u+\widetilde{\sigma}_u\right) \ = \ a\sin\left(\dfrac{\sigma-\widetilde{\sigma}}{2} \right),  \quad \dfrac{1}{2}\left(\sigma_v-\widetilde{\sigma}_v\right) \  = \ \dfrac{1}{a}\sin\left(\dfrac{\sigma+\widetilde{\sigma}}{2} \right),
\end{align*}
which are precisely the BT appearing in \cite{Lamb}.
\end{rem}

The following result is standard in the literature, justifying the introduction of the BT \eqref{b1}-\eqref{b2}.

\begin{lem}
If $(\vec\phi,\vec\varphi)$ are $C^2_{t,x}$ functions related via a BT \eqref{b1}-\eqref{b2}, then both solve \eqref{SG}. 
\end{lem}

\begin{proof} 
By smoothness, it is enough to check that both solve \eqref{sg1}. Now, we prove that $\varphi$ solves SG. We take derivative in \eqref{b1} and \eqref{b2}, so that  
\begin{align*}
\varphi_{tt}-\varphi_{xx}&=\dfrac{1}{2a}(\varphi_t-\varphi_x+\phi_t-\phi_x)\cos\left(\dfrac{\varphi+\phi}{2}\right) \\
& \quad +\dfrac{a}{2}(\phi_t+\phi_x-\varphi_t-\varphi_x)\cos\left(\dfrac{\varphi-\phi}{2}\right)
\\ & = -\sin\left(\dfrac{\varphi-\phi}{2}\right)\cos\left(\dfrac{\varphi+\psi}{2}\right)-\sin\left(\dfrac{\varphi+\phi}{2}\right)\cos\left(\dfrac{\varphi-\phi}{2}\right)
 = -\sin(\varphi).
\end{align*}
Similarly, one easily proves that $\phi$ satisfies SG.%, lo que concluye la demostración.
\end{proof}

Using a standard density argument, the previous result can be extended to solutions defined in the energy space, and satisfying the Duhamel formulation for SG. Now, we will need the following notion, generalization of Definition \ref{defi1}. 

%\subsection{Funcionales de B\"acklund} En lo que sigue estudiaremos la invertibilidad de la transformaci'on de Bäcklund sobre funciones a valores complejos. 

%Sea
%\[
%X(\K):= 
%\begin{cases} 
%H^1(\mathbb{R};\K)\times L^2(\mathbb{R};\K )\times H^1(\mathbb{R};\K)\times L^2(\mathbb{R};\K )\times \K,\\
%H^1(\mathbb{R};\K)\times L^2(\mathbb{R};\K )\times H^1(\mathbb{R};\K)\times L^2(\mathbb{R};\K )\times \K,
%\end{cases}
%\]
%con $\K=\Com$ o bien $\R$
\begin{defn}[B\"acklund Functionals]\label{fperturbacion}
Let $(\varphi_0,\,\varphi_1,\,\phi_0,\,\phi_1,\,a)$ be data in an space $X(\K)$ to be chosen later, with $\K=\Com$ or $\R$. Let us define the functional with vector values $\mathcal F:=(\mathcal F_1,\mathcal F_2)$, where $\mathcal F=\mathcal F(\varphi_0,\,\varphi_1,\,\phi_0,\,\phi_1,\,a) \in L^2(\K) \times L^2(\K)$, given by the system: 
\begin{align}
    \mathcal F_1 \big(\varphi_0,\,\varphi_1,\,\phi_0,\,\phi_1,\,a\big) &:=  \varphi_{0,x}-\phi_1 - \dfrac{1}{a}\sin\left(\dfrac{\varphi_0+\phi_0}{2}\right)-a\sin\left(\dfrac{\varphi_0-\phi_0}{2}\right), \label{f1}
    \\ \mathcal F_2 \big(\varphi_0,\,\varphi_1,\,\phi_0,\,\phi_1,\,a\big) &:=
    \varphi_1-\phi_{0,x}- \dfrac{1}{a}\sin\left(\dfrac{\varphi_0+\phi_0}{2}\right)+a\sin\left(\dfrac{\varphi_0-\phi_0}{2}\right).\label{f2}
\end{align}  
\end{defn}

%\begin{rem}
%Notar que $\mathcal F$ est'a bien definido para datos en el espacio $X(\K)$ arriba definido; mejor a'un, $\mathcal F$ retorna funciones a valores complejas en el espacio $L^2(\K)^2$.   Notar tambi'en que es claro que $\mathcal F$ define un funcional de clase $C^1$ en el espacio $X(\K)$.
%\end{rem}

\subsection{Conserved quantities} The following result establishes a direct relation between the BT \eqref{b1}-\eqref{b2} and the conserved quantities \eqref{Energy}-\eqref{Momentum}, without using the original equation \eqref{SG}.

\begin{lem}[BT and conserved quantities]\label{Energia_Back} 
Let\footnote{Note that not necessarily $\phi,\varphi$ belong to $L^2$.} $(\phi,\phi_t),\,(\varphi,\varphi_t)\in (L^\infty\cap \dot H^1)(\mathbb{R};\mathbb{C})\times L^2(\mathbb{R};\mathbb{C})$ be functions related by a BT with parameter $a$, i.e., such that  
\[
\mathbb{B}(\phi,\phi_t)\xrightarrow{\ a \ }{(\varphi,\varphi_t)}.
\]
Let us additionally assume that 
\be\label{ell}
\ell^+_\pm(t):= \lim_{x\to \pm \infty} \left(1-\cos\left(\dfrac{\varphi+\phi}{2}\right)\right), \quad  \ell^-_\pm(t) :=\lim_{x\to \pm \infty} \left(1-\cos\left(\dfrac{\varphi-\phi}{2}\right)\right),
\ee
are well-defined and finte. Then we have 
\begin{align}
E[\vec{\varphi}]&=E[\vec{\phi}]+\dfrac{2}{a}(\ell^+_+-\ell^+_-)(t)+2a (\ell^-_+-\ell^-_-)(t), \label{energia}
\\ P[\vec{\varphi}]&=P[\vec{\phi}] +\dfrac{1}{a} (\ell^+_+-\ell^+_-)(t) -a (\ell^-_+-\ell^-_-)(t),\label{momento}
\end{align}
where $E$ and $P$ are the corresponding energy and momentum defined in \eqref{Energy}-\eqref{Momentum}.
\end{lem}

A simple consequence of the previous result is the following:

\begin{cor}[Parametric rigidity of BT versus Energy and Momentum]\label{Rigidez}
Under the hypotheses from previous lemma, let us assume in addition that $\phi,\varphi $ are such that  $E[\vec{\varphi}]$, $E[\vec{\phi}]$ and $P[\vec{\varphi}]$ and $P[\vec{\phi}]$  are conserved in time $t\in\R$ (see Subsection \ref{LWP} below for details). Then, if both $(\ell^+_+-\ell^+_-)(t)$ and $(\ell^-_+-\ell^-_-)(t)$ do not depend on time, the parameter ``$a$'' in the BT cannot depend on time.
\end{cor}

\begin{rem}
In general, all solutions considered in this article do satisfy the hypotheses in Corollary \ref{Rigidez}. Even more, if the corresponding limits in \eqref{ell} are constant (our case), then the BT parameter $a$ cannot depend on time.
\end{rem}

\begin{proof}[Proof of Lemma \ref{Energia_Back}]
First we prove that \eqref{energia} holds. For that, adding the squares of equations \eqref{b1} and \eqref{b2}, we have 
\[
\varphi_x^2+\varphi_t^2+\phi_x^2+\phi_t^2-2\big(\varphi_x\phi_t+\varphi_t\phi_x\big)=\dfrac{2}{a^2}\sin^2\left(\dfrac{\varphi+\phi}{2}\right)+2a^2\sin^2\left(\dfrac{\varphi-\phi}{2}\right).
\]
Now, replacing the values of $\varphi_x$ and $\varphi_t$ given by equations \eqref{b1} and \eqref{b2},
\begin{align*}
\varphi_x^2+\varphi_t^2+\phi_x^2+\phi_t^2 &= 2\phi_t\left(\phi_t+\dfrac{1}{a}\sin\left(\dfrac{\varphi+\phi}{2}\right)+a\sin\left(\dfrac{\varphi-\phi}{2}\right)\right)
\\ & \qquad +2\phi_x\left(\phi_x+\dfrac{1}{a}\sin\left(\dfrac{\varphi+\phi}{2}\right)-a\sin\left(\dfrac{\varphi-\phi}{2}\right)\right)
\\ & \qquad  +\dfrac{2}{a^2}\sin^2\left(\dfrac{\varphi+\phi}{2}\right)+2a^2\sin^2\left(\dfrac{\varphi-\phi}{2}\right).
\end{align*}
Simplifying and gathering similar terms,
\begin{align}\label{energia1}
\varphi_x^2+\varphi_t^2-\phi_x^2-\phi_t^2 & = \dfrac{2}{a}\big(\phi_t+\phi_x\big)\sin\left(\dfrac{\varphi+\phi}{2}\right)+2a\big(\phi_t-\phi_x\big)\sin\left(\dfrac{\varphi-\phi}{2}\right) \nonu
\\ & \qquad +\dfrac{2}{a^2}\sin^2\left(\dfrac{\varphi+\phi}{2}\right)+2a^2\sin^2\left(\dfrac{\varphi-\phi}{2}\right).
\end{align}
Now, adding and sustracting $\varphi_x$ in the RHS of \eqref{energia1}, and integrating
\begin{align*}
& \int_\mathbb{R}\varphi_x^2+\varphi_t^2-\phi_x^2-\phi_t^2 \\
& = \dfrac{2}{a}\int_\mathbb{R}\big(\phi_t-\varphi_x\big)\sin\left(\dfrac{\varphi+\phi}{2}\right)+2a\int_\mathbb{R}\big(\phi_t-\varphi_x\big)\sin\left(\dfrac{\varphi-\phi}{2}\right)
\\ & \qquad +\dfrac{2}{a^2}\int_\mathbb{R}\sin^2\left(\dfrac{\varphi+\phi}{2}\right)+2a^2\int_\mathbb{R}\sin^2\left(\dfrac{\varphi-\phi}{2}\right)
\\ & \qquad +\dfrac{4}{a}\int_\mathbb{R}\partial_x\left(1-\cos\left(\dfrac{\varphi+\phi}{2}\right)\right)+4a\int_\mathbb{R}\partial_x\left(1-\cos\left(\dfrac{\varphi-\phi}{2}\right)\right).
\end{align*}
Recall that $\mathbb{B}(\phi,\phi_t)\xrightarrow{\ a \ }{(\varphi,\varphi_t)}$. Using \eqref{ell}, we conclude
%\begin{align} v_1:=\int_\mathbb{R}\partial_x\left(1-\cos\left(\dfrac{\varphi+\phi}{2}\right)\right) ,\qquad v_2:=\int_\mathbb{R}\partial_x\left(1-\cos\left(\dfrac{\varphi-\phi}{2}\right)\right)\label{v1v2}
%\end{align} 
\begin{align}\label{energia2}
& \int_\mathbb{R}\varphi_x^2+\varphi_t^2-\phi_x^2-\phi_t^2 \nonu \\
& =  -4\int_\mathbb{R}\sin\left(\dfrac{\varphi+\phi}{2}\right)\sin\left(\dfrac{\varphi-\phi}{2}\right)+\dfrac{4}{a}(\ell^+_+-\ell^+_-)(t)+4a(\ell^-_+-\ell^-_-)(t).
\end{align}
Lastly, multiplying \eqref{energia2} by $\frac{1}{2}$ and using that $\cos\varphi-\cos\phi=-2\sin (\frac{\varphi+\phi}{2} )\sin (\frac{\varphi-\phi}{2} )$, we arrive to the identity 
\begin{align*}
&\frac12\int_\mathbb{R}\varphi_x^2+\varphi_t^2+\int_\mathbb{R}\big(1-\cos \varphi \big) \nonu\\
& =\frac12\int_\mathbb{R}\phi_x^2+\phi_t^2+\int_\mathbb{R}\big(1-\cos \phi \big)+\dfrac{2}{a}(\ell^+_+-\ell^+_-)(t)+2a(\ell^-_+-\ell^-_-)(t),
\end{align*}
which finally proves \eqref{energia}. Similarly, we will show \eqref{momento}. Multiplying \eqref{b1} and \eqref{b2} we have 
\[
\varphi_x\varphi_t+\phi_x\phi_t-\varphi_x\phi_x-\varphi_t\phi_t=\dfrac{1}{a^2}\sin^2\left(\dfrac{\varphi+\phi}{2}\right)-a^2\sin^2\left(\dfrac{\varphi-\phi}{2}\right).
\]
Replacing $\varphi_x$ and $\varphi_t$ given by \eqref{b1} and \eqref{b2} we obtain 
\begin{align}
\varphi_x\varphi_t&=\phi_x\phi_t+\dfrac{1}{a}\big(\phi_x+\phi_t\big)\sin\left(\dfrac{\varphi+\phi}{2}\right)+a\big(\phi_x-\phi_t\big)\sin\left(\dfrac{\varphi-\phi}{2}\right)\nonumber
\\ & \qquad +\dfrac{1}{a^2}\sin^2\left(\dfrac{\varphi+\phi}{2}\right)-a^2\sin^2\left(\dfrac{\varphi-\phi}{2}\right). \label{momentum2}
\end{align}
Finally, using once again that $\mathbb{B}(\phi,\phi_t)\xrightarrow{ \ a \ }{(\varphi,\varphi_t)}$, multiplying \eqref{momentum2} by $\frac{1}{2}$ and integrating, we get
\[
\dfrac{1}{2}\int_\mathbb{R}\varphi_x\varphi_t=\dfrac{1}{2}\int_\mathbb{R}\phi_x\phi_t+\dfrac{1}{a}(\ell^+_+-\ell^+_-)(t) - a(\ell^-_+-\ell^-_-)(t),
\]
which finally ends the proof.
\end{proof}

\subsection{Local well-posedness}\label{LWP} The purpose of this paragraph is to announce the LWP results that we will need through this article. First of all, note that the energy \eqref{Energy} can be written as
\be\label{New_Energy}
E[\vec\phi](t)=\dfrac{1}{2}\int_\mathbb{R}(\phi_x^2+\phi_t^2)(t,x)dx+\int_{\mathbb{R}} \sin^2\left( \frac{\phi}2\right)(t,x)dx. 
\ee
Then, naturally the largest energy space for SG is $H_{\sin}^1\times L^2 $ \cite{DG}, where
\[
H_{\sin}^1 :=\{ \phi_0 \in \dot H^1 ~ :  ~\sin\phi_0 \in L^2 \}.
\]
Since we will consider small perturbations in this paper, $\phi_0 \in H^1$ small enough implies $\phi_0 \in H_{\sin}^1$.

\begin{thm}[GWP for real-valued data]\label{GWP0}
Let $(\phi_0,\phi_1) \in (H^1\times L^2)(\R)$ be initial data. Then there exists a unique solution $\vec \phi \in C(\R,(H^1\times L^2)(\R))$ (in the Duhamel sense) of \eqref{SG}. Moreover, both the momentum $P$ in \eqref{Momentum} and the energy $E$ in \eqref{Energy} are conserved by the flow, and we have
\begin{align}\label{gwp_estimacion}
\sup_{t\in \R} \|(\phi,\phi_t)(t)\|_{H^1\times L^2} \lesssim \|(\phi_0,\phi_1)\|_{H^1\times L^2},
\end{align}
with involved constants independent of time.
\end{thm}

\begin{proof}
This result is direct from the Duhamel formulation for \eqref{SG}, the conservation of energy, plus the fact that $\sin(\cdot )$ is smooth and bounded if the argument is real-valued.
\end{proof}

We will also need a LWP result for complex-valued initial data.

\begin{thm}[LWP for complex-valued data]\label{GWP1}
Let $(\phi_0,\phi_1) \in (H^1\times L^2)(\Com)$ be complex-valued initial data. Then there exists $T=T((\phi_0,\phi_1))>0$ and a unique solution $\vec \phi \in C((-T,T),(H^1\times L^2)(\Com))$ (in the Duhamel sense) of \eqref{SG}. Moreover, both the momentum  $P$ in \eqref{Momentum} as well as the energy $E$ in \eqref{Energy} are conserved by the flow during $(-T,T)$. 
\end{thm}

\begin{rem}
Note that SG with complex-valued data do have finite time blow-up solutions. See Lemma \ref{indefiniciones} for more details on this problem.
\end{rem}

\begin{proof}
The same proof for the real-valued case works for the complex-valued one. Only global existence is not satisfied.
\end{proof}

Finally, we will need a last result for the case of nontrivial values at infinity, more precisely for the case of the 2-kink $R$ in \eqref{R_0}.

\begin{thm}[Global well-posedness for real valued data with nontrivial values at infinity, see e.g. \cite{MV,DG}]\label{GWP2}
Let $(\phi_0,\phi_1)$ be initial data such that for $R=R(t,x;\bt,x_1,x_2)$ fixed 2-kink as in \eqref{R_0}, and $R_t$ its corresponding time derivative, one has
\[
\| (\phi_0,\phi_1) - (R,R_t)(t=0)\|_{(H^1\times L^2)(\R)}<+\infty.
\]
Then there exists a unique real-valued solution $(\phi,\phi_t)$ for SG such that  $(\phi,\phi_t) - (R,R_t)(t) \in C(\R,(H^1\times L^2)(\R))$ (in the Duhamel sense). Moreover, the momentum $P$ in \eqref{Momentum} as well as the energy $E$ in \eqref{Energy} are conserved by the flow. 
\end{thm}

\section{Real and complex valued kink profiles}\label{3}

\subsection{Definitions} The following concept is standard in the literature.%In what follows, we start with a standard definition.
%\be\label{y1y2}
%y_1 := x_1, \quad y_2 := x+ x_2.
%\ee

\begin{defn}[Real-valued kink profile]\label{Kink0}
%Now, our goal is to see that the complex kink in $(x,t)$ coordinates and his conjugated satisfy the equation. 
Let $\beta\in (-1,1)$, $\beta\neq0$, and $x_0\in \R$ be fixed parameters. we define the real-valued kink profile $\vec Q:=(Q,Q_t)$ with speed $\beta$ as
\begin{align}\label{Q} 
Q(x):= Q(x; \beta, x_0)=4\arctan\big(e^{\ga(x + x_0)}\big), \qquad \ga:= (1-\beta^2)^{-1/2},
\end{align}
and
\be\label{Qt}
Q_t(x):= Q_t(x; \beta,x_0)=    \frac{-4\beta \ga e^{\ga(x + x_0)}}{1+ e^{2\ga(x + x_0)}} = \frac{-2\beta\ga }{\cosh(\ga(x+x_0))}.
\ee
\end{defn}
\begin{rem}\label{Solucion_Exacta}
This profile $(Q,Q_t)$, although not an exact solution of \eqref{SG}, can be understood as follows: for each $(t,x)\in \R^2$, $(t,x) \mapsto (Q,Q_t)(x;\beta,x_0-\beta t)$ is an exact solution of \eqref{SG}, moving with speed $\beta$.
\end{rem}

With small but essential modifications, we introduce a complex-valued version of the previous kink profile.

\begin{defn}[Complex-valued kink profile]\label{Kink1}
%Now, our goal is to see that the complex kink in $(x,t)$ coordinates and his conjugated satisfy the equation. 
Let $\beta\in (-1,1)\setminus\{0\}$, $\al=\sqrt{1-\bt^2}$, be fixed, an consider shift parameters $x_1,x_2\in \R$. We define the complex-valued kink profile $(K,K_t)$ with zero speed as  
\begin{align}\label{kink} 
K(x):= K(x; \bt,x_1,x_2)=4\arctan\big(e^{\beta (x+x_2)+i\alpha x_1}\big),
\end{align}
and
\be\label{Kt}
\begin{aligned}
K_t(x):= K_t(x; \bt,x_1,x_2)= &~ \partial_{x_1}K(x; \bt,x_1,x_2) =  \frac{4i\al e^{\beta (x+x_2)+i\alpha x_1}}{1+ e^{2(\beta (x+x_2)+i\alpha x_1)}}.
\end{aligned}
\ee
\end{defn}

\begin{rem}[Multi-valued profiles]
Note that $K$ is well-defined for all $x\in\R$ as a univalued function with complex values, provided we choose a particular Riemann surface for the $\arctan z$ function. In this article we will assume that  $\arctan$ possesses two branch cuts in $C:=(-i\infty,-i]\cup [i,i\infty)$, in such a way that it remains univalued and analytic in $\Com-C$. However, in this paper this bad behavior will be of no importance, since we will work with functions of type $\sin$, $\cos$, or similar, for which all computation will remain well-defined. See \cite{AM1} for a similar phenomenon.
\end{rem}

\begin{rem}[Singular profile]\label{Blow_up}
Note now that $K_t$ is a function that may be singular for certain values of $x$. More precisely, whenever the condition
\[
e^{2(\beta (x+x_2)+i\alpha x_1)} = -1,
\]
(i.e., $2(\beta (x+x_2)+i\alpha x_1) = i(\pi +2k\pi )$, for some $k\in\Z$), is satisfied. In this case, one has
\be\label{x1_cond}
 x_1 =\frac\pi\al\Big( \frac12  + k\Big) , \quad \hbox{for some} \quad k\in \Z, 
\ee
and if $x= -x_2,$ then $K_t$ is singular. See \cite{AM1} for a similar phenomenon in the mKdV case.
\end{rem}

%\noindent
%Notar que el perfil de kink antes definido permite construir una soluci'on compleja expl'icita de \eqref{sg1}, siempre que tomemos en cuenta las condiciones expuestas en las dos observaciones anteriores. 

\begin{lem}[Blow-up]\label{indefiniciones}
Under the notation in Definition \ref{Kink1}, the function 
\[
(K,K_t)(t):=(K(x; \bt,t+x_1,x_2),K_t(x; \bt,t+ x_1,x_2))
\]
is a smooth solution of SG \eqref{sg1} for all $(t,x_1)$ such that  \eqref{x1_cond} is not satisfied; i.e., outside the countable set of points with no accumulation point: 
\be\label{t_k}
t_k=-x_1+ \frac\pi\al \left( \frac12  + k \right), \quad k\in \Z. 
\ee
\end{lem}

Note that, at each of the points $t_k$, $K_t(t)$ leaves the Schwartz class. Consequently, $K_t(t)$ \emph{blows up in finite time} (in $L^\infty$ norm), as $t$ approaches some $t_k$. 

\begin{proof}
Direct, see Remarks \ref{Solucion_Exacta} and \ref{Blow_up}.
\end{proof}

\subsection{Kink profiles and BT} In what follows, we prove connections between kink profiles and the zero solution in SG. Although some of this results are standard, recall that we prove below not only for exact solutions, but also for profiles which are not exact solutions of SG.
%Antes de continuar necesitaremos de la siguiente definición. 
%
%\begin{defn}
%Sea $v\in(-1,1)$. Definimos el kink real de SG, moviéndose hacia la izquierda con velocidad $v$ como
%\[ Q^+(t,x):=Q(x,v,vt)=4 \arctan \left(e^{\frac{x+vt}{\sqrt{1-v^2}}} \right). \] Así mismo, definimos el kink real de SG, moviéndose hacia la derecha con velocidad $v$ como \[Q^-(x,t):=Q(x,v,-vt).\]
%\end{defn} 

%{\color{blue} Adecuar la demostracion al esquema que escribi para $(K,K_t)$ mas abajo. Ser coherente con notacion, $Q^+$, $Q^-$ y otros creo estan de mas.}

\begin{lem}[Kink as BT of zero]\label{prkk1} Let $(Q,Q_t)$ be a SG kink profile with scaling parameter $\beta \in (-1,1)$, $\beta\neq 0$, and shift $x_0$, see Definition \ref{Kink0}. Then, 
\ben
\item We have the identities
\be\label{sinQ_cosQ}
\sin\left(\frac{Q}{2}\right) =  \sech(\ga(x+x_0)), \quad \cos\left(\frac{Q}{2}\right) = \tanh(\ga(x+x_0)).
\ee
\item For each $x\in \R$, $(Q,Q_t)$ is a BT of the origin $(0,0)$ with parameter 
\be\label{a(beta)}
a= a(\beta):= \left(\frac{1+\beta}{1-\beta}\right)^{1/2}.
\ee
That is,  %$a_1:= (1-v)^{1/2}/(1+v)^{1/2}$. Esto es,
\begin{align*}
    Q_x  = \ \dfrac{1}{a}\sin\left(\dfrac{Q}{2}\right)+a\,\sin\left(\dfrac{Q}{2}\right), \quad &  Q_t  = \ \dfrac{1}{a}\sin\left(\dfrac{Q}{2}\right)-a\,\sin\left(\dfrac{Q}{2}\right).
\end{align*}
%\smallskip
%\item  $Q_x$, $Q_t$ y $\sin (\frac Q2)$ son funciones {\bf pares} con respecto a $x=-x_0$, mientras que $ \cos(\frac{Q}{2})$ es impar con respecto al mismo eje.
\een 
%\begin{align}
%    K_x \ & = \ \dfrac{1}{\beta-i\alpha}\sin\left(\dfrac{K}{2}\right)+(\beta-i\alpha)\sin\left(\dfrac{K}{2}\right), \label{k1}
%    \\ 
%    K_t \ & = \ \dfrac{1}{\beta-i\alpha}\sin\left(\dfrac{K}{2}\right)-(\beta-i\alpha)\sin\left(\dfrac{K}{2}\right), \label{k2}
%\end{align}
%donde $\sin z$ y $\cos z$ están definidos en el plano complejo de la manera usual. 
\end{lem}

%Esquemáticamente, se tiene la relaci\'on (ver \eqref{Flecha})
% \begin{align*}
%     \mathbb{B}(0,0) \ \xrightarrow{\ a\ }{(Q,Q_t)}.    
% \end{align*}
%
%\begin{lem}\label{prkk1} Consideremos $\, Q^+$ el kink real de sine-Gordon moviéndose hacia la izquierda, esto es $$ Q^+:=4 \arctan \left(e^{\frac{x+vt}{\sqrt{1-v^2}}} \right),$$ 
% donde $v\in(0,1)$ es la velocidad del kink. Entonces, se cumple que 
% \begin{align*}
%    Q^+_x \ & = \ \dfrac{1}{a_1}\sin\left(\dfrac{Q^+}{2}\right)+a_1\,\sin\left(\dfrac{Q^+}{2}\right),
%    \\ Q^+_t \ & = \ \dfrac{1}{a_1}\sin\left(\dfrac{Q^+}{2}\right)-a_1\,\sin\left(\dfrac{Q^+}{2}\right),
%\end{align*}
%donde $$a_1=(1-v)^{1/2}/(1+v)^{1/2}.$$ Esquemáticamente tenemos que
% \begin{align*}
%     \mathbb{B}(0) \ \xrightarrow{\ a_1 \ }{ \ 4\cdot\arctan\left(e^{\frac{x+vt}{\sqrt{1-v^2}}}\right)}.
% \end{align*}
% \end{lem}

\begin{proof} 
Direct.% Ver Ap\'endice \ref{A1} para m'as detalles.
\end{proof}

\begin{rem}[Antikink and kink with opposite speeds]\label{Cambios de signo} 
Note that, thanks to Lemma \ref{prkk1}, both 
\[
(Q,Q_t)(x; -\beta,x_0)  \quad \hbox{and}\quad (Q,Q_t)( -x; -\beta,x_0),
\]
obey respective BT with properly chosen parameters. Indeed, for 
\be\label{a2_a3}
\begin{aligned}
a_2:= a(-\beta) = \dfrac{(1-\beta)^{1/2}}{(1+\beta)^{1/2}}, \qquad a_3:=  {} - a(\beta)= -\dfrac{(1+\beta)^{1/2}}{(1-\beta)^{1/2}},
\end{aligned}
\ee
we obtain
\be\label{coneccion_Q^+}
\mathbb{B}(0,0) \ \xrightarrow{\ a_2 \ }{(Q,Q_t)(x; -\beta,x_0)},
\quad \mathbb{B}(0,0) \ \xrightarrow{\ a_3 \ }{(Q,Q_t)(-x; -\beta, x_0)}.
\ee
These two profiles will be important in the next sections, when studying the dynamics of the kink-antikink and 2-kink respectively. 
\end{rem}

Now we deal with the case of complex-valued profiles. Here, we need additional conditions in order to ensure smooth functions in space.

\begin{lem}\label{back_kink} Let $(K,K_t)$ be a complex-valued kink profile, with scaling parameter $\beta \in (-1,1)\setminus\{0\}$ and shifts $x_1,x_2$, just as in Definition \ref{Kink1}, and such that \eqref{x1_cond} do not hold. Then, 
\smallskip
\ben
\item We have the identities
\be\label{sinK_cosK}
\begin{aligned}
\sin\left(\dfrac{K}{2}\right) = \sech(\beta (x+x_2)+i\alpha x_1)) , \quad  \cos\left(\dfrac{K}{2}\right) = \tanh(\beta (x+x_2)+i\alpha x_1)).
 \end{aligned}
\ee
\item For each $x\in \R$, $(K,K_t)$ is a BT of the origen $(0,0)$, with parameter $\bt- i\al$ (and where $\al^2 +\bt^2 =1$). That is to say,
\begin{align}
    K_x \ & = \ \dfrac{1}{\beta-i\alpha}\sin\left(\dfrac{K}{2}\right)+(\beta-i\alpha)\sin\left(\dfrac{K}{2}\right), \label{k1}
    \\ 
    K_t \ & = \ \dfrac{1}{\beta-i\alpha}\sin\left(\dfrac{K}{2}\right)-(\beta-i\alpha)\sin\left(\dfrac{K}{2}\right), \label{k2}
\end{align}
where $\sin z$ and $\cos z$ are defined in the complex place as usual. 
\medskip
\item Moreover, $K_x$,  $iK_t$, $\sin(K/2)$ and $i \cos(K/2)$ posses even real part and odd imaginary part, with respect to the axis $x=-x_2$.
\een
\end{lem}

%\begin{rem}
%El lema anterior nos dice que esquem\'aticamente tenemos que 
%\[
%\mathbb{B}(0,0) \  \xrightarrow{ \ \ a \ = \ \beta-i\alpha \ \ }{ \ (K,K_t)},
%\]
%para todos los $x_1$ para los cuales \eqref{x1_cond} no se cumple.
%\end{rem}

\begin{proof}[Proof of Lemma \ref{back_kink}] 
We prove first that $K$ satisfies \eqref{k1}. Indeed, from \eqref{kink} we have 
\be\label{Kxxx}
 K_x  = \dfrac{4\beta e^{\beta (x+x_2)+i\alpha x_1}}{1+e^{2\beta (x+x_2)+2i\alpha x_1}} = \frac{2\beta}{\cosh(\beta (x+x_2)+i\alpha x_1)}.
\ee
Using that $\cosh(a+ib)= \cosh(a) \cos(b) + i \sinh(a) \sin(b) $, we obtain 
\be\label{Kx_deco}
\begin{aligned}
 K_x = &~  \frac{2\beta}{\cosh(\beta (x+x_2)) \cos(\alpha x_1) + i\sinh(\bt(x+x_2)) \sin(\al x_1)} \\
 =&~  \frac{2\beta(\cosh(\beta (x+x_2)) \cos(\alpha x_1) - i\sinh(\bt(x+x_2)) \sin(\al x_1)) }{\cosh^2(\beta (x+x_2)) \cos^2(\alpha x_1) + \sinh^2(\bt(x+x_2)) \sin^2(\al x_1)}.
\end{aligned}
\ee
Therefore, $\re K_x$ is even wrt $-x_2$ and $\ima K_x$ is odd wrt $-x_2$.

\medskip

On the other hand, since  $\alpha^2+\beta^2=1$, we have $\frac1{\bt-i\al} +\bt- i\al = \bt+i\al + \bt -i\al = 2\bt$, and the RHS of  \eqref{k1} reads 
\begin{align*}
    \text{RHS}(\eqref{k1}) & = \ 2\beta\sin\left(\dfrac{K}{2}\right) 
    %\\ & = \ 2\beta\sin\left(2\arctan e^{\beta (x+x_2)+i\alpha x_1}\right) 
%    \\ & = \ 4\beta\sin\left(\arctan e^{\beta x+i\alpha t}\right)\cos\left(\arctan e^{\beta x+i\alpha t}\right)
  = \ 4\beta\dfrac{\sin(\arctan e^{\beta (x+x_2)+i\alpha x_1})}{\cos(\arctan e^{\beta (x+x_2)+i\alpha x_1})}\cos^2(\arctan e^{\beta (x+x_2)+i\alpha x_1})
%    \\ & = \ 4\beta e^{\beta x+i\alpha t} \cos^2\left(\arctan e^{\beta x+i\alpha t}\right)
  \\ & = \dfrac{4\beta e^{\beta (x+x_2)+i\alpha x_1}}{1+e^{2(\beta (x+x_2)+i\alpha x_1)}} = \frac{2\bt}{ \cosh(\beta (x+x_2)+i\alpha x_1)}.
\end{align*}
Similar to \eqref{Kx_deco}, we can conclude that $\sin (K/2)$ has real part even and imaginary part odd wrt to $x=-x_2.$ Finally, note that 
\[
\begin{aligned}
\cos\Big(\frac K2\Big) = &~ \tanh(\beta (x+x_2)+i\alpha x_1) = \frac{\tanh(\beta (x+x_2)) + i \tan(\alpha x_1)}{1+ i\tanh(\bt(x+x_2))\tan(\alpha x_1)}\\
=&~ \frac{ \tanh(\beta (x+x_2)) \sech^2(\al x_1) + i\sech^2(\bt(x+x_2)) \tan(\alpha x_1)  }{1+ \tanh^2(\bt(x+x_2))\tan^2(\alpha x_1)}.
\end{aligned}
\]
%$(1- i\tanh(\bt(x+x_2))\tan(\alpha x_1))(\tanh(\beta (x+x_2)) + i \tan(\alpha x_1))$
Therefore, $\cos(\frac K2)$ has odd real part and even imaginary part (wrt $-x_2$). This ends the proof of \eqref{k1}. 

\medskip

Now, in order to show that \eqref{k2} is satisfied, it is enough to see that from the definition in \eqref{Kt},
\begin{align*} 
K_t &= \dfrac{4i\alpha e^{\beta (x+x_2)+i\alpha x_1}}{1+e^{2(\beta (x+x_2)+i\alpha x_1)}} =\dfrac{i\alpha}{\beta}K_x = 2i\alpha\sin\left(\dfrac{K}{2}\right),
\end{align*}
which proves the result, since $\frac1{\bt-i\al} -(\bt-i\al) = \bt +i\al -\bt +i\al = 2i\al.$ The parity of $K_t$ is direct from that of $K_x$.
\end{proof}

%Este resultado será muy importante al momento de entender la dinámica del breather. Además, será esencial en la demostración del \textit{Theorem de permutabilidad} (ver tambi'en Theorem \ref{MT3}), paso clave en la demostración del Theorem principal. 
Let $(\overline{K},\overline{K}_t)$ denote the complex-valued kink profile of parameters $\beta$ and $-\alpha$, i.e.,
\begin{align}\label{conj_kink} 
\overline{K}(x)&=\overline{K}(x;\beta,x_1,x_2):=\,4\arctan\big(e^{\beta (x+x_2)-i\alpha x_1}\big), \qquad \hbox{and}
\end{align}
\[
\overline{K}_t(x)=\overline{K}_t(x;\beta,x_1,x_2):=-\dfrac{4i\alpha e^{\beta(x+x_2)-i\alpha x_1}}{1+e^{2(\beta(x+x_2)-i\alpha x_1)}}.
\]
\begin{cor}\label{back_kink_conj}  Let $(\overline{K},\overline{K}_t)$ be a SG conjugate kink profile, with scaling parameter $\beta \in (-1,1)\setminus\{0\}$ and shifts $x_1,x_2$, as in \eqref{conj_kink}, and such that \eqref{x1_cond} do not hold. Then, for each $x\in \R$, $(\overline{K},\overline{K}_t)$ is a BT of the origen $(0,0)$ with parameter $\bt+ i\al$:
\begin{align*}
     \overline{K}_x \ & = \ \dfrac{1}{\beta+i\alpha}\sin\left(\dfrac{\overline{K}}{2}\right)+(\beta+i\alpha)\sin\left(\dfrac{\overline{K}}{2}\right),
    \\ 
    \overline{K}_t \ & = \ \dfrac{1}{\beta+i\alpha}\sin\left(\dfrac{\overline{K}}{2}\right)-(\beta+i\alpha)\sin\left(\dfrac{\overline{K}}{2}\right).
\end{align*}
\end{cor}
%\begin{rem}
%El corolario anterior nos dice que esquemáticamente tenemos que
%\[
%\mathbb{B}(0,0) \  \xrightarrow{ \ \ a \ = \ \beta+i\alpha \ \ }{ (\overline{K},\overline{K}_t)}.
%\]
%\end{rem}

\begin{proof}
Direct from Lemma \ref{back_kink} after conjugation of \eqref{k1} and \eqref{k2}.
\end{proof}

%\medskip

\section{2-soliton profiles}\label{4}

\subsection{Definitions} 
With a small abuse of notation (wrt the exact solutions of SG \eqref{breather0000}-\eqref{R_0}-\eqref{A0}, denoted in the same form), we will introduce profiles of 2-soliton solutions. The following definition is standard, see e.g. \cite{AMP1}.

\begin{defn}[Static breather profile]\label{Perfil_B}
Let $\beta \in (-1,1)$, $\beta\neq 0$, and $x_1,x_2\in\mathbb{R}$ be fixed parameters. We define the static breather profile as 
\be\label{perfil_Breather}
B:=B(x;\beta,x_1,x_2):= 4\arctan\left( \dfrac{\beta}{\alpha}\dfrac{\sin(\alpha x_1)}{\cosh (\beta (x+x_2))}\right), \ \ \ \alpha := \sqrt{1-\beta^2}.
\ee 
We also define the ``time-derivative profile'' as
\be\label{perfil_dtBreather}
B_t :=B_t (x;\beta,x_1,x_2):=   \frac{4\al^2\bt \cos (\al x_1) \cosh(\bt (x+x_2))}{\al^2 \cosh^2(\bt (x+x_2)) + \bt^2 \sin^2(\al x_1)}. 
\ee
Finally, note that  $B_t$ vanishes only if $x_1$ satisfies \eqref{x1_cond}.
\end{defn}

%A continuación se muestra una imágen de la solución recien encontrada. 

\begin{rem}
Note that from the previous definition we can recover the standing SG breather \cite{Lamb,AMP1} if we put $t+ x_1$ instead of $x_1$:
\be\label{B_solucion}
B(t,x) = 4\arctan\left( \dfrac{\beta}{\alpha}\dfrac{\sin(\alpha (t+x_1))}{\cosh (\beta (x+x_2))}\right), \ \ \ \alpha := \sqrt{1-\beta^2},
\ee
and similar for $B_t(t,x)$ (see Fig. \ref{br1}).
\end{rem}

%\begin{rem}
%La soluci\'on \eqref{B_solucion} puede ser tambi'en ``empujada'' por una velocidad relativista $v\in (-1,1)$, $v\neq 0$, usando la transformaci'on de Lorentz cl'asica \eqref{Lorentz}:
%\[
%B_v(t,x) := B(\ga(t-vx),\ga(x-vt)), \quad \ga:=(1-v^2)^{-1/2}, 
%\]
%y de manera similar en el caso de $B_t$. Esta nueva funci\'on describe un breather movi'endose con velocidad $v \neq 0$. Sin embargo, en el resto de este art'iculo asumiremos (s'olo para el caso de los breathers) que $v=0$, por simplicidad en los c'alculos (ver \cite{AMP1} para el caso general). Sin embargo, remarcamos que nuestros resultados tambi'en funcionar'an para el caso $v\neq 0$. 
%\end{rem}

\begin{figure}[h!] \centering
{\includegraphics[scale=0.14]{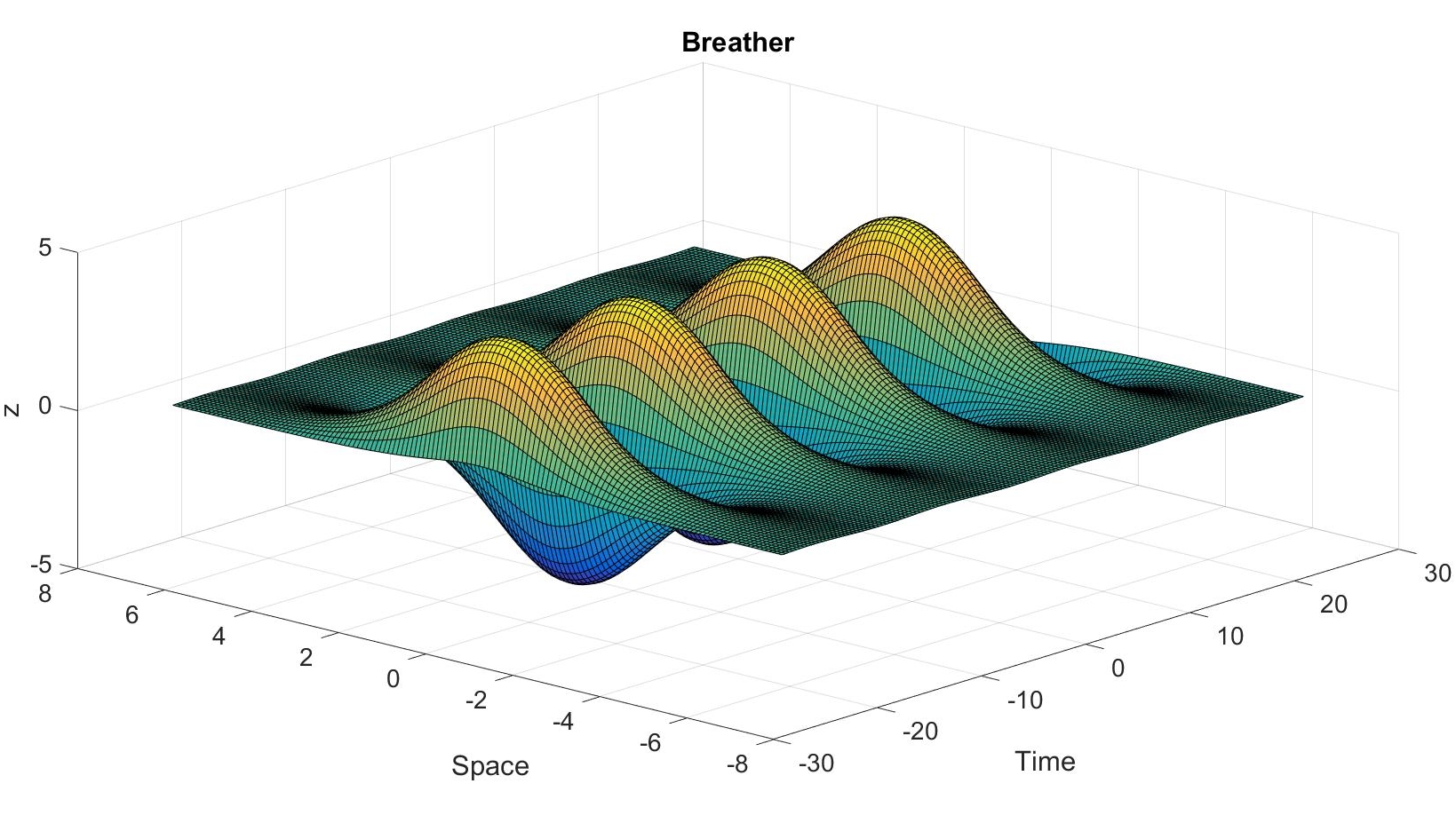}} 
%\hspace{0.15cm} %Espacio horizontal
{\includegraphics[scale=0.14]{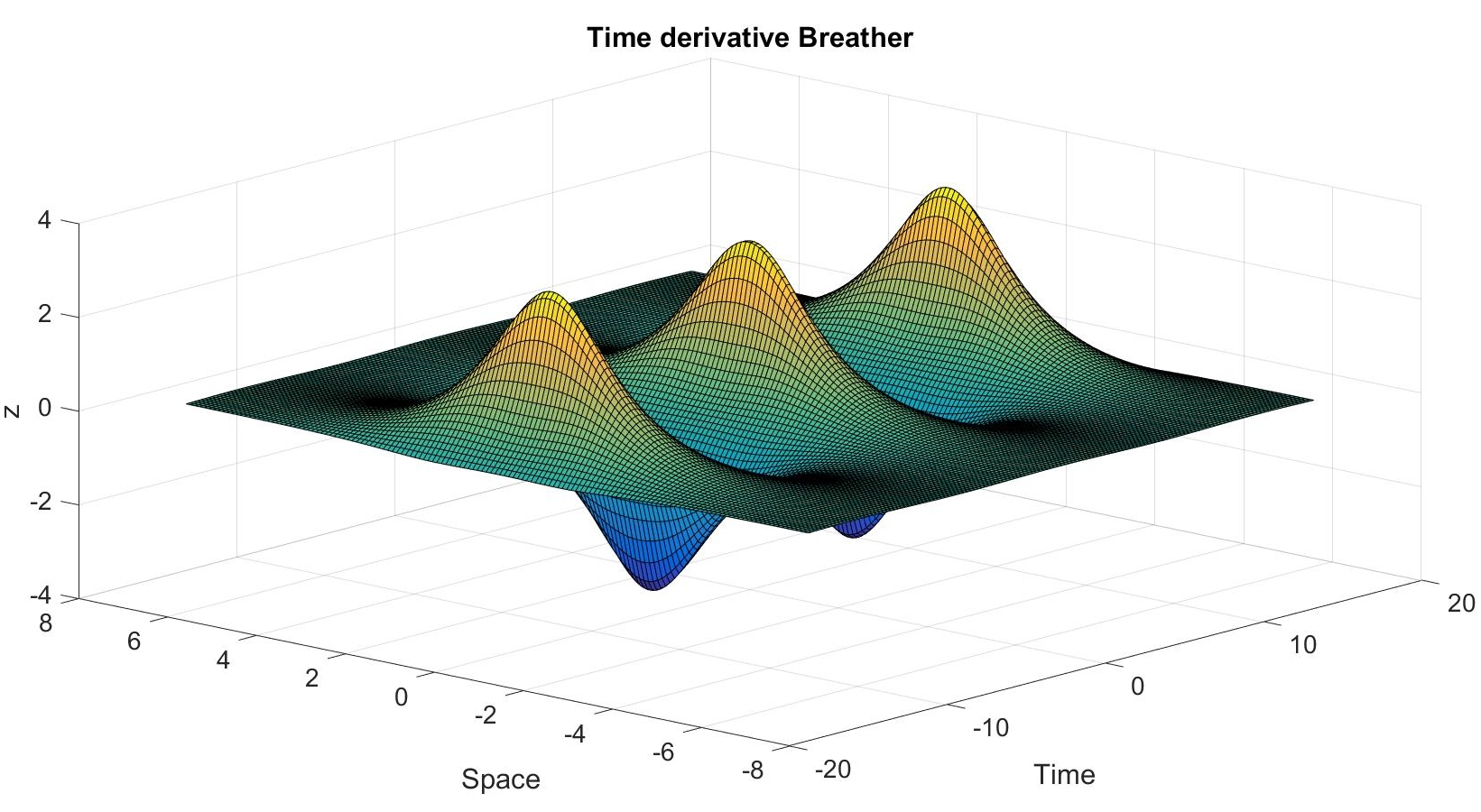}}
\caption{Static breather profile $(B,B_t)$, defined in  \eqref{perfil_Breather} with $\al =\frac12$, $\bt = \frac{\sqrt{3}}2$ and $x_1=t$. Above, $B$, and below, $B_t$. Under these parameters, $(B,B_t)$ is an exact solution for SG as in \eqref{breather0000}.}\label{br1}
\end{figure}

In what follows, we want to study the remaining two SG 2-solitons. Recall that $R(t,x)$ and $A(t,x)$ represent the $2$-kink and kink-antikink, respectively, see \eqref{R_0} and \eqref{A0}. Once again, with a small abuse of notation, we define first the generalized associated profile for the 2-kink.  

\begin{defn}[2-kink profile]\label{Perfil_2K}
Let $\beta \in (-1,1)$, $\beta\neq 0$, and $x_1,x_2\in\mathbb{R}$ be fixed parameters. We define that 2-kink profile with speed $\beta$ as
\be\label{perfil_2Kink}
R:=R(x;\beta,x_1,x_2):= 4\arctan \left(\dfrac{\beta \sinh (\ga (x+x_2)) }{\cosh (\ga x_1)} \right), \ \ \ \ga := (1-\beta^2)^{-1/2}.
\ee 
We also define the ``time derivative profile'' $R_t$ by 
\be\label{perfil_dt_2Kink}
R_{t} :=R_{t} (x;\beta,x_1,x_2):= -\dfrac{4\beta^2\gamma\sinh(\gamma(x+x_2))\sinh(\gamma x_1)}{\cosh^2(\gamma x_1)+\beta^2\sinh^2(\gamma(x+x_2))} . % \frac{-4\al^2\bt \sin (\al x_1) \cosh(\bt (x+x_2))}{\al^2 \cosh^2(\bt (x+x_2)) + \bt^2 \cos^2(\al x_1)}. 
\ee
Note that $(R,R_t)$ is odd wrt $x=-x_2$.
\end{defn}

\begin{rem}\label{2k_solucion}
The SG 2-kink solution $R(t,x)$ \cite{Lamb} written in \eqref{R_0} can be recovered if $x_1$ is replaced by $x_1 + \beta t$ in \eqref{perfil_2Kink}. Fig. \ref{2K_fig} shows the evolution of this exact SG solution in time.
\end{rem}

\begin{figure}[h!] \centering
{\includegraphics[scale=0.141]{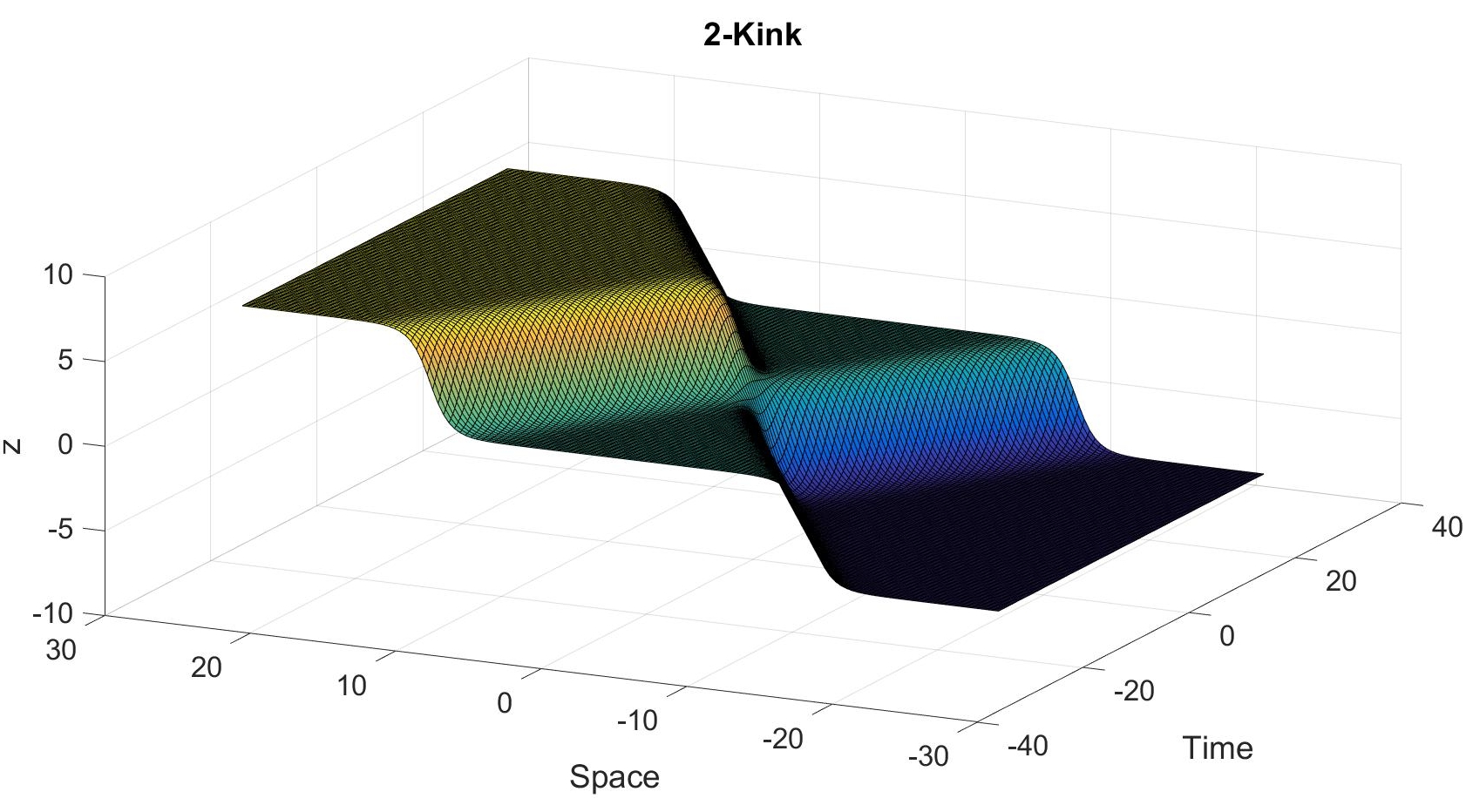}} 
%\hspace{0.15cm} %Espacio horizontal
{\includegraphics[scale=0.141]{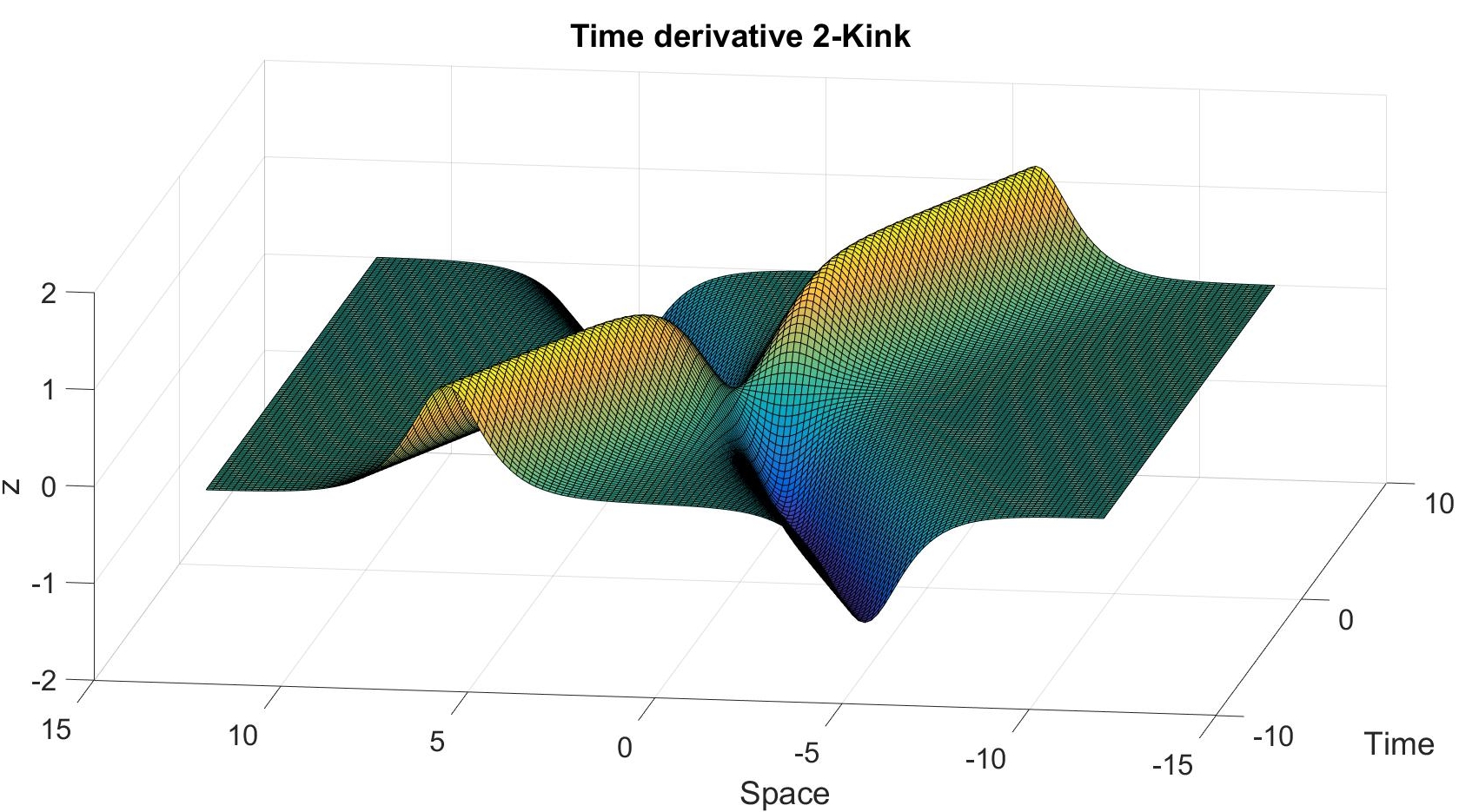}}
\caption{Above: space-time evolution of a 2-kink $R$ with parameters $\beta=\frac{1}{2}$, $x_2=0$ and $x_1 = \bt t$;  below: its corresponding time derivative $R_t$. Here $(R,R_t)$ is an exact solution of SG \eqref{sg1}, see \eqref{R_0}.}\label{2K_fig}
\end{figure}

Finally, with a slight abuse of notation wrt \eqref{A0}, we define the kink-antikink profile.

\begin{defn}[kink-antikink profile]\label{Perfil_KaK}
Let $\beta \in (-1,1)$, $\beta\neq 0$ and $x_1,x_2\in\mathbb{R}$ be fixed parameters. We define the kink-antikink profile with speed $\beta$ by
\be\label{perfil_KaK}
A:=A(x;\beta,x_1,x_2):= 4\arctan \left(\dfrac{ \sinh (\ga x_1) }{\beta\cosh (\ga (x+x_2))} \right), \ \ \ \ga := (1-\beta^2)^{-1/2}.
\ee 
We also define the ``time derivative profile'' $A_t$ as follows: 
\be\label{perfil_dt_KaK}
A_t :=A_t (x;\beta,x_1,x_2):= \dfrac{4\beta^2\gamma \cosh(\gamma(x+x_2))\cosh(\gamma x_1)}{\beta^2\cosh^2(\gamma(x+x_2))+\sinh^2(\gamma x_1)} . % \frac{-4\al^2\bt \sin (\al x_1) \cosh(\bt (x+x_2))}{\al^2 \cosh^2(\bt (x+x_2)) + \bt^2 \cos^2(\al x_1)}. 
\ee
Note that  $(A,A_t)$ are even wrt $x=-x_2$.
\end{defn}

\begin{figure}[h!] \centering
{\includegraphics[scale=0.14]{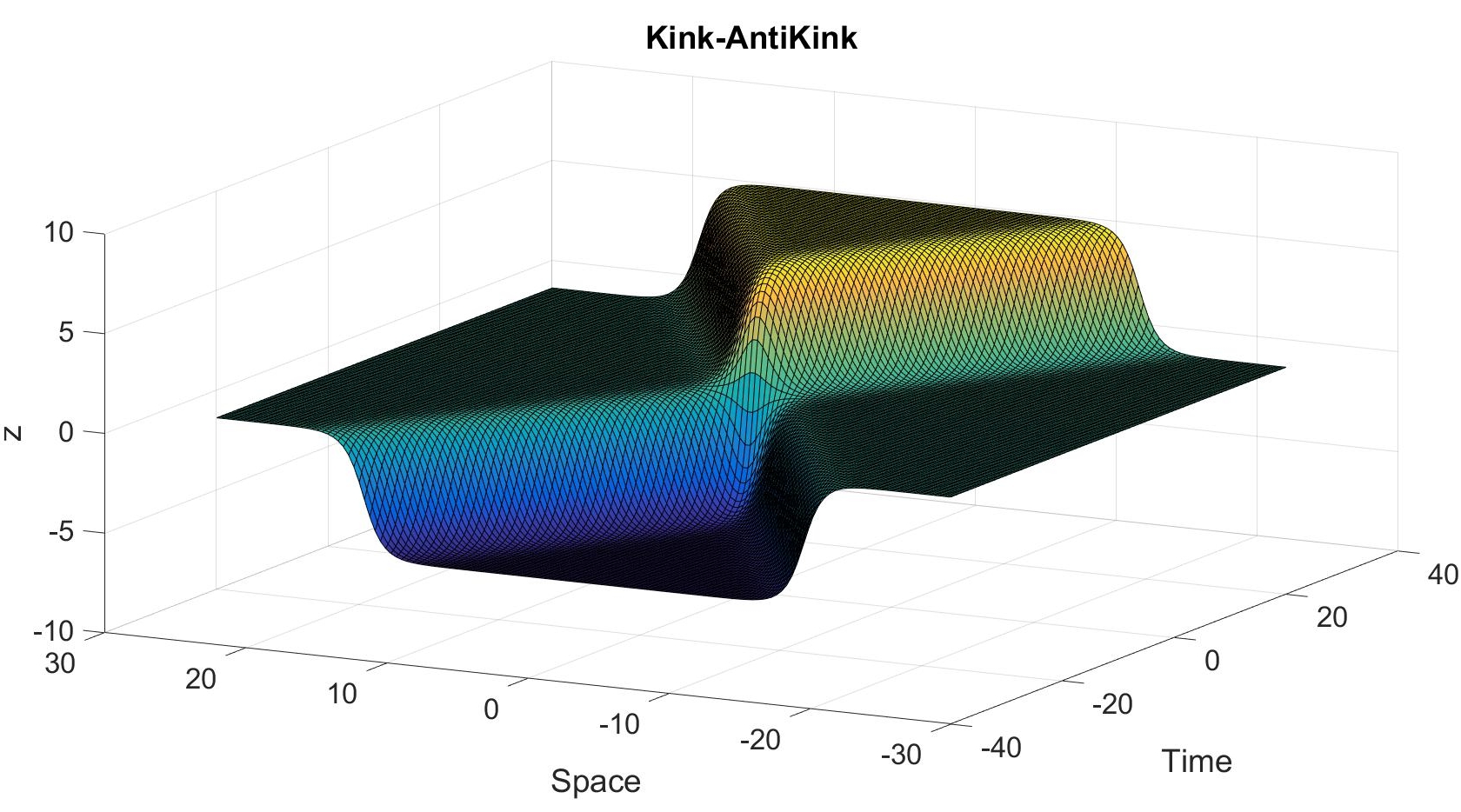}} 
%\hspace{0.15cm} %Espacio horizontal
{\includegraphics[scale=0.14]{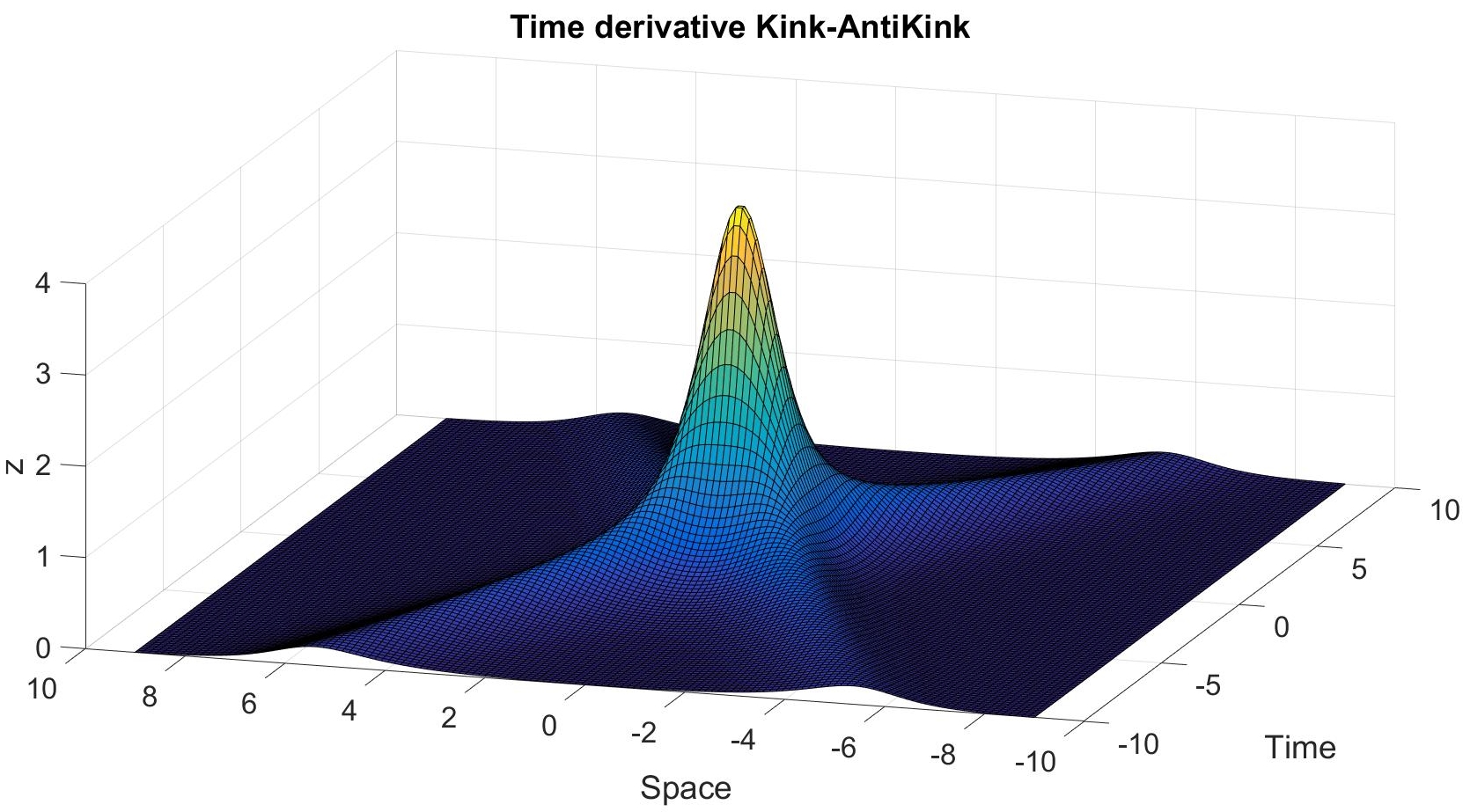}}
\caption{Above: representation of the kink-antikink solution (as the collision of kink and antikink), with speed $\beta=\frac{1}{2}$, and parameters $x_2=0$, $x_1 =\bt t$. Below: the corresponding time derivative $A_t$. Here, $(A,A_t)$ is an exact solution of SG \eqref{sg1}, just like $A(t,x)$ in \eqref{A0}.}\label{KaK_fig}
\end{figure}

\begin{rem}\label{kak_solucion}
Similarly to the previous case, the kink-antikink solution $A(t,x)$ \cite{Lamb} mentioned in the Introduction (see \eqref{A0}) can be recovered by replacing $x_1$ by $x_1 + \beta t$ in \eqref{perfil_KaK}. Figure \ref{KaK_fig} shows this exact SG solution.
\end{rem}

\subsection{2-soliton profiles and BT}
In what follows we will study how to connect breathers and complex-valued kinks, by means of a BT.

\begin{prop}\label{back_breather}  Let $(B,B_t) $ and $(K,K_t)$ be SG breather and complex-valued kink profiles respectively, both with parameters $\beta \in (-1,1)\setminus\{0\}$ and $x_1,x_2$, as in Definitions \ref{perfil_Breather} and \ref{Kink1}, and such that condition \eqref{x1_cond}  is not satisfied. Then, 
\ben
%\item Se tienen las identidades
%\be\label{SinBCosB}
%\begin{aligned}
%\sin \left(\frac B2\right)= &~ \frac{2\al\bt \sin(\al x_1) \cosh(\bt(x+x_2))}{\al^2 \cosh^2(\bt(x+x_2)) + \bt^2 \sin^2(\al x_1)},\\
% \cos \left(\frac B2\right)=&~ \frac{\al^2 \cosh^2(\bt(x+x_2)) - \bt^2 \sin^2(\al x_1)}{\al^2 \cosh^2(\bt(x+x_2)) + \bt^2 \sin^2(\al x_1)}.
%\end{aligned} %\cos\left(\dfrac{B\pm K}{2}\right)=1+\dfrac{1}{4}(\beta\pm i\alpha)\int_{-\infty}^x \big( K_x^2-B_x^2\mp B_xB_t+B_xK_t-B_tK_x\pm K_xK_t \big).
%\ee
\item We have the limits %Adem'as,
\be\label{CosBK}
\lim_{x\to \pm \infty}\cos \left(\frac{B+  K}2\right)=\lim_{x\to \pm \infty}\cos \left(\frac{B - K}2\right) = \mp 1.
\ee
\item For each $x\in \R$, $(B,B_t)$ is a BT of $(K,K_t)$ with complex-valued parameter $\bt+ i\al$. That is,
\begin{align}
    B_x-K_t \ & = \ \dfrac{1}{\beta+i\alpha}\sin\left(\dfrac{B+K}{2}\right)+(\beta+i\alpha)\sin\left(\dfrac{B-K}{2}\right), \label{bk1}
    \\ 
    B_t-K_x \ & = \ \dfrac{1}{\beta+i\alpha}\sin\left(\dfrac{B+K}{2}\right)-(\beta+i\alpha)\sin\left(\dfrac{B-K}{2}\right). \label{bk2}
\end{align}
\een
\end{prop}

%\begin{rem}
%La propiedad anterior nos entrega la validez del siguiente esquema:
%\[
%\mathbb{B}\big(K,K_t\big)  \  \xrightarrow{ \ \ a \ = \ \beta+i\alpha \ \ }{(B,B_t)},
%\]
%siempre que $x_1$ no satisfaga \eqref{x1_cond}.
%\end{rem}

\begin{proof}[Proof of Proposition \ref{back_breather}]  
%Las identidades \eqref{SinBCosB} son directas de la definici'on de $B$ m'as ciertas propiedades trigonom'etricas.
%\medskip
For proving \eqref{CosBK}, we simply use the values of $B$ and $K$ at infinity, and the fact that $\cos$ is analytic in $\Com$.

\medskip

Let us show now \eqref{bk1} and \eqref{bk2}. Let us start by proving \eqref{bk1}. Taking derivative of $B$ in \eqref{perfil_Breather} wrt to $x$ and simplifying, we have 
%Derivando $\phi$ obtenemos que, 
%\begin{align*}
%     %\phi_x &= \ 4\cdot\dfrac{\beta e^{\beta x+i\alpha t}}{1+e^{2\beta x+2i\alpha t}} = 2\beta \sin\left(\dfrac{\phi}{2}\right),
%     \phi_t &= \ 4i\cdot \dfrac{\alpha e^{\beta x+i\alpha t}}{1+e^{2\beta x+2i\alpha t}}% = 2i\alpha \sin\left(\dfrac{\phi}{2}\right).
%\end{align*}
%Por otro lado,
\begin{align}
     B_x  \ & = \ 4\partial_x \arctan\left( \dfrac{\beta}{\alpha}\dfrac{\sin(\alpha x_1)}{\cosh (\beta (x+x_2))}\right)  \nonumber \\
     &= \ \dfrac{4\alpha^2\cosh^2(\beta(x+x_2))}{\alpha^2\cosh^2(\beta(x+x_2))+\beta^2\sin^2(\alpha x_1)} \dfrac{-\beta\sin(\alpha x_1)}{\alpha \cosh^2(\beta(x+x_2))} \beta \sinh(\beta(x+x_2)) \nonumber     
     \\ &= \ \dfrac{-4\alpha\beta^2\sin(\alpha x_1) \sinh (\beta (x+x_2))}{\alpha^2\cosh^2(\beta (x+x_2))+\beta^2\sin^2(\alpha (t+x_1))}.\label{B_x}
\end{align}
On the other hand, basic trigonometric identities show that 
\begin{align}
    & \sin\left(\dfrac{B\pm K}{2}\right)  = \ 2\sin\left(\dfrac{B\pm K}{4}\right)\cos\left(\dfrac{B\pm K}{4}\right) = \ 2\tan\left(\dfrac{B\pm K}{4}\right)\cos^2\left(\dfrac{B\pm K}{4}\right) \nonumber \\ 
    & = \ 2\tan\left(\dfrac{B\pm K}{4}\right)\Bigg( 1+ \tan^2\left(\dfrac{B\pm K}{4}\right) \Bigg)^{-1} \nonumber \\ 
    & = \  \dfrac{2\tan\left(\arctan\left(\frac{\beta}{\alpha}\frac{\sin \alpha x_1}{\cosh \beta (x+x_2)}\right) \pm \arctan\left(e^{\beta (x+x_2)+i\alpha x_1}\right)\right)}{1+\tan^2\left(\arctan\left(\frac{\beta}{\alpha}\frac{\sin \alpha x_1}{\cosh \beta (x+x_2)}\right) \pm \arctan\left(e^{\beta (x+x_2)+i\alpha x_1}\right)\right)}. \label{rhs1}
\end{align}
%A=\tan \left( \arctan\left(\dfrac{\beta}{\alpha}\dfrac{\sin\alpha t}{\cosh \beta x}\right)\pm\arctan\left(e^{\beta x+i\alpha t}\right)\right)
%B=
%Notando que,
%\begin{align*}
%    \tan\left(\arctan\left(\frac{\beta}{\alpha}\frac{\sin \alpha t}{\cosh \beta x}\right) \pm \arctan\left(e^{\beta x+i\alpha t}\right)\right) \ & =  \ \dfrac{\beta \sin \alpha t\pm\alpha e^{\beta x +i\alpha t}\cosh\beta x}{\alpha\cosh\beta x\mp\beta\sin\alpha t e^{\beta x +i\alpha t}},
%\end{align*}
%obtenemos que \eqref{rhs1} se reduce a 
%Luego, mediante identidades trigonométricas básicas obtenemos que \eqref{rhs1} es reduce a
%\begin{align}
%     2\cdot \dfrac{\big(\beta \sin \alpha t\pm \alpha e^{\beta x +i\alpha t}\cosh\beta x\big) \cdot\big( \alpha\cosh\beta x\mp\beta e^{\beta x +i\alpha t}\sin\alpha t \big)}{
%    \left(\alpha\cosh\beta x\mp\beta e^{\beta x +i\alpha t}\sin\alpha t\right)^2 +\left(\beta \sin \alpha t\pm\alpha e^{\beta x +i\alpha t}\cosh\beta x\right)^2}.
%\end{align}
For the sake of notation, let $\theta := \beta (x+x_2)+i\alpha x_1$. Then, using that $\tan(a\pm b)=\frac{\tan a\pm\tan b}{1\mp \tan a \tan b}$, we obtain that  \eqref{rhs1} reads now
\[
\begin{aligned}
  & \sin\left(\dfrac{B\pm K}{2}\right) =   \frac{ 2 \left( \frac{\frac{\beta}{\alpha}\frac{\sin (\alpha x_1)}{\cosh \beta (x+x_2)} \pm e^{\theta}}{1\mp  \frac{\beta}{\alpha}\frac{\sin (\alpha x_1)e^{\theta} }{\cosh \beta (x+x_2)}} \right) }{ 1+ \left(\frac{\frac{\beta}{\alpha}\frac{\sin (\alpha x_1)}{\cosh \beta (x+x_2)} \pm e^{\theta}}{1\mp  \frac{\beta}{\alpha}\frac{\sin (\alpha x_1)e^{\theta} }{\cosh \beta (x+x_2)}} \right)^2} 
  =  \frac{ 2 \left( \frac{\beta \sin (\alpha x_1) \pm \al e^{\theta} \cosh \beta (x+x_2)  }{ \al \cosh \beta (x+x_2) \mp  \bt \sin (\alpha x_1)e^{\theta} } \right) }{ 1+ \left(   \frac{\beta \sin (\alpha x_1) \pm \al e^{\theta} \cosh \beta (x+x_2)  }{ \al \cosh \beta (x+x_2) \mp  \bt \sin (\alpha x_1)e^{\theta} } \right)^2} \\
& =   \frac{ 2  (\beta \sin (\alpha x_1) \pm \al e^{\theta} \cosh \beta (x+x_2) )(\al \cosh \beta (x+x_2) \mp  \bt \sin (\alpha x_1)e^{\theta} )  }{  (\al \cosh (\beta (x+x_2)) \mp  \bt \sin (\alpha x_1)e^{\theta} )^2 +(\beta \sin (\alpha x_1) \pm \al e^{\theta} \cosh (\beta (x+x_2)))^2} ,
\end{aligned}
\]
and simplifying,
\be\label{rhs2}
  \sin\left(\dfrac{B\pm K}{2}\right) =  \dfrac{2f_1(x)}{\big(1+e^{2\theta}\big)\big(\alpha^2\cosh^2(\beta (x+x_2))+\beta^2\sin^2(\alpha x_1)\big)},
\ee
where $ f_1(x)= f_1(x; \bt, x_1,x_2) $ is such that 
\[
\begin{aligned}
 f_1(x)  := & ~{} \alpha\beta\cosh(\beta (x+x_2))\sin(\alpha x_1)\mp\beta^2e^\theta \sin^2(\alpha x_1)  \\
& ~{} \pm \alpha^2e^\theta\cosh^2(\beta (x+x_2))-\alpha\beta e^{2\theta}\cosh(\beta (x+x_2)) \sin (\alpha x_1).
\end{aligned}
\] 
%Por otra parte, notando que
%$$ \phi_t = \ 4i\cdot\dfrac{\alpha e^{\theta}}{1+e^{2\theta}} \ , \ \ \ \  \varphi_x = \ -4\cdot\dfrac{\beta^2\alpha\sin\alpha t \cdot\sinh \beta x}{\alpha^2\cosh^2\beta x+\beta^2\sin^2\alpha t},$$ 
%Obtenemos que \begin{align}\label{eq:ecql}
%    a\cdot\sin\left(\dfrac{\phi-\widetilde{\phi}}{2}\right)+\dfrac{1}{a}\sin\left(\dfrac{\phi+\widetilde{\phi}}{2}\right) 
%     \ & = \ (\beta+i\alpha)\cdot\sin\left(\dfrac{\phi-\widetilde{\phi}}{2}\right)+(\beta-i\alpha)\sin\left(\dfrac{\phi+\widetilde{\phi}}{2}\right)  
%\end{align}
%Reemplazando lo anterior obtenemos que \eqref{eq:ecql} es equivalente a 
%reemplazando en la primera ecuación obtenemos \begin{align*}  \varphi_x-\phi_t \ = \  4\cdot\dfrac{\alpha\beta^2\cosh\beta x\sin\alpha t-\alpha\beta^2e^{2\theta}\cosh\beta x\sin \alpha t+i\alpha\beta^2e^\theta\sin^2\alpha t-i\alpha^3e^\theta\cosh^2\beta x}{\big(1+e^{2\theta}\big)\big(\alpha^2\cosh^2\beta x+\beta^2\sin^2\alpha t\big)}.
%\end{align*}
%Por último, reemplazando \eqref{rhs2}, \eqref{B_x} y \eqref{Qt} en \eqref{bk1} y recordando que $\alpha^2+\beta^2=1$, obtenemos que (EXPANDIR AQUI se ve tricky)
Now we show \eqref{bk1}. Substracting \eqref{Kt} from  \eqref{B_x}, we get 
\[
B_x-K_t=\dfrac{-4\alpha\beta^2\sin(\alpha x_1) \cdot\sinh (\beta (x+x_2))}{\alpha^2\cosh^2(\beta (x+x_2))+\beta^2\sin^2(\alpha (t+x_1))}-\frac{4i\al e^{\theta}}{1+ e^{2\theta}}= \dfrac{\widetilde{A}}{\widetilde{C}},
\]
where 
\begin{align}\label{c_br_back}
 \widetilde{C} \  = \ \big(1+e^{2\theta}\big)\,\big(\alpha^2\cosh^2(\beta (x+x_2))+\beta^2\sin^2\alpha x_1\big),
\end{align}
\begin{align*}
\widetilde{A}   = &   -4\alpha\beta^2(1+e^{2\theta})\sin\alpha x_1\sinh(\beta (x+x_2)) \\
&- 4i\alpha e^{\theta}\big(\alpha^2\cosh^2(\beta (x+x_2)) +\beta^2\sin^2\alpha x_1\big).
\end{align*}
On the other hand, recalling that $\alpha^2+\beta^2=1$, from \eqref{rhs2} we obtain 
\begin{align}\label{rhs3}
   (\beta+i\alpha)\sin\left(\dfrac{B-K}{2}\right) +\dfrac{1}{\beta+i\alpha}\sin\left(\dfrac{B+K}{2}\right) 
      \ = \ \dfrac{\widetilde{B}}{\widetilde{C}},
\end{align} 
where $\widetilde{C}$ is given by \eqref{c_br_back} and
\begin{align*}
\widetilde{B}  = & ~ 4\alpha\beta^2\big(1-e^{2\theta}\big)\sin\alpha x_1\cosh(\beta (x+x_2))+4i\alpha\beta^2e^\theta\sin^2\alpha x_1 \\
& {} -4i\alpha^3e^\theta\cosh^2(\beta (x+x_2)).
\end{align*}
Therefore, \eqref{bk1} reduces to prove $\widetilde{A}-\widetilde{B}\equiv 0$. Indeed,
\begin{align*}
    \widetilde{A}-\widetilde{B}  & = -4\alpha\beta^2\big((1+e^{2\theta}) \sin\alpha x_1\sinh(\beta (x+x_2))+2ie^\theta\sin^2\alpha x_1\big)
    \\ & \qquad -4\alpha\beta^2(1-e^{2\theta})\sin\alpha x_1\cosh(\beta (x+x_2)) =  0.
    \end{align*}
This proves \eqref{bk1}. Finally, we prove that \eqref{bk2} is satisfied. We follow the same idea as before. From \eqref{Kx_deco} and \eqref{perfil_dtBreather} we obtain
%Derivando directamente $K$ respecto a $x$ tenemos que \begin{align}\label{kx}
%K_x = \dfrac{4\beta e^{\beta (x+x_2)+i\alpha x_1}}{1+e^{2\beta (x+x_2)+2i\alpha x_1}}.
%\end{align} Con esto estamos listos para mostrar \eqref{bk2}. De 
\[
B_t-K_x=\frac{4\al^2\bt \cos (\al x_1) \cosh(\bt (x+x_2))}{\al^2 \cosh^2(\bt (x+x_2)) + \bt^2 \sin^2(\al x_1)}-\dfrac{4\beta e^{\theta}}{1+e^{2\theta}} = \dfrac{\widetilde{A}_2}{\widetilde{C}},
\] 
where $\widetilde{C}$ is given by \eqref{c_br_back} and
\begin{align*}%\label{bkrhs}
    \widetilde{A}_2 = & ~ 4\alpha^2\beta\cos(\alpha x_1)\cosh (\beta (x+x_2))\big(1+e^{2\theta}\big) \nonu \\
    &-4\beta e^\theta\big(\alpha^2\cosh^2(\beta (x+x_2))+\beta^2\sin^2(\alpha x_1)\big).
 \end{align*}
On the other hand, recalling that $\alpha^2+\beta^2=1$ and making similar simplifications as for \eqref{rhs3}, we have 
\[
   \dfrac{1}{\beta+i\alpha}\sin\left(\dfrac{B+K}{2}\right) - (\beta+i\alpha)\sin\left(\dfrac{B-K}{2}\right) 
      \ = \ \dfrac{\widetilde{B}_2}{\widetilde{C}},
\]
where $\widetilde{C}$ is given by \eqref{c_br_back} and
\begin{align*}
   \widetilde{B}_2\ =\ &  4\big(\alpha^2\beta e^\theta\cosh^2(\beta (x+x_2))-\beta^3e^\theta\sin^2(\alpha x_1) \nonumber
      \\ & \quad +i\alpha^2\beta e^{2\theta}\cosh(\beta (x+x_2))\sin(\alpha x_1)  -i\alpha^2\beta \cosh^2(\beta (x+x_2))\sin(\alpha x_1)\big). %\label{bkrhs2}
 \end{align*}
Hence, \eqref{bk2} is reduced to show that $\widetilde{A}_2-\widetilde{B}_2\equiv 0$. Indeed, simplifying, 
\begin{align*}
    & \widetilde{A}_2-\widetilde{B}_2 \\
    & = \  4\alpha^2\beta\cosh(\beta (x+x_2))\left(\cos\alpha x_1+i\sin\alpha x_1+e^{2\theta}(\cos\alpha x_1-i\sin\alpha x_1)\right)
    \\ & \quad -8\alpha^2\beta e^\theta\cosh^2(\beta( x+x_2))
     \\ & =\   8\alpha^2\beta e ^\theta\cosh^2 (\beta (x+x_2))-8\alpha^2\beta e^\theta\cosh^2(\beta (x+x_2)) =  0.
\end{align*}\end{proof}
 
The following corollary shows that there is also a relationship between the breather and the conjugate of the complex-valued kink profile. 

\begin{cor}\label{Coro4p5}  Let $(B,B_t) $ and $(\overline{K},\overline{K}_t)$ be SG breather and complex-valued kink profiles respectively, both with scaling parameters $\beta \in (-1,1)\setminus\{0\}$ and shifts $x_1,x_2$ such that \eqref{x1_cond} do not satisfy. Then, for each $x\in \R$, $(B,B_t)$ is a BT of $(\overline{K},\overline{K}_t)$ with parameter $\bt- i\al$:
%Sean $\overline{K}(t,x)$ y $B(t,x)$ el conjugado del kink y el breather de sine-Gordon respectivamente, ambos con parámetros de scaling $\alpha,\,\beta>0$, es decir, tales que $\alpha^2+\beta^2=1$. Entonces se cumple que 
\begin{align}
    B_x-\overline{K}_t \ & = \ \dfrac{1}{\beta-i\alpha}\sin\left(\dfrac{B+\overline{K}}{2}\right)+(\beta-i\alpha)\sin\left(\dfrac{B-\overline{K}}{2}\right), \label{bk1_c}
    \\ 
    B_t-\overline{K}_x \ & = \ \dfrac{1}{\beta-i\alpha}\sin\left(\dfrac{B+\overline{K}}{2}\right)-(\beta-i\alpha)\sin\left(\dfrac{B-\overline{K}}{2}\right). \label{bk2_c}
\end{align}
\end{cor}
%\begin{rem}
%El lema anterior nos entrega la validez del siguiente esquema:
%\[
%\mathbb{B}\big(\overline{K},\overline{K}_t\big)  \  \xrightarrow{ \ \ a \ = \ \beta-i\alpha \ \ }{(B,B_t)}.
%\]
%\end{rem}
\begin{proof}
Direct from previous result.
\end{proof}

When working with multiple profiles it is convenient to introduce a schematic representation of the BT, see \cite{Lamb}. Figure \ref{Flechas} shows a diagram where each arrow represents the BT of the SG solution $(\phi_i,\phi_{i,t})$ towards another solution $(\phi_j,\phi_{j,t})$ with parameter $a_k$, and given in Definition \ref{defi1}. The fact that both BT arrive to the same solution is not a coincidence and it is called in the literature as \emph{Permutability Theorem}. In this article we will present a rigorous proof of this result for solutions of SG which are perturbations of the profiles showed in the previous section.

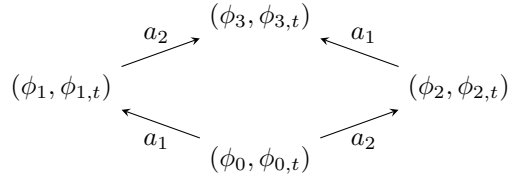
\begin{figure}[h!]
\begin{center}
\begin{tikzpicture}
  \matrix (m) [matrix of math nodes,row sep=1em,column sep=3em,minimum width=1em]
  {
    & (\phi_3, \phi_{3,t}) & \\
     (\phi_1, \phi_{1,t}) &  & (\phi_2, \phi_{2,t})\\
    &  (\phi_0, \phi_{0,t}) & \\};
  \path[-stealth]
  
    (m-2-3) 
            edge node [above] {$\ a_1$} (m-1-2)

    (m-3-2) 
            edge node [below] {$ a_1 \ $} (m-2-1)

    (m-3-2) 
            edge node [below] {$\ a_2$} (m-2-3)
  
    (m-2-1) 
            edge node [above] {$ a_2 \  $} (m-1-2);
\end{tikzpicture}
\end{center}
\caption{A diagram representing two consecutive applications of the BT with inverse parameters $a_1$ y $a_2$. The permutability property says that $ (\phi_3, \phi_{3,t})$ is the unique final function, independently of the two considered paths.}\label{Flechas}
\end{figure}

We remark that Proposition \ref{back_breather}, together with Corollary \ref{Coro4p5} show the validity of the diagram in Fig. \ref{Fig_flechas} for SG profiles, and not only solutions of the equation itself. This diagram is valid as soon as $x_1$ do not satisfy \eqref{x1_cond}, in order to avoid the lack of good definition for $K$ and $\overline K$.

\begin{figure}[h!]
\begin{tikzpicture}
  \matrix (m) [matrix of math nodes,row sep=2em,column sep=3em,minimum width=3em]
  {
    & (B,B_t) & \\
     (K,K_t) &  & (\overline{K},\overline{K}_t) \\
    &  (0,0) & \\};
  \path[-stealth]
  
    (m-2-3) 
            edge node [above] {$\ \ \ \ \ \  \beta-i\alpha$} (m-1-2)

    (m-3-2) 
            edge node [below] {$ \beta-i\alpha \ \ \ \ \ \ $} (m-2-1)

    (m-3-2) 
            edge node [below] {$\ \ \ \ \ \  \beta+i\alpha$} (m-2-3)
  
    (m-2-1) 
            edge node [above] {$ \beta+i\alpha \ \ \ \ \ \  $} (m-1-2);
\end{tikzpicture}
\caption{Diagram for the breather $B$ in Proposition \ref{back_breather}. Note that $(B,B_t)$ is obtained independently of the chosen path \cite{Lamb}.}\label{Fig_flechas}
\end{figure}
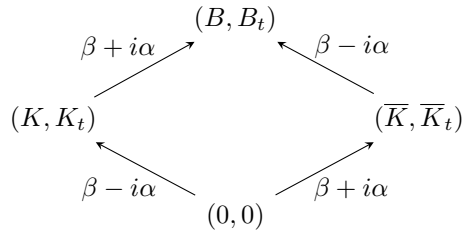

\medskip

Now we want to study the conection between the SG kink and kink-antikink. 

\begin{prop}[Kink-Antikink connection]\label{back_kak}  Let $(A,A_t) $ be a SG  kink-antikink profile, with speed parameter $\beta \in (-1,1)\setminus\{0\}$ and shifts $x_1,x_2$, as was introduced in Definition \ref{Perfil_KaK}. Let also 
\be\label{Q_util}
\vec{Q}:= (Q,Q_t):=(Q,Q_t)(x;-\bt,x_1+x_2),
\ee
be a real-valued kink profile (see Definition \ref{Kink0} and Observation \ref{Cambios de signo}), with speed parameter $-\beta\in(-1,1)\backslash\{0\}$ and shift $(x_1+x_2)$.\footnote{Note the specific character of the choice in the shift parameter.} Then, the following is satisfied:
\smallskip
\ben
\item We have the identities
\be\label{CosAQ}
\lim_{x\to \pm\infty}\cos\left(\dfrac{A\pm Q}{2}\right) = \begin{cases} -1, \quad x\to +\infty \\ 1, \quad x\to -\infty \end{cases}.
\ee
\item For each $x\in \R$, $(A,A_t)$ is a BT of $(Q,Q_t)$ with real-valued parameter $a=a(\beta)$ (see \eqref{a(beta)}). That is,
\begin{align}
    A_x-Q_t \ & = \ \dfrac{1}{a}\sin\left(\dfrac{A+Q}{2}\right)+a\,\sin\left(\dfrac{A-Q}{2}\right), \label{kakbk1}
\\ A_t-Q_x \ & = \ \dfrac{1}{a}\sin\left(\dfrac{A+Q}{2}\right)-a\,\sin\left(\dfrac{A-Q}{2}\right).\label{kakbk2}
\end{align}
\een
\end{prop} 

\begin{rem} Generally speaking, we have the validity of the diagram in Fig. \ref{Fig_flechas_2} (above), as soon as we choose kink profiles of parameters $(Q,Q_t)(x,\beta,-x_1+x_2)$ and $(Q,Q_t)(x;-\beta,x_1+x_2)$. In this sense, the reconstruction of $(A,A_t)$ requires a different rigidity than that of the breather. In this paper, we will only use the RHS connection via $(Q,Q_t)(x;-\beta,x_1+x_2)$.
\end{rem}

%\begin{center}
\begin{figure}[h!]
\begin{tikzpicture}
  \matrix (m) [matrix of math nodes,row sep=2em,column sep=1em,minimum width=1em]
  {
    & (A,A_t)(x; \bt,x_1,x_2) & \\
     (Q,Q_t)(x; \bt,-x_1+x_2) &  & (Q,Q_t)(x; -\bt,x_1+x_2)\quad \star \\
    &  (0,0) & \\};
  \path[-stealth]
  
    (m-2-3) 
            edge node [above] {$\ \ \ \ \ \  a$} (m-1-2)

    (m-3-2) 
            edge node [below] {$ a \ \ \ \ \ \ $} (m-2-1)

    (m-3-2) 
            edge node [below] {$\ \ \ \ \ \  1/a$} (m-2-3)
  
    (m-2-1) 
            edge node [above] {$ 1/a \ \ \ \ \ \  $} (m-1-2);
\end{tikzpicture}
\begin{tikzpicture}
  \matrix (m) [matrix of math nodes,row sep=3em,column sep=1em,minimum width=1em]
  {
    & (R,R_t) (x; \bt,x_1,x_2)  & \\
   \star  \quad  (Q,Q_t)(x, -\bt, x_1+x_2)&  & (Q,Q_t)(-x; -\bt, x_1- x_2) \\
    &  (0,0) & \\};
  \path[-stealth]
  
    (m-2-3) 
            edge node [above] {$\ \ \ \ \ \  1/a$} (m-1-2)

    (m-3-2) 
            edge node [below] {$ 1/a \ \ \ \ \ \ $} (m-2-1)

    (m-3-2) 
            edge node [below] {$\ \ \ \ \ \  -a$} (m-2-3)
  
    (m-2-1) 
            edge node [above] {$ -a \ \ \ \ \ \  $} (m-1-2);
\end{tikzpicture}
\caption{Schematic diagram for the kink-antikink pair $(A,A_t)$ (above), and the 2-kink $(R,R_t)$ (below).  In this paper, we will follow the paths refereed with $\star$.}\label{Fig_flechas_2}
\end{figure}
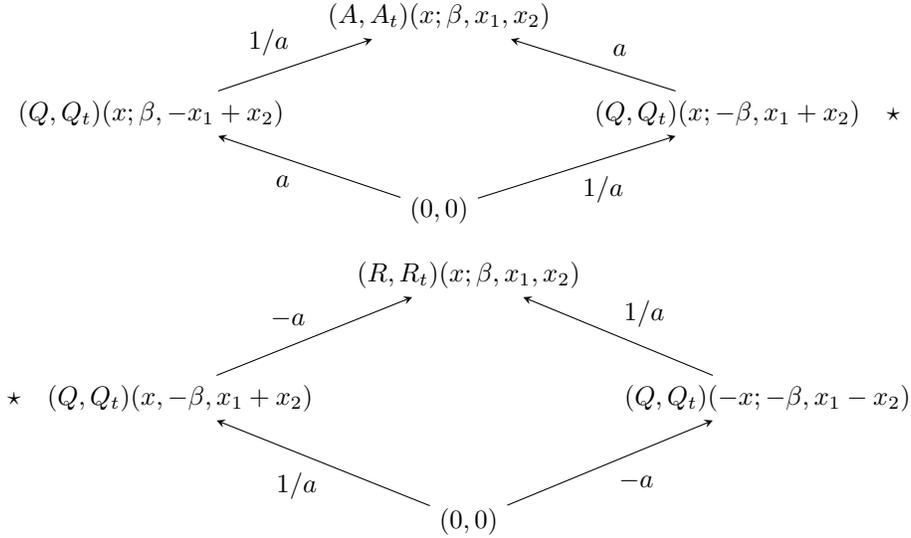
%\end{center}
 
\begin{proof}[Proof of Proposition \ref{back_kak}]
The proof of this result is very similar to that of Proposition \ref{back_breather}. See Appendix \ref{ABC}.
\end{proof}

In order to conclude this section we will study the relationship between real-valued kinks and 2-kinks of SG. 
%La siguiente Propostition nos muestra la profunda relación que existe entre ellos.

\begin{cor}[$2$-kink connection]\label{back_2k} Let $(R,R_t) $ be a SG 2-kink profile, with speed parameter $\beta \in (-1,1)\setminus\{0\}$ and shifts $x_1,x_2$. Let $\vec Q$ denote the kink defined in \eqref{Q_util}, with speed parameter $-\beta\in(-1,1)$ and shift $(x_1+x_2)$. Then, 

\ben
\item We have the limits
\be\label{CosRQ}
\lim_{x\to \pm\infty}\cos\left(\dfrac{R\pm Q}{2}\right) = \begin{cases} 1, \quad x\to +\infty \\ -1, \quad x\to -\infty \end{cases}.
\ee
\item For each $x\in \R$, $(R,R_t)$ is a BT of $(Q,Q_t)$ with parameter $a_3=-a(\beta)$ (see \eqref{a2_a3}):
\begin{align}
    R_{x}-Q_t \ & = \ \dfrac{1}{a_3}\sin\left(\dfrac{R+Q}{2}\right)+a_3\,\sin\left(\dfrac{R-Q}{2}\right), \label{2kbk1}
    \\ 
    R_{t}-Q_x \ & = \ \dfrac{1}{a_3}\sin\left(\dfrac{R+Q}{2}\right)-a_3\,\sin\left(\dfrac{R-Q}{2}\right).\label{2kbk2}
\end{align}
\een
\end{cor} 
\begin{rem} We have in general the validity of the diagram in Fig. \ref{Fig_flechas_2} (below), but we will only use its left side component.
\end{rem}

\begin{proof} Direct from Proposition \ref{back_kak}, it is enough to change the roles of $x+x_2$ and $x_1$, and $a(\beta)$ by $-a(\beta)$.
\end{proof}

\section{Modulation of 2-solitons}\label{5}

In order to prove Theorem \ref{MT1}, we will show first some modulation lemmas. Here we will follow the ideas in \cite{MMT} and \cite{AMP1}.

\subsection{Static modulation}\label{mod_estatica}
We will consider three pair of objects to deal with:
\ben
\item $(B,B_t)$ a SG breather profile with scaling parameter $\beta\in(-1,1)$, $\beta\neq 0$ fixed, and shifts $x_1,x_2\in\mathbb{R}$, as in Definition \ref{Perfil_B}. 

\smallskip

\item $(R,R_t)$ a SG 2-kink profile with speed $\beta\in(-1,1)$, $\beta\neq 0$ fixed, and shifts $x_1,x_2\in\mathbb{R}$, as in Definition \ref{Perfil_2K}.

\smallskip

\item $(A,A_t)$ a SG kink-antikink profile with speed $\beta\in(-1,1)$, $\beta\neq 0$ fixed, and shifts $x_1,x_2\in\mathbb{R}$, as in Definition \ref{Perfil_KaK}.
\een

Let $D$ denote any of the capital letters $A$, $B$ or $R$. We will use subindexes $1$ and $2$ to denote derivatives of $A$, $B$ and $R$ wrt the shifts $x_1$ and $x_2$ respectively, namely for $j=1,2$
\begin{align}
D_j( x;\bt,x_1,x_2)&:=\partial_{x_j}D( x;\bt,x_1,x_2), \label{D1}
\\ 
(D_t)_j( x;\bt,x_1,x_2)&:=\partial_{x_j}D_t( x;\bt,x_1,x_2). \label{D2}
\end{align}

\begin{rem}
In Appendix \ref{ap_derivadas} we can find an explicit description of the derivatives above mentioned in the cases $D=A$ and $D=R$, showing clearly that these are localized functions (see Subsection \ref{Orto_KK}).
\end{rem}
%Asimismo,
%\begin{align}
%R_1( x;x_1,x_2)&:=\partial_{x_1}R( x;x_1,x_2), \ \ \ \ \ \ \  R_2( x;x_1,x_2):=\partial_{x_2}R( x;x_1,x_2), \label{R1}
%\\ (R_t)_1( x;x_1,x_2)&:=\partial_{x_1}R_t( x;x_1,x_2), \ \ (R_t)_2( x;x_1,x_2):=\partial_{x_2}R_t( x;x_1,x_2). \label{R2}
%\end{align}
%Y finalmente,
%\begin{align}
%A_1( x;x_1,x_2)&:=\partial_{x_1}A( x;x_1,x_2), \ \ \ \ \ \ \  A_2( x;x_1,x_2):=\partial_{x_2}A( x;x_1,x_2), \label{A1}
%\\ (A_t)_1( x;x_1,x_2)&:=\partial_{x_1}A_t( x;x_1,x_2), \ \ (A_t)_2( x;x_1,x_2):=\partial_{x_2}A_t( x;x_1,x_2). \label{A2}
%\end{align}

Let $\nu>0$ be a small real number. Let us also consider the following tubular neighborhood of a 2-soliton $(D,D_t)$ of radius $\nu$:
\begin{align*}
\mathcal{U}(\nu):=\Big\{(\phi,\phi_t) ~ : ~ \ \inf_{x_1,x_2\in\mathbb{R}} \big\Vert (\phi,\phi_t) -(D,D_t)(\cdot; \bt,x_1,x_2) \big\Vert_{H^1\times L^2}< \nu \Big\}.
\end{align*}
It is important to mention that this set has no temporal dependence. Since $(\phi,\phi_t)$ does not necessarily decay to zero (e.g. 2-kink case), the key is the difference with $(D,D_t)$. However in the case of kink-antikink or breather, $(\phi,\phi_t)\in H^1\times L^2$. For the proof of next result, see Appendix \ref{Modula}.

\begin{lem}[Static Modulation]\label{Mod_esta}
There exists $\nu_0>0$ such that for each $0<\nu<\nu_0$, the following is satisfied. For each pair $(\phi,\phi_t)\in\mathcal{U}(\nu)$, there exists a unique couple of $C^1$ functions $\tilde{x}_1,\tilde{x}_2:\mathcal{U}(\nu)\to\mathbb{R}$ such that , if we consider $z=z(x)$ and $w=w(x)$ defined as 
\[ 
z(x):=\phi(x)-D(x;\bt,\tilde{x}_1,\tilde{x}_2), \quad w(x):= \phi_t(x)-D_t(x;\bt,\tilde{x}_1,\tilde{x}_2),
\] 
then, the following orthogonality conditions hold: 
\begin{align*} 
\int_{\mathbb{R}} (z,w)\cdot \big(D_1,(D_1)_t\big)dx=\int_{\mathbb{R}} (z,w)\cdot\big(D_2,(D_2)_t\big)dx=0.
\end{align*} 
%Más aún, existe $K>0$ tal que, si $u\in\mathcal{U}(\varepsilon)$, con $0<\varepsilon<\eta$, entonces $$\Vert z\Vert_{H^1}+\vert \alpha-\alpha_0\vert+\vert \vert  \leq K\varepsilon. $$
\end{lem}

\subsection{Dynamical modulation} 
We need now a dynamical version of the previous lemma. Let $(\phi,\phi_t)$ be a solution of \eqref{sg1}, with initial data $(\phi_0,\phi_1)$ such that
\begin{align}
\Vert (\phi_0,\phi_1)-(D,D_t)(\cdot;\bt,0,0)\Vert_{H^1\times L^2}< \eta, \label{initial_data}
\end{align} 
for some $0<\eta<\eta_0$ small enough, with $\eta_0$ given by Theorem \ref{MT1}.

\begin{defn}[Recurrence Time]\label{Tiempo}
Let $C^*>1$ be a large parameter (to be chosen later), and let $(\phi,\phi_t)(t)$ be the unique globally defined solution of SG with initial data $(\phi_0,\phi_1)$, and satisfying \eqref{initial_data}.  We define $T^*:=T^*(C^*)>0$ as the maximal time for which there are parameters $\tilde x_1(t)$ and $\tilde x_2(t)$ such that   
\be\label{Tubular}
\sup_{t\in[0,T^*]}\Vert (\phi,\phi_t)(t)-(D,D_t)(\cdot;\bt, \tilde x_1(t), \tilde x_2(t))\Vert_{H^1\times L^2} \leq C^* \eta. 
\ee
\end{defn}

Note that $T^*$ is well-defined thanks to continuity of the SG flow, \eqref{initial_data} and the fact that $C^*>1$. Later we will prove that $T^*$ can be taken infinity for all $C^*$ large enough. Even more,
\be\label{Mar}
\hbox{In what follows we will assume that $T^*$ is {\bf finite}. }
\ee
By choosing $\eta_0$ sufficiently small if necessary, we will have  $C^* \eta<\nu_0$ in Lemma \ref{Mod_esta}, and the following result will be valid:

\begin{cor}[Dynamical modulation]\label{Mod_Dinamica}
Under the assumptions of Definition \ref{Tiempo}, there are $C^1$ functions $x_1,x_2:[0,T^*]\to\mathbb{R}$ such that, if
\be\label{z_y_w}
\begin{aligned}
z(t,x):=&~\phi(t,x)-D(x;\bt,x_1(t),x_2(t)), \\
 w(t,x):=&~ \phi_t(t,x)-D_t(x;\bt,x_1(t),x_2(t)),
\end{aligned} 
\ee
then, for each $t\in [0,T^*],$
\be\label{Modulacion_temporal} 
\int_{\mathbb{R}} (z,w)\cdot \big(D_1,(D_1)_t\big)(t,x)dx=\int_{\mathbb{R}} (z,w)\cdot\big(D_2,(D_2)_t\big)(t,x)dx=0,
\ee
and moreover
\be\label{New_Tubular}
\sup_{t\in[0,T^*]}\Vert (z,w)(t) \Vert_{H^1\times L^2}\lesssim C^* \eta,
\ee
\be\label{Derivadas_0}
\Vert (z,w)(0) \Vert_{H^1\times L^2}+ |x_1(0)| + |x_2(0)| \lesssim \eta,
\ee
and
\be\label{Derivadas}
\sup_{t\in[0,T^*]} (|x_1'(t)| + |x_2'(t)|) \lesssim \sup_{t\in[0,T^*]}\Vert (z,w)(t) \Vert_{H^1\times L^2} \lesssim C^*\eta. 
\ee
Moreover, if $D=R$ and $(z_0,w_0)$ are odd, or if $D=B,A$ and $(z_0,w_0)$ are even, then we can choose $x_2(t)\equiv 0$, and the parity property on $(z,w)$ is preserved in time.
\end{cor}

\begin{proof} 
Direct from Lemma \ref{Mod_esta} and \eqref{Tubular}.
\end{proof}

\section{Perturbations of breathers}\label{Perturbaciones_B}\label{6}

\subsection{Statement} In this section we will assume $\K=\Com$ in Definition \ref{fperturbacion}. Our goal will be to show the following result.

\begin{prop}[Descent to the zero solution]\label{Descenso_global}
Let $(B,B_t)$ be a SG breather profile, as in Definition \ref{Perfil_B}, with scaling parameter $\beta\in(-1,1)\setminus\{0\}$ and shifts $x_1,x_2\in\mathbb{R}$, such that $x_1$ do not satisfy \eqref{x1_cond}. Let also $(K,K_t)$ be the complex-valued kink profile associated to $(B,B_t)$, that is with same parameters as $(B,B_t)$. Then, there are constants $\eta_0>0$ and $C>0$ such that, for all $0<\eta<\eta_0$ and all $(z_0,w_0)\in H^1\left(\mathbb{R}\right)\times L^2\left(\mathbb{R}\right)$ such that\footnote{Note that both $(z_0,w_0)$ are real-valued.} 
\begin{align*}
\Vert (z_0,w_0)\Vert_{H^1(\mathbb{R})\times L^2(\mathbb{R})}<\eta, 
\end{align*}
then the following is satisfied:
\ben
\item There are unique $(u_0,s_0,\delta)$ defined in an open subset of  $H^1\left(\mathbb{R};\mathbb{C}\right)\times L^2\left(\mathbb{R};\mathbb{C}\right)\times\mathbb{C}$ such that the B\"acklund functional \eqref{fperturbacion} satisfies
\begin{align*}
\mathcal F(B+z_0,B_t+w_0,K+u_0,K_t+s_0,\beta+i\alpha+\delta) & = (0,0),
\end{align*}
and where
\begin{align*}
\Vert(u_0,s_0)\Vert_{H^1\times L^2}+\vert \delta\vert< C\eta.
\end{align*}
%{\color{red}Mejor a'un, si $(z_0,w_0)$ son impares, entonces $(u_0,s_0)$ poseen parte real par y parte imaginaria impar.}
\item 
Making $\eta_0$ even smaller if necessary, there are unique $(y_0,v_0,\tilde \delta)$, defined in an open subset of  $H^1\left(\mathbb{R};\mathbb{C}\right)\times L^2\left(\mathbb{R};\mathbb{C}\right)\times\mathbb{C}$, and such that 
\begin{align*}
\mathcal F(K+u_0,K_t+s_0,y_0,v_0,\beta-i\alpha+\tilde \delta) &=(0,0), 
\end{align*} 
and
\begin{align*}
 \Vert (y_0,v_0)\Vert_{H^1\times L^2}+\vert \tilde \delta\vert < C\eta.
\end{align*}
%{\color{red}M'as a'un, si $(u_0,s_0)$ son de parte real par y parte imaginaria impar, entonces $(y_0,v_0)$ son impares.}
\een
\end{prop}

The rest of the section will be devoted to the proof of this result, for which we will need some auxiliary lemmas.

\subsection{Integrant Factor} Let us start with an auxiliary result on existence of integrant factors for some ODEs appearing naturally when studying breathers and BT.

\begin{lem}[Existence of Integrant Factor]\label{mu_b_bajada} Let $(B,B_t)$ and $(K,K_t)$ be breather and complex-valued kink profiles, both with scaling parameter $\beta\in(-1,1)$, $\beta\neq 0$, and shifts $x_1,x_2\in\mathbb{R}$. Let us consider
\be\label{mu_kink}
 \mu_K(x)  := \dfrac{1}{\cosh(\beta (x+x_2)+i\alpha x_1)} = \frac{K_x(x)}{2\bt}, \qquad (\hbox{see \eqref{Kxxx}}),
\ee
and
\be\label{mu_breather}
    \mu_B(x)  := \dfrac{\cosh(\beta (x+x_2)+i\alpha x_1)}{\alpha^2\cosh^2(\beta (x+x_2))+\beta^2\sin^2(\alpha x_1)}=\dfrac{1}{4\alpha^2\beta^2} (\bt B_t - i \al B_x)(x).
\ee
Then the following holds:
%definimos entonces $\mu_K$ como 
%\begin{align*}
%    \mu(x) \ := \ \exp\left(\,\int \Psi  \,\right)=\dfrac{1}{\cosh(\beta (x+x_2)+i\alpha x_1)},
%\end{align*}
%donde $\int \Psi$ es una primitiva de $\Psi$. Notemos que $\mu$ decae exponencialmente en espacio cuando $x\to\pm\infty$.

\ben
\item (Local and global behavior) 
\smallskip
\ben
\item $\mu_K(x)$ is well-defined and smooth for any $\bt \in (-1,1)\setminus\{0\}$, and  $x_1,x_2\in \R$, provided $x_1$ do not satisfy \eqref{x1_cond}. Additionally, it decays exponentially fast in space as $x\to\pm\infty$. 

\smallskip

\item $\mu_B(x)$ is well-defined and smooth for any $\bt \in (-1,1)\setminus\{0\}$, and $x_1,x_2\in \R$. Additionally, it decays exponentially fast in space as $x\to\pm\infty$. Finally, $\mu_B$ is not zero if \eqref{x1_cond} is not satisfied.
\een

\medskip

\item (ODEs) We have that $\mu_K(x)$ satisfies the ODE
\be\label{edomuK}
\mu_x- \beta\cos\left(\dfrac{K}{2}\right) \mu=0,
\ee
and $\mu_B(x)$ solves the ODE
\begin{align}\label{edomu}
\qquad  \mu_x-\left(\dfrac{(\beta-i\alpha)}{2}\cos\left(\dfrac{B+K}{2}\right)+ \dfrac{(\beta+i\alpha)}{2}\cos\left(\dfrac{B-K}{2}\right)\right) \mu = 0.
\end{align}
\item (Non orthogonality) For each $x_1$ such that  \eqref{x1_cond} is not satisfied, we have
\be\label{nonzero_K}
\int_\mathbb{R} \mu_K\sin\left(\dfrac{K}{2}\right)=\dfrac{2}{\beta},
\ee
and $\mu_B$ is not orthogonal to $(B_x-K_t)$, that is:
\begin{align}\label{integral}
\int_\mathbb{R} \mu_B\left(B_x-K_t\right) =-\dfrac{4i}{\alpha\beta}.
\end{align}
Finally, these identities can be extended by continuity to all $x_1 \in \R$. 
\een
\end{lem}

\begin{proof}
The proof of this result is direct but cumbersome, see Appendix \ref{demostraciones_fi} for the proof.
\end{proof}
%%%%%%%%%%%%%%%%%%%%%%%%%%%%%%%%%%%%%%%%%%%%%%%%%%%%%%%%%%%%%%%%%%
%UN PAR DE CALCULOS UTILES 
%\begin{align*}
%B_{xx}&=K_{tx} - (\beta-i\alpha)\dfrac{B_x+K_x}{2}\cos\left(\dfrac{B+K}{2}\right)-(\beta+i\alpha)\dfrac{B_x-K_x}{2}\cos\left(\dfrac{B-K}{2}\right) \rightarrow -\dfrac{i}{\beta}
%\\ B_{tx} &=K_{xx} - (\beta-i\alpha)\dfrac{B_x+K_x}{2}\cos\left(\dfrac{B+K}{2}\right)+(\beta+i\alpha)\dfrac{B_x-K_x}{2}\cos\left(\dfrac{B-K}{2}\right)\rightarrow \dfrac{1}{\alpha}
%\end{align*}

%\begin{align*}
%4\alpha\beta \mu_x =(B_x+K_x)\cos\left(\dfrac{B+K}{2}\right)
%\end{align*}

%$$ - \dfrac{(\beta-i\alpha)}{2}\dfrac{B_x+K_x}{2}\sin\left(\dfrac{B+K}{2}\right)- \dfrac{(\beta+i\alpha)}{2}\dfrac{B_x-K_x}{2}\sin\left(\dfrac{B-K}{2}\right) =\dfrac{1}{4}\left(K_x^2-B_x^2+B_xK_t-B_tK_x\right)$$
%%%%%%%%%%%%%%%%%%%%%%%%%%%%%%%%%%%%%%%%%%%%%%%%%%%%%%%%%%%%%%%%%%

\subsection{Proof of Proposition \ref{Descenso_global}} Using Lemma \ref{mu_b_bajada}, the first item in Proposition \ref{Descenso_global} will be a consequence of the following result.

\begin{lem}\label{bbajada}
Let $(B,B_t)$ and $(K,K_t)$ be breather and complex-valued kink profiles, both with scaling parameter $\beta\in(-1,1)\setminus\{0\}$ and shifts $x_1,x_2\in\mathbb{R}$, and such that \eqref{x1_cond} is not satisfied. Then, there are constants $\eta_0>0$ and $C>0$ such that for all $0<\eta<\eta_0$ and for all $(z_0,w_0)\in H^1\left(\mathbb{R}\right)\times L^2\left(\mathbb{R}\right)$ such that  
\begin{align*}
\Vert (z_0,w_0)\Vert_{H^1(\mathbb{R})\times L^2(\mathbb{R})}<\eta, 
\end{align*}
there are unique $(u_0,s_0,\delta)$ defined in an open subset of $H^1\left(\mathbb{R};\mathbb{C}\right)\times L^2\left(\mathbb{R};\mathbb{C}\right)\times\mathbb{C}$ and such that $\mathcal F$ in \eqref{fperturbacion} satisfies
\be\label{F_new}
\mathcal F(B+z_0,B_t+w_0,K+u_0,K_t+s_0,\beta+i\alpha+\delta) = (0,0),
\ee
and 
\be\label{F_new_2}
\Vert(u_0,s_0)\Vert_{H^1\times L^2}+\vert \delta\vert\leq C\eta.
\ee
\end{lem}

\begin{proof} Let $(z_0,w_0)\in H^1\left(\mathbb{R}\right)\times L^2\left(\mathbb{R}\right)$ be given, with a size to be defined below. Consider the system of equations give by the B\"acklund functionals \eqref{f1}-\eqref{f2} in the variables $(u_0,s_0,\delta)\in H^1(\mathbb{R};\mathbb{C})\times L^2(\mathbb{R};\mathbb{C})\times\mathbb{C}$ (note that this space and $H^1\left(\mathbb{R}\right)\times L^2\left(\mathbb{R}\right)$ define the space $X(\K)$ for $\mathcal F$): 
\begin{align}
   & \mathcal F_1 \big(B+z_0, B_t+w_0, K + u_0, K_t+s_0, \beta+i\alpha+\delta\big)= \nonu \\
   & \quad =  B_x+z_{0,x}-K_t-s_0 - \dfrac{1}{\beta+i\alpha+\delta}\sin\left(\dfrac{B+z_0+K+u_0}{2}\right)
   \nonumber \\ & \qquad   -(\beta+i\alpha+\delta)\sin\left(\dfrac{B+z_0-K-u_0}{2}\right), \label{eqbbb10}
    \\ 
    & \mathcal F_2 \big(B+z_0, B_t+w_0, K+u_0, K_t+s_0, \beta+i\alpha+\delta\big) =\nonu \\
    & \quad =  B_t+w_0-K_x-u_{0,x} -\dfrac{1}{\beta+i\alpha+\delta}\sin\left(\dfrac{B+z_0+K+u_0}{2}\right) \nonumber \\ & \qquad  +(\beta+i\alpha+\delta)\sin\left(\dfrac{B+z_0-K-u_0}{2}\right). \label{eqbbb20} 
\end{align}  
We look for a unique choice of $(u_0,s_0,a)$ such that  
\[
\mathcal F(B+z_0,B_t+w_0,K+u_0,K_t+s_0,\beta+i\alpha+\delta)= (0,0).
\]
We will use the Implicit Function Theorem for $(\mathcal F_1, \mathcal F_2)$. Note that from \eqref{eqbbb10} that once $(u_0,\,\delta)$ are defined, $\,s_0$ gets completely determined from \eqref{eqbbb10}. Hence, we will only solve \eqref{eqbbb20} for $(u_0,\delta)$. Gracias a que $\mathcal F(B,B_t,K,K_t,\bt+i\al)=(0,0)$, un cierto re-arreglo de \eqref{eqbbb10} y \eqref{eqbbb20} nos dice que estas ecuaciones pueden escribirse como
\begin{align}
    & \mathcal{\widetilde F}_1 \big(z_0,w_0, u_0, s_0,\delta\big) \nonu  \\
    &:  =  \ z_{0,x}- s_0 - \dfrac{1}{\beta+i\alpha+\delta}\sin\left(\dfrac{B+K+ z_0+u_0}{2}\right)  + \dfrac{1}{\beta+i\alpha}\sin\left(\dfrac{B+K}{2}\right)
   \nonumber \\ 
   &  \quad  -(\beta+i\alpha+\delta)\sin\left(\dfrac{B-K +z_0 -u_0}{2}\right)   + (\beta+i\alpha )\sin\left(\dfrac{B-K }{2}\right) =0, \label{eqbbb1}\\
    & \mathcal{\widetilde F}_2 \big(z_0, w_0, u_0, s_0,\delta\big) \nonu \\
    &  : = \ w_0 -u_{0,x} -\dfrac{1}{\beta+i\alpha+\delta}\sin\left(\dfrac{B+K+z_0 +u_0}{2}\right)  + \dfrac{1}{\beta+i\alpha}\sin\left(\dfrac{B+K}{2}\right)  \nonumber \\
     &  \quad +(\beta+i\alpha+\delta)\sin\left(\dfrac{B-K+z_0-u_0}{2}\right) -(\beta+i\alpha)\sin\left(\dfrac{B-K}{2}\right)=0 . \label{eqbbb2} 
\end{align}  

Clearly $\mathcal{\widetilde F}_2$ defines a $\mathcal{C}^1$ functional in the vicinity of zero, and $\mathcal{\widetilde F}_2 \big(0, 0, 0, 0,0\big)=0$. Then, we must verify that the partial derivative of $\mathcal{\widetilde F}_2$ at $(0, 0, 0, 0,0)$ defines a bounded linear operator, invertible with continuous inverse. From \eqref{eqbbb2} we must check that the ODE
\begin{align} 
-u_{0,x}+\dfrac{\delta}{(\beta+i\alpha)^2}&\sin\left(\dfrac{B+K}{2}\right)-\dfrac{u_0}{2(\beta+i\alpha)}\cos\left(\dfrac{B+K}{2}\right)
    \nonumber \\ & \quad+\delta \sin\left(\dfrac{B-K}{2}\right)-\dfrac{(\beta+i\alpha)u_0}{2}\cos\left(\dfrac{B-K}{2}\right) \  = \ f, \label{eqbb2}
\end{align} 
has a unique solution $(u_0,\,\delta)$ such that $u_0\in  H^1\left(\mathbb{R};\mathbb{C}\right)$, $\delta\in \Com$, for each $f\in H^1(\mathbb{R};\mathbb{C})$. Rewriting  \eqref{eqbb2}, calling $f\mapsto -f$, and using that $(\bt+i\al)^{-1} = \bt-i\al$, we have
\begin{align}
    u_{0,x}+\bigg(\dfrac{(\beta-i\alpha)}{2}&\cos\left(\dfrac{B+K}{2}\right)+ \dfrac{(\beta+i\alpha)}{2}\cos\left(\dfrac{B-K}{2}\right)\bigg) u_0  \nonumber \\ & \quad  = \  f + \dfrac{\delta}{(\beta+i\alpha)^2}\sin\left(\dfrac{B+K}{2}\right)+\delta\sin\left(\dfrac{B-K}{2}\right).\label{eqbbbb33}
\end{align}

Consider $\mu_B=\mu_B(x)$ defined in Lemma \ref{mu_b_bajada}, see \eqref{mu_breather}.  Thanks to  \eqref{edomu}, we have
\begin{align*}
    u_0 \ = \ \dfrac{1}{\mu_B}\int_{-\infty}^x\mu_B \left(f + \delta(\beta-i\alpha)^2\sin\left(\dfrac{B+K}{2}\right)+\delta\sin\left(\dfrac{B-K}{2}\right)\right).
\end{align*}
Recalling that  $(B,B_t)$ and $(K,K_t)$ satisfy \eqref{bk1}, and since $\alpha^2+\beta^2=1$, we arrive to the simplified expression
\begin{align*}
u_0 = \dfrac{1}{\mu_B}\int_{-\infty}^x\mu_B \left(f + \delta(\beta-i\alpha)\left(B_x-K_t\right)\right).
\end{align*}
From \eqref{integral}, we know that $\int_\mathbb{R} \mu_B \cdot\left(B_x-K_t\right) \neq 0.$ Consequently, we can choose $\delta\in\mathbb{C}$ in a unique fashion and such that   
\begin{align}\label{cndcero}
\int_{\mathbb{R}}\mu_B \left(f + \delta(\beta-i\alpha)\left(B_x-K_t\right)\right)=0.
\end{align}
Note that from this choice we have  $\vert \delta\vert \leq C\Vert f\Vert_{L^2(\mathbb{R})}$, where $C$ is a constant depending on $\bt$ and 
%\begin{align*}
%\left\vert\int_\mathbb{R}\mu(x)\cdot \left(B_x-K_t\right)dx\right\vert\neq 0,
%\end{align*}
 $\Vert \mu_B\Vert_{L^2(\mathbb{R};\mathbb{C})}$. Let us prove that $u_0\in H^1(\mathbb{R};\mathbb{C})$. Indeed, from 
\begin{align*}
\lim_{x\to\pm \infty}f(x)  =  \lim_{x\to\pm \infty}\mu_B(x)=\lim_{x\to\pm \infty}B_x=\lim_{x\to\pm \infty}K_t=0,
\end{align*}  
(see \eqref{mu_breather}, \eqref{B_x} and \eqref{Kt}), we obtain  
\begin{align*}
\lim_{x\to\pm \infty} u_0 =\lim_{x\to\pm\infty}\dfrac{\mu_B}{(\mu_B)_x}(f+\delta(\beta-i\alpha)(B_x-K_t))= 0.
\end{align*}
Lastly, note that if $s\leq x\ll -1$, then we have that
\begin{align*}
\left\vert\dfrac{\mu_B(s)}{\mu_B(x)}\right\vert \ \leq \ C \left\vert\dfrac{\cosh(\beta(x+x_2))}{\cosh(\beta(s+x_2))}\right\vert \ \leq \ C\left\vert\exp\big(\beta(s-x)\big)\right\vert  . 
\end{align*}
Hence, for $x\ll -1$ we get 
\begin{align*}
    \vert u_0(x)\vert \ & \leq \ C\,\int_{-\infty}^x e^{-\beta(x-s)} \,\left\vert f+\delta(\beta-i\alpha)\left(B_x-K_t\right)\right\vert ds
    \\ & \leq \ C e^{-\beta x}\,\star\, \left(\left\vert f(\cdot)  +\delta(\beta-i\alpha)\left(B_x-K_t\right)\right\vert \,\mathds{1}_{(-\infty,\,x]}(\cdot)\right).
\end{align*}
On the other hand, if $x\gg 1 $, using \eqref{cndcero} we have 
\[
u_0(x)=-\dfrac{1}{\mu_B}\int_x^\infty \mu_B \left(f + \delta(\beta-i\alpha)\left(B_x-K_t\right)\right).
\]
From this last result, it is not difficult to show decay estimates for $x\gg 1$, changing $e^{-\beta x}$ by $e^{\beta x}$. In consequence, from Young's inequality,  
\begin{align*}
    \Vert u_0\Vert_{L^2(\mathbb{R};\mathbb{C})} \lesssim \left\Vert f +\delta(\beta-i\alpha)\left(B_x-K_t\right)\right\Vert_{L^2(\mathbb{R};\mathbb{C})}.
\end{align*} 
Finally, in order to prove  $u_0\in H^1$ we only must check that $u_{0,x}\in L^2(\mathbb{R};\mathbb{C})$, which is direct if we recall that l that $f\in H^1$and $(\mu_B)_x/\mu_B$ is bounded. Therefore, $u_{0}\in H^1(\mathbb{R};\mathbb{C})$. The Implicit Function Theorem guaranties \eqref{F_new}. The proof of  \eqref{F_new_2} is direct from the smallness of the data.
\end{proof}

%
%
%\subsection{Perturbaciones del perfil de kink complejo} En esta sección nos proponemos a estudiar la invertibilidad de la T.B. para pequeñas perturbaciones de $(K,K_t)$ en $H^1(\mathbb{R};\mathbb{C})\times L^2(\mathbb{R};\mathbb{C})$. 

Finally, the second item in Proposition \ref{Descenso_global} is consequence of the following:

\begin{lem}\label{kbajada}
Let $(K,K_t)$ be a complex-valued kink profile with scaling parameter $\beta\in(-1,1)\setminus\{0\}$ and shifts $x_1,x_2\in\mathbb{R}$, and such that  $x_1$ do not satisfy \eqref{x1_cond}. Then, there are constants $\nu_0>0$ and $C>0$ such that for all $0<\nu<\nu_0$ and for all $(u_0,s_0)\in H^1\left(\mathbb{R};\mathbb{C}\right)\times L^2\left(\mathbb{R};\mathbb{C}\right)$ such that  
\[
\Vert u_0\Vert_{H^1(\mathbb{R};\mathbb{C})} +\Vert s_0\Vert_{L^2(\mathbb{R};\mathbb{C})} <\nu, 
\]
there are unique $(y_0,v_0,\widetilde{\delta})$ defined in an open subset of $H^1\left(\mathbb{R};\mathbb{C}\right)\times L^2\left(\mathbb{R};\mathbb{C}\right)\times\mathbb{C}$ and such that
\be\label{F_new_K}
\mathcal{F}(K+u_0,K_t+s_0,y_0,v_0,\beta-i\alpha+\tilde{\delta})  = (0,0), 
\ee 
and moreover, 
\be\label{F_new2_K}
 \Vert (y_0,v_0)\Vert_{H^1\times L^2}+\vert \tilde{\delta}\vert < C\nu.
\ee
%y que \begin{align*}
%\lim_{x\to-\infty}K+u_0+y_0=0, \ \ \ \lim_{x\to\infty}K+u_0+y_0=2\pi
%\end{align*}
\end{lem}
\begin{proof}[Idea of proof] The proof is very similar to that of Lemma \ref{bbajada}, so we will only sketch the main steps. 

\medskip

Let $(u_0,s_0)\in H^1\left(\mathbb{R};\mathbb{C}\right)\times L^2\left(\mathbb{R};\mathbb{C}\right)$ be given. Consider the rescaled BT functionals (see \eqref{f1}-\eqref{f2} and Lemma \ref{back_kink}),  
\begin{align}
   & \mathcal{\widetilde F}_1 \big(u_0, \ s_0,\,  y_0, \, v_0,\, \tilde{\delta}\big) \nonu \\
     &\ = \  u_{0,x}-v_0 - \dfrac{1}{\beta-i\alpha+\tilde{\delta}}\sin\left(\dfrac{K+u_0+y_0}{2}\right) + \dfrac{1}{\beta-i\alpha}\sin\left(\dfrac{K}{2}\right)
   \nonumber \\ &\qquad  - (\beta-i\alpha+\tilde{\delta}) \sin\left(\dfrac{K+u_0-y_0}{2}\right) +  (\beta-i\alpha ) \sin\left(\dfrac{K}{2}\right), \label{eqkkk1}\\  
    & \mathcal{\widetilde F}_2 \big(u_0, \ s_0,\,  y_0, \, v_0,\, \tilde{\delta}\big) \nonu \\
    &  \ = \  s_0-y_{0,x}-\dfrac{1}{\beta-i\alpha+\tilde{\delta}} \sin\left(\dfrac{K+u_0+y_0}{2}\right) + \dfrac{1}{\beta-i\alpha} \sin\left(\dfrac{K}{2}\right)  \nonumber \\ &  \qquad +(\beta-i\alpha+\tilde{\delta})\sin\left(\dfrac{K+u_0-y_0}{2}\right)  - (\beta-i\alpha)\sin\left(\dfrac{K}{2}\right). \label{eqkkk2} 
\end{align}  
%\begin{align}
%    \mathcal F_1 \big(K+u_0, \ K_t+s_0,\,  y_0, \, v_0,\, \beta-i\alpha+\tilde{\delta}\big) \ = \  & K_x+u_{0,x}-v_0- \dfrac{1}{\beta-i\alpha+\tilde{\delta}}\sin\left(\dfrac{K+u_0+y_0}{2}\right)
%   \nonumber \\ &  - (\beta-i\alpha+\tilde{\delta}) \sin\left(\dfrac{K+u_0-y_0}{2}\right), \label{eqkkk10}
%    \\ \nonumber \\ \mathcal F_2 \big(K+u_0, \ K_t+s_0,\,  y_0, \, v_0,\, \beta-i\alpha+\tilde{\delta}\big) \ = \ & K_t+s_0-y_{0,x}-\dfrac{1}{\beta-i\alpha+\tilde{\delta}} \sin\left(\dfrac{K+u_0+y_0}{2}\right)  \nonumber \\ &   +(\beta-i\alpha+\tilde{\delta})\sin\left(\dfrac{K+u_0-y_0}{2}\right), \label{eqkkk20} 
%\end{align}  
%
for some $(y_0,v_0,a)\in H^1(\mathbb{R};\mathbb{C})\times L^2(\mathbb{R};\mathbb{C})\times \mathbb{C}$. 
%Buscamos mostrar que existe una (única) elección de $(y_0,v_0,a)$ tales que 
%\[
%\mathcal F(K+u_0, K_t+s_0,  y_0, v_0, \beta-i\alpha+\tilde{\delta})= (0,0).
%\]
We will use the Implicit Function Theorem on the previous system. Note that once we find $(y_0,\,\tilde{\delta})$,  $\,v_0$ rests completely determined from \eqref{eqkkk1}, so that we only need to solve for \eqref{eqkkk2} and $(y_0,\tilde{\delta})$. 

%\medskip
%
%Redefiniendo variables por simplicidad, y usando \eqref{F(K,0)}, tenemos los funcionales
 
%Recordemos que ambos $u_0, s_0$ son datos del problema. Debemos pues resolver las correspondientes ecuaciones  $\mathcal{\widetilde F}_1 =\mathcal{\widetilde F}_2 =0$ para las desconocidas $(y_0,v_0,\tilde{\delta})$. Primero notemos que, conocidos   $y_0$ y $\tilde{\delta}$, el valor de $v_0$ es evidente de \eqref{eqkkk1}. Por lo tanto, debemos concentrarnos solamente en \eqref{eqkkk2}.

\medskip

%Claramente $\mathcal{\widetilde F}_2$ define un funcional $\mathcal{C}^1$ en una vecindad del origen, más aún, por el lemma \ref{f_kink} tenemos que $\mathcal{\widetilde F}_2(0,0,0,0,0)=0$. Para aplicar Funci'on Impl'icita, debemos verificar que la derivada de Gateaux de $\mathcal{\widetilde F}_2$ define un funcional lineal continuo, como as'i tambi'en un homeomorfismo entre los espacios involucrados. 
A simple computation in \eqref{eqkkk2} reveals that the problem is reduced to prove that the equation 
\be\label{eqkk2}
\begin{aligned}
&  -y_{0,x}+\dfrac{\tilde{\delta}}{(\beta-i\alpha)^2}\sin\left(\dfrac{K}{2}\right)-\dfrac{y_0}{2(\beta-i\alpha)}\cos\left(\dfrac{K}{2}\right) \\
 & \qquad  +\tilde{\delta} \sin\left(\dfrac{K}{2}\right)-\dfrac{(\beta-i\alpha)y_0}{2}\cos\left(\dfrac{K}{2}\right) \  = \ f, 
\end{aligned} 
\ee
has a unique solution $(y_0,\,a)$ such that $y_0\in H^1(\mathbb{R};\mathbb{C})$, for each $f\in H^1\left(\mathbb{R};\mathbb{C}\right)$, continuous in function of the parameters of the problem. Simplifying \eqref{eqkk2}, we obtain the equation 
\begin{align}
    y_{0,x}+\beta \cos\left(\dfrac{K}{2}\right) y_0 =  f + \tilde{\delta}(1+(\beta+i\alpha)^2)\sin\left(\dfrac{K}{2}\right). \label{eqkkkk33}
\end{align}
%Luego, de lo anterior obtenemos que \begin{align*}
%    y_0 \ = \ \dfrac{1}{\mu}\int_{-\infty}^x\mu \left(f + \tilde{\delta}(\beta+i\alpha)^2\sin\left(\dfrac{K}{2}\right)+\tilde{\delta}\sin\left(\dfrac{K}{2}\right)\right).
%\end{align*}
Recalling that $\alpha^2+\beta^2=1$, and that $\mu_K$ in \eqref{mu_kink} is integrant factor for the last ODE, we obtain
\begin{align*}
y_0 = \dfrac{1}{\mu_K}\int_{-\infty}^x\mu_K \left(f + \dfrac{2\tilde{\delta}\beta }{\beta-i\alpha}\sin\left(\dfrac{K}{2}\right)\right).
\end{align*}
On the other hand, from \eqref{nonzero_K} we conclude that we can choose $\tilde \delta \in\mathbb{C}$ in a unique form and such that   
\begin{align}\label{cndcerok}
\int_{\mathbb{R}}\mu_K \left(f + \dfrac{2\tilde{\delta}\beta}{\beta-i\alpha}\sin\left(\dfrac{K}{2}\right)\right) =0.
\end{align}
We also have $\vert \tilde{\delta}\vert \leq C\Vert f\Vert_{L^2(\mathbb{R})}$. Finally, note that
%con $C$ una constante que depende de 
%\begin{align*}
%\left\vert\int_\mathbb{R}\mu(x)\cdot \sin\left(\dfrac{K}{2}\right)\, dx\right\vert\,\neq\, 0,
%\end{align*}
%y de $\Vert \mu\Vert_{L^2(\mathbb{R})}$. 
%Sólo resta probar que $y_0\in H^1(\mathbb{R};\mathbb{C})$. 
%En efecto, del hecho que 
\begin{align*}
\lim_{x\to\pm \infty}f(x) = \lim_{x\to\pm \infty}\sin\left(\dfrac{K}{2}\right)=0,
\end{align*} 
and that from \eqref{sinK_cosK} $\lim_{x\to\pm\infty} \, \bt \cos(\frac K2) \ = \ \mp \beta$. The rest of the proof is very similar to the proof of Lemma \ref{bbajada}.
\end{proof}

\section{Perturbations of breathers: inverse dynamics}\label{7}

\subsection{Preliminaries} In this Section we will continue assuming $\K=\Com$ in Definition \ref{fperturbacion}. Proposition \ref{Descenso_global} showed us the connection between a vicinity of $(B,B_t)$ with another vicinity of the vacuum solution. Our objetive now will be the proof of an inverse result. Important differences will appear in this case, in particular we will need the orthogonality conditions \eqref{Modulacion_temporal} in the case of the breather:
\be\label{Modulacion_temporal_breather} 
\int_{\mathbb{R}} (z,w)\cdot \big(B_1,(B_1)_t\big)(t,x)dx=\int_{\mathbb{R}} (z,w)\cdot\big(B_2,(B_2)_t\big)(t,x)dx=0.
\ee
Recall that $B_1$ and $B_2$, defined in general in \eqref{D1}-\eqref{D2}, are given explicitly in \eqref{B1B2}. 

\begin{lem}[Nondegenerate profile $\widetilde B_0$]\label{tB0}
Let us define the function
\be\label{B0}
\begin{aligned}
\widetilde B_0:=  B_{xxt} &{}  + \dfrac{1}{2}(\beta-i\alpha)(B-B_{t,x}) \cos\left(\dfrac{B+K}{2}\right) \\
&  {} - \frac12(\beta+i\alpha) (B+B_{t,x})\cos\left(\dfrac{B-K}{2}\right).
\end{aligned}
\ee
Then $\widetilde B_0$ is in the Schwartz class, provided  $x_1$ do not satisfy \eqref{x1_cond}. 
%y se tiene la identidad
%\be\label{B0_simple}
%\begin{aligned}
%\widetilde B_0 =&~ B_{txx}  -\bt B_{tx} -i\al B + \frac14 B \int_{-\infty}^x \Big(  \frac{i\al}{\bt} K_x^2 - B_xB_t \Big)\\
%&~ {} -\frac14 B_{tx} \int_{-\infty}^x \Big( K_x^2 -B_x^2 +  \frac{i\al}{\bt} B_xK_x -B_t K_x \Big).
%\end{aligned}
%\ee
Additionally, we have the nondegeneracy condition
\be\label{No_deg}
\int_\R \widetilde B_0 K_x \in \R \backslash\{0\}.
\ee
%\[
%\dfrac{1}{4\alpha^2\beta^2} (\bt B_t - i \al B_x)
%\]
\end{lem}

%\begin{rem}
%La funci\'on $\widetilde B_0$ definida en \eqref{B0} puede simplificarse un poco m'as si usamos la identidad para $B_{txx}$
%\[
%B_{txx} +\frac18 B_t^3 +\frac38 B_x^2 B_t -\frac14 B_t \cos B +\Big(\bt^2 -\frac14\Big)B_t=0,
%\]
%que fue obtenida en \cite[eqn. (2.5)]{AMP1}. Sin embargo, esta simplificaci\'on adicional no es necesaria por el momento.
%\end{rem}

\begin{proof}[Proof of Lemma \ref{tB0}]
The fact that $\widetilde B_0$ belongs to the Schwartz class is direct, provided that $K_t$ or $K_x$ are well-defined, which is the case if $x_1$ does not satisfy \eqref{x1_cond}. For the numerical computation of \eqref{No_deg}, see Appendix \ref{F}.
\end{proof}

In next result, we will translate one of the orthogonality conditions in \eqref{Modulacion_temporal_breather} to the case of a pair of functions $(u,s)(t)$ already unknown.

\begin{lem}[A priori almost orthogonality conditions]\label{apriori}

Let $t\in [0,T^*]$  be fixed as in Definition \ref{Tiempo}. Let $(z,w)(t)$ be $H^1\times L^2$ functions, and $x_1(t),x_2(t)$ modulational parameters given by Corollary \ref{Mod_Dinamica}, such that the second condition in \eqref{Modulacion_temporal_breather} and the bound \eqref{New_Tubular} are satisfied, and where $x_1(t)$ do not satisfy \eqref{x1_cond}. Finally, let $\delta\in \Com$ be a small fixed parameter, independent of time. Let us assume also that, for all $\eta>0$ small, there are functions $(u ,s)(t)$, defined in $H^1\left(\mathbb{R};\mathbb{C}\right)\times L^2\left(\mathbb{R};\mathbb{C}\right)$, and such that 
\be\label{A_priori0}
 \sup_{t\in [0,T^*]}\Vert (u ,s)(t)\Vert_{H^1\times L^2} \lesssim \eta,
\ee
and satisfy, for each $t\in [0,T^*]$:
\be\label{A_priori}
\mathcal F(B+z ,B_t+w ,K+u ,K_t+s ,\beta+i\alpha+\delta) =(0,0).
\ee 
Then, necessarily we have the almost orthogonality condition 
\be\label{New_condition}
\int_\mathbb{R}(u,s)\cdot (\widetilde B_0 , B)=\mathcal N(\delta,u, z),
\ee
where $\mathcal N$ satisfies $\mathcal N(0,0,z)=O(z^2)$ (see \eqref{MathcalN}), and $\widetilde B_0$ is given by \eqref{B0}.
\end{lem}

\begin{rem}
Condition \eqref{New_condition} can be recast as a necessary condition for $(u,s)$ close to zero, for being candidate to solution in \eqref{A_priori}. This condition, motivated by \eqref{Modulacion_temporal_breather}, implies that no every pair of functions $(u,s)$ is allowed at the time of solving the inverse dynamics of B\"acklund equations. This new condition will be essential to get uniqueness when applying the Implicit Function Theorem. See \cite{PS} for another approach to this method, involving the Lyapunov-Schmidt reduction.
\end{rem}

\begin{proof}  Explicitly writing \eqref{A_priori}, and using \eqref{f1}-\eqref{f2}, we get the equations
\begin{align*}
   & B_x+z_{x} -K_t-s  \\
   & \quad -\dfrac{1}{\beta+i\alpha+\delta}\sin\left(\dfrac{B+z+K+u}{2}\right) -(\beta+i\alpha+\delta)\sin\left(\dfrac{B+z-K-u}{2}\right)=0, \\ 
   &  B_t+w-K_x-u_{x}  \\
   & \quad -\dfrac{1}{\beta+i\alpha+\delta}\sin\left(\dfrac{B+z+K+u}{2}\right) +(\beta+i\alpha+\delta)\sin\left(\dfrac{B+z-K-u}{2}\right)=0.  
\end{align*}
Let us try to use the second orthogonality condition in \eqref{Modulacion_temporal} with $D=B$, so that $D_2 =B_2 =\partial_{x_2}B$ (see \eqref{D1}-\eqref{D2}). Since $B_2=B_x$ and $B_{2,t}=B_{t,x}$ (see \eqref{perfil_Breather}), we have that multiplying the first equation above by $B$, and the second by $B_{t,x}$, and integrating on $x$, we will get (after some simple cancelations, see the end of Lemma \ref{back_kink})
\begin{align*}
&  \int B_2 z + i\ima\int BK_t + \int Bs  + \dfrac{1}{\beta+i\alpha+\delta}\int B \sin\left(\dfrac{B+z+K+u}{2}\right)  \\
& \qquad  + (\beta+i\alpha+\delta)\int B\sin\left(\dfrac{B+z-K-u}{2}\right)=0,\\
& \int B_{2,t} w - i\ima \int B_{t,x}K_x- \int B_{2,t}u_{x}   -\dfrac{1}{\beta+i\alpha+\delta} \int B_{t,x}\sin\left(\dfrac{B+z+K+u}{2}\right)\\
& \qquad  +(\beta+i\alpha+\delta)\int B_{t,x}\sin\left(\dfrac{B+z-K-u}{2}\right)=0.
\end{align*}
Adding both equations, and using \eqref{Modulacion_temporal}, we have
\begin{align}
& \int B_{xxt}u + \int Bs  + \dfrac{1}{\beta+i\alpha+\delta}\int (B-B_{t,x}) \sin\left(\dfrac{B+z+K+u}{2}\right) \nonu \\
& \qquad + (\beta+i\alpha+\delta)\int (B+B_{t,x})\sin\left(\dfrac{B+z-K-u}{2}\right) \nonu\\
& \quad = i\ima \int B_{t,x}K_x - i\ima\int BK_t .\label{Suma}
\end{align}
The term $\sin\left(\frac{B+z \pm K \pm u}{2}\right)$ can be expanded as 
\begin{align*}
 \sin\left(\dfrac{B+z \pm K \pm u}{2}\right) =&~ \sin\left(\dfrac{B\pm K}{2}\right) + \frac12 \cos\left(\dfrac{B\pm K}{2}\right)(z\pm u) \\
& + \mathcal N_{2,\pm}(x,z,u).
% \sin\left(\dfrac{B+z-K-u}{2}\right) = & \sin\left(\dfrac{B-K}{2}\right) + \frac12 \cos\left(\dfrac{B-K}{2}\right)(z-u)\\
% & + \mathcal N_{2,2}(x,z,u).
\end{align*}
Here, $\mathcal N_{2,\pm}$ are nonlinear functions in $(x,z,u)$, quadratic in $(z,u)$. Hence, replacing in \eqref{Suma} we get
\begin{align*}
& \int B_{xxt}u + \int Bs  + \dfrac{1}{2(\beta+i\alpha+\delta)}\int (B-B_{t,x}) \cos\left(\dfrac{B+K}{2}\right)u   \\
& \qquad  - \frac12(\beta+i\alpha+\delta)\int (B+B_{t,x})\cos\left(\dfrac{B-K}{2}\right)u \\
& \quad =  i\ima \int B_{t,x}K_x - i\ima\int BK_t  - \dfrac{1}{\beta+i\alpha+\delta}\int (B-B_{t,x}) \sin\left(\dfrac{B+K}{2}\right) \\ 
& \qquad   - (\beta+i\alpha+\delta)\int (B+B_{t,x})\sin\left(\dfrac{B-K}{2}\right) \\
& \qquad - \dfrac{1}{2(\beta+i\alpha+\delta)}\int (B-B_{t,x}) \cos\left(\dfrac{B+K}{2}\right)z \\
& \qquad- \frac12(\beta+i\alpha+\delta)\int (B+B_{t,x})\cos\left(\dfrac{B-K}{2}\right)z  +\mathcal N_{2}(z,u).
\end{align*}
Here, $\mathcal N_{2}$ is a nonlinear term of second order in $(z,u)$. Let us define 
\[
\begin{aligned}
\widetilde B_\delta:= &~ B_{txx}+ \dfrac{1}{2(\beta+i\alpha+\delta)}(B-B_{t,x}) \cos\left(\dfrac{B+K}{2}\right)  \\
& \quad {}- \frac12(\beta+i\alpha+\delta) (B+B_{t,x})\cos\left(\dfrac{B-K}{2}\right).
\end{aligned}
\]
Thanks to Lemma \ref{tB0}, $\widetilde B_\delta = \widetilde B_0 +O_{S}(\delta)$, where $O_{S}(\delta)$ represents a function in the Schwartz class, bounded by $\delta$, uniformly in space. Then,
\begin{align}
 \int \widetilde B_0 u + \int Bs  =  &~   i\ima \int B_{t,x}K_x - i\ima\int BK_t  \nonu\\ 
&   - \dfrac{1}{\beta+i\alpha+\delta}\int (B-B_{t,x}) \sin\left(\dfrac{B+K}{2}\right) \nonu \\
&  - (\beta+i\alpha+\delta)\int (B+B_{t,x})\sin\left(\dfrac{B-K}{2}\right) \nonu \\
&  + \mathcal N_{1,1}(\delta,z)  +\mathcal N_{2}(z,u). \label{PaPa}
\end{align}
Here, $\mathcal N_{1,1}(\delta, z)$ represents a quadratic term in $\delta, z$, with $\mathcal N_{1,1}(0, z)=\mathcal N_{1,1}(\delta, 0)=0$. Lastly, we will use the following result:

\begin{lem}\label{Cancelacion}
For each $\bt>0$, and $x_1,x_2$ shifts such that $x_1$ does not satisfy \eqref{x1_cond}, we have
\begin{align}
&  i\ima \int B_{tx}K_x - i\ima\int BK_t  - \dfrac{1}{\beta+i\alpha}\int (B-B_{t,x}) \sin\left(\dfrac{B+K}{2}\right) \nonu \\ 
& \qquad  - (\beta+i\alpha)\int (B+B_{t,x})\sin\left(\dfrac{B-K}{2}\right) =0.\label{Cancelacion1}
\end{align}
\end{lem}
Assuming this result, we have 
\begin{align*}
&  i\ima \int B_{tx}K_x - i\ima\int BK_t  - \dfrac{1}{\beta+i\alpha+\delta}\int (B-B_{t,x}) \sin\left(\dfrac{B+K}{2}\right) \\ 
& \qquad   - (\beta+i\alpha+\delta)\int (B+B_{t,x})\sin\left(\dfrac{B-K}{2}\right) =\mathcal N_{1,2}(\delta),
\end{align*}
where $\mathcal N_{1,2}$ is a term of first order in $\delta$, with $\mathcal N_{1,2}(0)=0$.  Therefore, coming back to \eqref{PaPa}, we can conclude that
\begin{align*}
 \int \widetilde B_0 u + \int Bs = \mathcal N_{1,2}(\delta) + \mathcal N_{1,1}(\delta, z)  +\mathcal N_{2}(z,u),
\end{align*}
which shows \eqref{New_condition}. For further references, $\mathcal N$ is given by
\be\label{MathcalN}
\mathcal{N}(\delta,u,z) :=  \mathcal N_{1,2}(\delta) + \mathcal N_{1,1}(\delta, z)  +\mathcal N_{2}(z,u).
\ee
Clearly, $\mathcal{N}(0,0,z)=O(z^2)$.
\end{proof}

\begin{proof}[Proof of Lemma \ref{Cancelacion}]
From \eqref{bk1}-\eqref{bk2}, we have
\[
\begin{aligned}
& \hbox{RHS of } \eqref{Cancelacion1}  \\
& = ~ {}  i\ima \int B_{tx}K_x - i\ima\int BK_t  - \int B( B_x-K_t ) +\int B_{tx} (B_t-K_x)\\
& = ~ {}  i\ima \int B_{tx}K_x - i\ima\int BK_t  +  \int B K_t  -\int B_{tx} K_x= 0.
\end{aligned}
\]
Last cancelations are coming from the parity properties of $K_x$ and $K_t$, see Lemma \ref{back_kink}.
\end{proof}

Our second result is the following (compare with Proposition \ref{Descenso_global}):

\begin{prop}[Ascent to the perturbed breather profile]\label{Ascenso_global}
Let $(B,B_t)$ be a breather profile as in Definition \ref{Perfil_B}, with scaling parameter $\beta\in(-1,1)$ and shifts $x_1,x_2\in\mathbb{R}$, and such that $x_1$ does not satisfy \eqref{x1_cond}. Let also $(K,K_t)$ denote the complex-valued kink profile associated to $(B,B_t)$, that is, with same parameters as $(B,B_t)$. Then, there exist constants $\eta_1>0$ and $C>0$ such that for all $0<\eta<\eta_1$ and for all $(y,v,\tilde{\delta})\in H^1\left(\mathbb{R}\right)\times L^2\left(\mathbb{R}\right)\in\mathbb{R}$ such that\footnote{Note that $(y,v,\tilde{\delta})$ are real-valued.} 
\begin{align*}
\Vert (y,v)\Vert_{H^1(\mathbb{R})\times L^2(\mathbb{R})}+\vert \tilde{\delta}\vert \leq\eta, 
\end{align*}
then the following is satisfied:
\ben
\item There are unique $(u,s)$ defined in a subset of $H^1\left(\mathbb{R};\mathbb{C}\right)\times L^2\left(\mathbb{R};\mathbb{C}\right)$ such that \begin{align*}
\mathcal F(K+u,K_t+s,y,v,\beta-i\alpha+\tilde{\delta}) & = (0,0),
\end{align*}
\eqref{New_condition} is satisfied, and
\begin{align*}
\Vert(u,s)\Vert_{H^1\times L^2}\leq C\eta.
\end{align*}
\item For all $\delta>0$ small enough, making $\eta_1$ smaller if necessary, there are unique $(z ,w )$, defined in a subset of $H^1\left(\mathbb{R};\mathbb{C}\right)\times L^2\left(\mathbb{R};\mathbb{C}\right)$, and such that  
\begin{align*}
\mathcal F(B+z ,B_t+w ,K+u ,K_t+s ,\beta+i\alpha+\delta) &=(0,0), 
\end{align*} 
\eqref{Modulacion_temporal_breather} is satisfied for $B_1$, and also, 
\begin{align*}
 \Vert (z ,w )\Vert_{H^1\times L^2}+\vert \delta \vert \leq C\eta.
\end{align*}
\een
\end{prop}

For the proof of this result we will use several auxiliary results. The first item in Proposition \ref{Ascenso_global} is consequence of the following result.

\begin{prop}\label{ksubida}
Let $(K,K_t)$ be a complex-valued kink profile, with scaling parameter $\beta\in(-1,1)$, $\beta\neq 0$, and shifts $x_1,x_2\in\mathbb{R}$. Then, there are constants  $\nu_1>0$ and $C>0$ such that for all $0<\nu<\nu_1$ and for all $(y ,v ,\tilde{\delta})\in H^1(\mathbb{R};\mathbb{C})\times L^2(\mathbb{R};\mathbb{C})\times\mathbb{C}$ such that 
\begin{align*}
\Vert y \Vert_{H^1(\mathbb{R};\mathbb{C})}+\vert \tilde{\delta}\vert < \nu,
\end{align*}
there are unique $(u ,s )\in H^1(\mathbb{R};\mathbb{C})\times L^2(\mathbb{R};\mathbb{C})$ such that 
\smallskip
\ben
\item Smallness. We have
\[
\Vert (u ,s )\Vert_{H^1\times L^2}\leq C\nu, 
\]
\item The BT are satisfied, in the sense that $(u,s)$ solve (see \eqref{New_condition}): 
\be\label{bk_subida_cerokink}
\mathcal{F}(K+u ,\,K_t+s ,\,y ,\,v ,\,\beta-i\alpha+\tilde{\delta})\equiv (0,0), 
\ee
\be\label{New_Condition_A}
\int_\mathbb{R}(u,s)\cdot (\widetilde B_0 , B)=\mathcal N(\delta,u, z), %\int_\mathbb{R}u_{0,x}\left(\dfrac{1}{\mu}\right)_x=0.
\ee
where $\mathcal N$ was defined in \eqref{MathcalN}.
\een
\end{prop}

\begin{rem}
Note that \eqref{New_Condition_A} is a necessary condition to get
\[
\int_{\mathbb{R}} (z,w)\cdot\big(B_2,(B_2)_t\big)(t,x)dx=0,
\]
obtained via modulation theory. Additionally, \eqref{New_Condition_A} ensures existence and uniqueness for the solution constructed via Implicit Function.
\end{rem}

\begin{proof}[Proof of Proposition \ref{ksubida}] Let $(y ,v ,\tilde{\delta})\in H^1\left(\mathbb{R};\mathbb{C}\right)\times L^2\left(\mathbb{R};\mathbb{C}\right)\times\mathbb{C}$ be given and small. Let us consider the BT functionals equal zero: 
\be\label{ks5}
\begin{aligned}
    \mathcal{F}_1 & \big(K+u , K_t+s , y , v , \beta-i\alpha+\tilde{\delta}\big)  = ~  K_x+u_{0,x}-v   \\ & {} -\dfrac{1}{\beta-i\alpha+\tilde{\delta}}\sin\left(\dfrac{K+u +y }{2}\right)-(\beta-i\alpha+\tilde{\delta})\sin\left(\dfrac{K+u -y }{2}\right) =0,
\end{aligned}
\ee
\be\label{ks6}
\begin{aligned}   
\mathcal{F}_2& \big(K+u , K_t+s ,y , v , \beta-i\alpha+\tilde{\delta}\big) = ~  s -y_{0,x}  \\ & {}-\dfrac{1}{\beta-i\alpha+\tilde{\delta}}\sin\left(\dfrac{K+u +y }{2}\right) +(\beta-i\alpha+\tilde{\delta})\sin\left(\dfrac{K+u -y }{2}\right)=0,
\end{aligned}
\ee% 
plus the almost orthogonality condition \eqref{New_Condition_A},
%\be\label{ks7}
%\int_\mathbb{R}(u,s)\cdot (\widetilde B_0 , B)- \mathcal N(\tilde \delta,u, z) =0, %\int_\mathbb{R}u_{0,x}\left(\dfrac{1}{\mu}\right)_x=0.
%\ee
for some $(u ,s )\in H^1(\mathbb{R};\mathbb{C})\times L^2(\mathbb{R};\mathbb{C})$. Here, $z$ in \eqref{New_Condition_A} is given by a modulation (in a fixed time $t$ far enough from the times $t_k$ in \eqref{t_k}) on the breather profile. We look for a unique choice of $(u ,s )$ such that \eqref{ks5}-\eqref{ks6} are satisfied. %Note that once found $u $, $s $ is completely determined from \eqref{ks5}, therefore we only solve \eqref{ks6}.
\medskip

For simplicity, we shall redefine variables. Using $\mathcal F(K,K_t,0,0,\beta-i\alpha)= (0,0)$ (Lemma \ref{back_kink}), we have
\begin{align}
    \widetilde{\mathcal{F}}_1(u ,s , y , v , \tilde{\delta})  =& ~  u_{x}-v  -\dfrac{1}{\beta-i\alpha+\tilde{\delta}}\sin\left(\dfrac{K+u +y }{2}\right)+\dfrac{1}{\beta-i\alpha}\sin\left(\dfrac{K}{2}\right) \nonumber
    \\ & {} -(\beta-i\alpha+\tilde{\delta})\sin\left(\dfrac{K+u -y }{2}\right)+(\beta-i\alpha)\sin\left(\dfrac{K}{2}\right), \label{funcional_subida_kink_cero_1}
\\ \nonumber   
\widetilde{\mathcal{F}}_2\big(u ,s ,y , v , \tilde{\delta}\big) =& ~  s -y_{x} -\dfrac{1}{\beta-i\alpha+\tilde{\delta}}\sin\left(\dfrac{K+u +y }{2}\right)+\dfrac{1}{\beta-i\alpha}\sin\left(\dfrac{K}{2}\right) \nonumber
\\ & {} +(\beta-i\alpha+\tilde{\delta})\sin\left(\dfrac{K+u -y }{2}\right)-(\beta-i\alpha)\sin\left(\dfrac{K}{2}\right). \label{funcional_subida_kink_cero_2}
\end{align}
Recall that $y $, $v $ and $\tilde{\delta}$ are data of the problem. We must then solve $\widetilde{\mathcal{F}}_1=\widetilde{\mathcal{F}}_2=0$ mas \eqref{New_Condition_A}, for the unknown $(u ,s )$. First of all, note that once we know $u $, the value of  $s $ is evident from \eqref{funcional_subida_kink_cero_2}. Therefore, we only solve \eqref{funcional_subida_kink_cero_1}, for $u$.
\medskip

Clearly $\widetilde{\mathcal{F}}_1$ defines a $\mathcal{C}^1$ functional in a neighborhood of the origin. Even more, using Lemma \ref{back_kink}, we have $\mathcal F(K,K_t,0,0,\beta-i\alpha)= (0,0)$ and then, $\widetilde{\mathcal{F}}_1(0,0,0,0)=0$. In order to apply Implicit Function, we must verify that the Gateaux derivative of $\widetilde{\mathcal{F}}_1$ defines a linear continuous functional, and a homeomorphism between the considered spaces. A simple checking in \eqref{funcional_subida_kink_cero_1} reveals that the problem is reduced to show that the equations 
\begin{align}
u_{x}-\dfrac{u}{2(\beta-i\alpha)}\cos\left(\dfrac{K}{2}\right)-\dfrac{(\beta-i\alpha)}{2}\cos\left(\dfrac{K}{2}\right)u  = & ~f,  \label{ks0} \\
 \int_\mathbb{R}(u,s)\cdot (\widetilde B_0 , B) = & ~c, \label{ks_2}
\end{align}    
have a unique solution $u\in H^1\left(\mathbb{R};\mathbb{C}\right)$, for all $f\in H^1\left(\mathbb{R};\mathbb{C}\right)$ and $c\in \Com$ given, continuous wrt the parameters of the problem. Simplifying \eqref{ks0} we get
\begin{align*}
    u_{x} -\beta\cos\left(\dfrac{K}{2}\right) \,u = f.  
\end{align*}
Recall that $\lim_{x\to\pm \infty}\cos\left(\frac{K}{2}\right)  =  \mp 1$ (see \eqref{sinK_cosK}). From $\mu_K$ in \eqref{mu_kink}, 
%\begin{align}\label{mu_subida_kink}
%\mu(x) \ = \ \exp\left(-\beta\int_0^x\cos\left(\dfrac{K}{2}\right)\right).
%\end{align}
%De lo anterior se deduce que $\mu(x)$ crece de manera exponencial cuando $x\to\pm\infty$. Luego, 
we have  
\[
\begin{aligned}
    u \ = & \ \dfrac{\mu_K}{\mu_K(0)}\,u(x=0)+ \mu_K \int_0^x \frac{f}{\mu_K}. %\\
    %=& \  \frac{\cosh(\beta x_2+i\alpha x_1) u(x=0) }{\cosh(\beta (x+x_2)+i\alpha x_1)} \\
    %& \ + \frac1{\cosh(\beta (x+x_2)+i\alpha x_1)} \int_0^x \cosh(\beta (\sigma +x_2)+i\alpha x_1) f (\sigma) d\sigma.
\end{aligned}
\]
In what follows, \eqref{ks_2} will help us to find $u$ in a unique form. Indeed, it is enough to show that 
\[
\int \widetilde B_0 \mu_K \sim \int \widetilde B_0(x) \sech(\beta (x+x_2)+i\alpha x_1))dx  \sim \int \widetilde B_0 K_x \neq 0,
\]
which holds thanks to \eqref{No_deg}. The rest of the proof is similar to the one for Lemma \ref{kbajada}.
%Mostremos ahora que $u \in L^2(\mathbb{R})$. Para ello, notemos que si $x>0$ grande, de \eqref{limcos1}, utilizando la forma explícita de $\mu$ dada en  \eqref{mu_subida_kink} y argumentos análogos a los utilizados en el capítulo anterior, obtenemos que \begin{align*} 
%\vert u (x)\vert \ & \leq \ C\int ^xe^{-\beta(x-s)}\,\left\vert f\right\vert ds = \ C\,e^{-\beta x}\star \left\vert f\right\vert\cdot\mathds{1}_{(0,\,x]}(\cdot) 
%\end{align*}
%Para $x<0$ se puede establecer un resultado similar procediendo de manera análoga. Por lo tanto, usando la desigualdad de Young obtenemos que 
%\begin{align*}
%    \Vert u \Vert_{L^2(\mathbb{R};\mathbb{C})} \ \lesssim \ \Vert f\Vert_{L^2(\mathbb{R};\mathbb{C})}.
%\end{align*}
%Tal como se deseaba. Para ver que $u \in H^1$ sólo falta mostrar que $u_{x}\in L^2$. Para ello basta derivar directamente, notar que $\mu_x/\mu$ es acotado y ocupar directamente las estimaciones recien mostradas. Luego, por el teorema de la función implícita se concluye lo deseado.
\end{proof}

The second item in Proposition \ref{Ascenso_global} requires the following previous result.

\begin{lem}\label{mu_b_subida} Let $(B,B_t)$ and $(K,K_t)$ breather and complex-valued kink profiles respectively, both with parameters $\beta\in(-1,1)\setminus\{0\}$, shifts $x_1,x_2\in\mathbb{R}$ and such that \eqref{x1_cond} is not satisfied. Let us consider
\begin{align*}
    \mu^B(x)=\dfrac{1}{\mu_B}(x)  := \dfrac{\alpha^2\cosh^2(\beta (x+x_2))+\beta^2\sin^2(\alpha x_1)}{\cosh(\beta (x+x_2)+i\alpha x_1)}%=\dfrac{1}{4\alpha^2\beta} B_t -\dfrac{i}{4\alpha\beta^2}B_x.
\end{align*}
Then, $\mu^B(x)$ solves the ODE 
\begin{align}
 \mu_x +\left(\dfrac{(\beta-i\alpha)}{2}\cos\left(\dfrac{B+K}{2}\right)+ \dfrac{(\beta+i\alpha)}{2}\cos\left(\dfrac{B-K}{2}\right)\right) \mu = 0.
\end{align}
\end{lem}

\begin{proof}
Direct from Lemma \ref{mu_b_bajada}.
\end{proof}

%%%%%%%%%%%%%%%%%%%%%%%%%%%%%%%%%%%%%%%%%%%%%%%%%%%%%%%%%%%%%%%%%%
%UN PAR DE CALCULOS UTILES 
%\begin{align*}
%B_{xx}&=K_{tx} - (\beta-i\alpha)\dfrac{B_x+K_x}{2}\cos\left(\dfrac{B+K}{2}\right)-(\beta+i\alpha)\dfrac{B_x-K_x}{2}\cos\left(\dfrac{B-K}{2}\right) \rightarrow -\dfrac{i}{\beta}
%\\ B_{tx} &=K_{xx} - (\beta-i\alpha)\dfrac{B_x+K_x}{2}\cos\left(\dfrac{B+K}{2}\right)+(\beta+i\alpha)\dfrac{B_x-K_x}{2}\cos\left(\dfrac{B-K}{2}\right)\rightarrow \dfrac{1}{\alpha}
%\end{align*}

%\begin{align*}
%4\alpha\beta \mu_x =(B_x+K_x)\cos\left(\dfrac{B+K}{2}\right)
%\end{align*}

%$$ - \dfrac{(\beta-i\alpha)}{2}\dfrac{B_x+K_x}{2}\sin\left(\dfrac{B+K}{2}\right)- \dfrac{(\beta+i\alpha)}{2}\dfrac{B_x-K_x}{2}\sin\left(\dfrac{B-K}{2}\right) =\dfrac{1}{4}\left(K_x^2-B_x^2+B_xK_t-B_tK_x\right)$$
%%%%%%%%%%%%%%%%%%%%%%%%%%%%%%%%%%%%%%%%%%%%%%%%%%%%%%%%%%%%%%%%%%

Finally, the second item in Proposition \ref{Ascenso_global} is consequence of the following result.

\begin{prop}\label{breather_subida}
Let $(B,B_t)$ and $(K,K_t)$ denote breather and complex-valued kink profiles respectively, both with scaling parameter $\beta\in(-1,1)\setminus\{0\}$ and shifts $x_1,x_2\in\mathbb{R}$, with $x_1$ not satisfying \eqref{x1_cond}. Then, there are constants $\eta_1>0$ and $C>0$ such that for all $0<\eta<\eta_1$ and for all $(u,s,\delta)\in H^1(\mathbb{R};\mathbb{C})\times L^2(\mathbb{R};\mathbb{C})\times\mathbb{C}$ such that 
\begin{align*}
\Vert u \Vert_{H^1(\mathbb{R};\mathbb{C})}+\vert \delta\vert < \eta,
\end{align*}
there are unique $(z,w)\in H^1(\mathbb{R};\mathbb{C})\times L^2(\mathbb{R};\mathbb{C})$ with 
\[
\Vert (z,w)\Vert_{H^1\times L^2}\leq C\eta,
\]
\[
\mathcal{F}(B+z,\,B_t+w,\,K+u,\,K_t+s,\,\beta+i\alpha+\delta)\equiv (0,0),
\]
and
\be\label{Ortho_B1}
\int_{\mathbb{R}} (z,w)\cdot \big(B_1,(B_1)_t\big)(t,x)dx=0. %\int_\mathbb{R}z_{x}\left(\dfrac{1}{\mu^1}\right)_x=0.
\ee
\end{prop}

%%%%%%%%%%%%%%%%%%%%%%%%%%%%%%%%%%%%%%%%%%%%%%%%%%%%%%%%%%%%%%%%%%%%%%%%%%%%%%%%%%%%%%

\begin{proof} Let $(u,s,\delta)\in H^1\left(\mathbb{R};\mathbb{C}\right)\times L^2\left(\mathbb{R};\mathbb{C}\right)\times\mathbb{C}$ be given. Let us consider the system of equations for the BT \eqref{f1}-\eqref{f2}:
\begin{align}
   & \mathcal{F}_1 \big(B+z, \, B_t+w,\, K + u, \, K_t+s,\, \beta+i\alpha+\delta\big) \nonu\\
   &  = B_x+z_{x} -K_t-s   -\dfrac{1}{\beta+i\alpha+\delta}\sin\left(\dfrac{B+z+K+u}{2}\right)  \nonumber \\
     & \quad -(\beta+i\alpha+\delta)\sin\left(\dfrac{B+z-K-u}{2}\right)=0, \label{brup1}\\ 
   & \mathcal{F}_2 \big(B+z,\, B_t+w,\, K+u,\, K_t+s,\, \beta+i\alpha+\delta\big) \nonu\\
   & = B_t+w-K_x-u_{x} -\dfrac{1}{\beta+i\alpha+\delta}\sin\left(\dfrac{B+z+K+u}{2}\right)  \nonumber \\
     & \quad +(\beta+i\alpha+\delta)\sin\left(\dfrac{B+z-K-u}{2}\right)=0, \label{brup2}  
\end{align}
for some $(z,w)\in H^1(\mathbb{R};\mathbb{C})\times L^2(\mathbb{R};\mathbb{C})$. 
%Buscamos mostrar que existe una única elección de $(z_0,w_0)$ tal que 
%\[
%\mathcal{F}(B+z_0,B_t+w_0,K+u_0,K_t+s_0,\beta+i\alpha+\delta)=(0,0).
%\]
We will use the Implicit Function Theorem in $(\mathcal{F}_1,\mathcal{F}_2)$. Note that once defined $z_0$, $w_0$ gets completely defined from \eqref{brup2}, therefore we just need to solve \eqref{brup1} para $z_0$. Thanks to the identity $\mathcal F(B,B_t,K,K_t,\beta+i\alpha)= (0,0)$, rearranging \eqref{brup1} and \eqref{brup2} we have 
\begin{align}
    & \mathcal{\widetilde F}_1 \big(z,w, u, s,\delta\big) \nonu \\
     & :  = \ z_{x}- s - \dfrac{1}{\beta+i\alpha+\delta}\sin\left(\dfrac{B+K+ z+u}{2}\right)  + \dfrac{1}{\beta+i\alpha}\sin\left(\dfrac{B+K}{2}\right)
   \nonumber \\ & \quad   -(\beta+i\alpha+\delta)\sin\left(\dfrac{B-K +z -u}{2}\right)   + (\beta+i\alpha )\sin\left(\dfrac{B-K }{2}\right) =0, \label{brup10}
    \\ 
    & \mathcal{\widetilde F}_2 \big(z, w, u, s,\delta\big) \nonu \\
    &  : = \ w_0 -u_{x} -\dfrac{1}{\beta+i\alpha+\delta}\sin\left(\dfrac{B+K+z+u}{2}\right)  + \dfrac{1}{\beta+i\alpha}\sin\left(\dfrac{B+K}{2}\right)  \nonumber \\  & \quad   +(\beta+i\alpha+\delta)\sin\left(\dfrac{B-K+z-u}{2}\right) -(\beta+i\alpha)\sin\left(\dfrac{B-K}{2}\right)=0 . \label{brup20} 
\end{align}  
Clearly $\widetilde{\mathcal{F}}_1$ defines a $\mathcal{C}^1$ functional near zero, moreover, we have $\widetilde{\mathcal{F}}_2(0,0,0,0,0)=0$. Then, from \eqref{brup10} we obtain that the problem is reduced to show that the equation
\[
\begin{aligned}
   z_{x} -\dfrac{z_0}{2(\beta+i\alpha)}\cos\left(\dfrac{B+K}{2}\right)-\dfrac{(\beta+i\alpha)z}{2}\cos\left(\dfrac{B-K}{2}\right) \ = \ f, 
\end{aligned}    
\]
possesses a unique solution $z\in H^1\left(\mathbb{R};\mathbb{C}\right)$ for all $f\in H^1\left(\mathbb{R};\mathbb{C}\right)$. Rearranging terms, 
\begin{align*}
    z_{x} -\left(\dfrac{\beta-i\alpha}{2}\cos\left(\dfrac{B+K}{2}\right)+\dfrac{\beta+i\alpha}{2}\cos\left(\dfrac{B-K}{2}\right)\right) \,z \ = \ f.  
\end{align*}
Thanks to Lemma \ref{mu_b_bajada}, we can use the integrant factor $1/\mu_B$ (exponentially increasing) defined in \eqref{mu_breather} and \eqref{Kx_deco} to obtain
%, tenemos Con esto, definimos 
%\[
%\mu^1(x):=\dfrac{\alpha^2\cosh^2(\beta(x+x_2))+\beta^2\sin^2(\alpha x_1)}{\cosh(\beta(x+x_2)+i\alpha x_1)}.
%\]
%Notemos que $\mu(x)$ crece de manera exponencial en espacio cuando $x\to\pm\infty$. Luego, de esto último y del lemma \ref{mu_b_subida} obtenemos que 
\be\label{z_cero_subida_br}
\begin{aligned}
    z \ = & \ \dfrac{\mu_B}{\mu_B(x=0)}\,z(x=0)+\mu_B \int_0^x \frac{f}{\mu_B}.
%  =& \ \dfrac{z(x=0)}{\mu_B(x=0)} \dfrac{\cosh(\beta (x+x_2)+i\alpha x_1)}{\alpha^2\cosh^2(\beta (x+x_2))+\beta^2\sin^2(\alpha x_1)}  \\
%    &\ + \dfrac{\cosh(\beta (x+x_2)+i\alpha x_1)}{\alpha^2\cosh^2(\beta (x+x_2))+\beta^2\sin^2(\alpha x_1)} \int_0^x \dfrac{f(\alpha^2\cosh^2(\beta (x+x_2))+\beta^2\sin^2(\alpha x_1)) }{\cosh(\beta (x+x_2)+i\alpha x_1)}  \\
%  =& \ \dfrac{z(x=0)}{\mu_B(x=0)} \dfrac{  \cosh(\beta (x+x_2)) \cos(\alpha x_1) + i\sinh(\bt(x+x_2)) \sin(\al x_1)}{\alpha^2\cosh^2(\beta (x+x_2))+\beta^2\sin^2(\alpha x_1)}  \\
%    &\ + \dfrac{\cosh(\beta (x+x_2)+i\alpha x_1)}{\alpha^2\cosh^2(\beta (x+x_2))+\beta^2\sin^2(\alpha x_1)} \int_0^x \dfrac{f(\alpha^2\cosh^2(\beta (x+x_2))+\beta^2\sin^2(\alpha x_1)) }{\cosh(\beta (x+x_2)+i\alpha x_1)} .
\end{aligned}
\ee
Note that $\mu_B$ is zero only if $x_1$ satisfies \eqref{x1_cond}, which is not the case. On the other hand, $z$ is well-defined from  condition \eqref{Ortho_B1}, which holds true because of
\[
\int_{\mathbb{R}} \mu_B B_1dx \sim \int_{\mathbb{R}} \mu_B B_t dx \neq 0. %\int_\mathbb{R}z_{x}\left(\dfrac{1}{\mu^1}\right)_x=0.
\]
In fact, thanks to \eqref{mu_breather} and Corollary \eqref{paridad}, and that $B_t$ is not zero,
\[
\int_{\mathbb{R}} \mu_B B_t dx \sim \int B_t (\bt B_t - i \al B_x) \sim \int B_t^2.
\]
%\eqref{Modulacion_temporal} aplicado a $D=B$ nos entrega
%\[
%\int_{\mathbb{R}} (z,w)\cdot \big(B_1,(B_1)_t\big)(t,x)dx=\int_{\mathbb{R}} (z,w)\cdot\big(B_2,(B_2)_t\big)(t,x)dx=0,
%\]
%con $B_j$ definidos en \eqref{B1B2}. 
The rest of the proof is very similar to the one in Lemma \ref{bbajada}.
\end{proof}

\section{Permutability}\label{8}

\subsection{Preliminaries} In this section we want to answer the following question: are $(y_0,v_0)$, the functions obtained in Proposition \ref{Descenso_global}, real-valued?  We will show here that, if $(z_0,w_0)$ in Proposition \ref{Descenso_global} are real-valued, then $(y_0,v_0)$ will also be real-valued. \emph{This fact shows Theorem \ref{MT3}.}

\medskip

This result will hold true because of two main ingredients: $(i)$ Propositions \ref{back_breather} and \ref{Coro4p5} combined, and $(ii)$ the uniqueness property of perturbations as consequence of the Implicit Function Theorem. These two properties will imply that all possible perturbation equals its conjugate.

\medskip

In what follows, we will work in an abstract form. Let us consider $(z_0,w_0)\in H^1(\mathbb{R})\times L^2(\mathbb{R})$, be real-valued functions, and let $(u_0,s_0,\delta)$ be the functions obtained from Lemma \ref{bbajada} starting at $(z_0,w_0)$, i.e., $(u_0,s_0,\delta)$ are such that
\begin{align}
B_x+z_{0,x}-K_t-s_0 &= \dfrac{1}{\beta+i\alpha+\delta}\sin\left(\dfrac{B+z_0+K+u_0}{2}\right) \nonu\\
& \qquad +(\beta+i\alpha+\delta)\sin\left(\dfrac{B+z_0-K-u_0}{2}\right), \label{perm_z_1}
    \\ B_t+w_0-K_x-u_{0,x} &=\dfrac{1}{\beta+i\alpha+\delta}\sin\left(\dfrac{B+z_0+K+u_0}{2}\right) \nonu \\
    & \qquad -(\beta+i\alpha+\delta)\sin\left(\dfrac{B+z_0-K-u_0}{2}\right), \label{perm_z_2}
\end{align}
for some $\delta\in\mathbb{C}$ small. Considering $\eta_0>0$ small enough such that $C\eta<\nu_0$, we have the validity of the hypotheses in Lemma \ref{kbajada} for $(u_0,s_0)$. With these in mind, we obtain $(y_0,v_0,\tilde\delta)\in H^1(\mathbb{R};\mathbb{C})\times L^2(\mathbb{R};\mathbb{C})\times \Com$ satisfying \eqref{F_new_K}, i.e.,
\begin{align}
 K_x+u_{0,x}-v_0&= \dfrac{1}{\beta-i\alpha+\tilde\delta}\sin\left(\dfrac{K+u_0+y_0}{2}\right)  \nonu\\
& \qquad + (\beta-i\alpha+\tilde\delta) \sin\left(\dfrac{K+u_0-y_0}{2}\right),\label{perm_k_bajada1}
    \\ K_t+s_0-y_{0,x}&=\dfrac{1}{\beta-i\alpha+\tilde\delta}\sin\left(\dfrac{K+u_0+y_0}{2}\right)  \nonu \\
    & \qquad  -(\beta-i\alpha+\tilde\delta)\sin\left(\dfrac{K+u_0-y_0}{2}\right),\label{perm_k_bajada2}
\end{align} 
for some small $\tilde\delta\in\mathbb{C}$.
\medskip

We want now to invert the order of the transformations. First, we apply Proposition \ref{ksubida}, starting at $(y_0,v_0)$, with fixed parameter $\beta+i\alpha+\delta$, and from Corollary \ref{back_kink_conj} we obtain $(\tilde{u}_0,\tilde{s}_0)\in H^1(\mathbb{R};\mathbb{C})\times L^2(\mathbb{R};\mathbb{C})$ satisfying \eqref{bk_subida_cerokink} (using naturally condition \eqref{New_Condition_A} applied this time to $(\overline K,\overline K_t)$).
%\begin{align*}
% \bar{K}_x+\widetilde{u}_{0,x}-v_0&=\dfrac{1}{\beta+i\alpha+p^0}\sin\left(\dfrac{\bar{K}+\widetilde{u}_0+y_0}{2}\right)+(\beta+i\alpha+p^0)\sin\left(\dfrac{\bar{K}+\widetilde{u}_0-y_0}{2}\right), 
% \\ \bar{K}_t+\widetilde{s}_0-y_{0,x}&=\dfrac{1}{\beta+i\alpha+p^0}\sin\left(\dfrac{\bar{K}+\widetilde{u}_0+y_0}{2}\right)-(\beta+i\alpha+p^0)\sin\left(\dfrac{\bar{K}+\widetilde{u}_0-y_0}{2}\right). 
%\end{align*}
Then, invoking Proposition \ref{breather_subida} starting at $(\tilde{u}_0,\tilde{s}_0)$ with transformation parameter $\beta-i\alpha+\tilde \delta$, Corollary \ref{Coro4p5} ensures the existence of functions  $(\tilde{z}_0,\tilde{w}_0)\in H^1(\mathbb{R};\mathbb{C})\times L^2(\mathbb{R};\mathbb{C})$ such that 
\begin{align}
B_x+\widetilde{z}_{0,x}-\overline{K}_t-\widetilde{s}_0 ~ = &~ \dfrac{1}{\beta-i\alpha+\tilde \delta}\sin\left(\dfrac{B+\widetilde{z}_0+\overline{K}+\widetilde{u}_0}{2}\right) \nonumber
\\ &\quad +(\beta-i\alpha+\tilde \delta)\sin\left(\dfrac{B+\widetilde{z}_0-\overline{K}-\widetilde{u}_0}{2}\right), \label{perm_ztilde_1}
    \\ 
B_t+\widetilde{w}_0-\overline{K}_x-\widetilde{u}_{0,x} ~ = &~\dfrac{1}{\beta-i\alpha+\tilde \delta}\sin\left(\dfrac{B+\widetilde{z}_0+\overline{K}+\widetilde{u}_0}{2}\right) \nonumber
    \\ & \quad -(\beta-i\alpha+\tilde \delta)\sin\left(\dfrac{B+\widetilde{z}_0-\overline{K}-\widetilde{u}_0}{2}\right). \label{perm_ztilde_2}
\end{align}

\subsection{Statement and proof} 

This being said, we are ready to announce and prove a permutability theorem.

\begin{thm}[Permutability Theorem]\label{teorema_permutabilidad}
Let $(z_0,w_0)$ and $(\widetilde{z}_0,\widetilde{w}_0)$ be the perturbacions defined by \eqref{perm_z_1}-\eqref{perm_z_2} and \eqref{perm_ztilde_1}-\eqref{perm_ztilde_2} respectively. Then, we have $(z_0,w_0)\equiv(\widetilde{z}_0,\widetilde{w}_0)$. In particular $\widetilde{z}_0$ and $\widetilde{w}_0$ are real-valued functions.
\end{thm}

\begin{rem}
The previous result can be represented by the diagram in Fig. \ref{Fig:B}.
\end{rem}

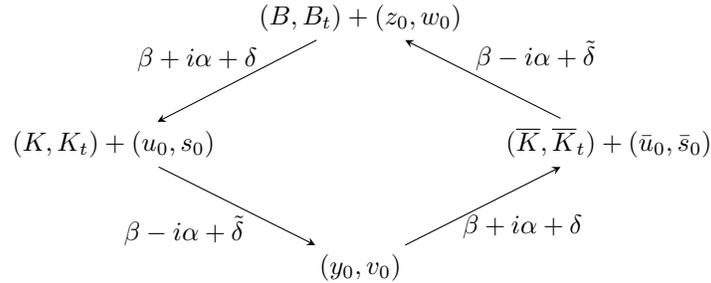
\begin{figure}[h!]
\begin{tikzpicture}
  \matrix (m) [matrix of math nodes,row sep=3em,column sep=1em,minimum width=1em]
  {
    & (B, B_t)+(z_0,w_0) & \\
     (K,K_t)+(u_0,s_0) &  & (\overline{K},\overline{K}_t)+(\bar{u}_0,\bar{s}_0) \\
    &  (y_0,v_0) & \\};
  \path[-stealth]
  
    (m-2-3) 
            edge node [above] {$\qquad\qquad \beta-i\alpha+\tilde \delta $} (m-1-2)

    (m-2-1) 
            edge node [below] {$ \beta-i\alpha+\tilde \delta \qquad \qquad $} (m-3-2)

    (m-3-2) 
            edge node [below] {$\quad \qquad \beta+i\alpha+\delta$} (m-2-3)
  
    (m-1-2) 
            edge node [above] {$ \beta+i\alpha+\delta \quad \qquad  $} (m-2-1);
\end{tikzpicture}
\caption{Theorem \ref{teorema_permutabilidad} about permutability, explained. }\label{Fig:B}
 \end{figure}

In order to prove this result, we will need the following auxiliary lemma.

\begin{lem}\label{Aux_10}
Let $(B,B_t)$ and $(K,K_t)$ be breather and kink profiles with parameters $\bt\in(-1,1)$, $\beta\neq 0$, and $x_1,x_2\in\mathbb{R}$. Let also $(\overline K,\overline{K}_t)$ be the corresponding conjugate kink profile. Then, the following relations are satisfied:  

\smallskip

\ben
\item[(i)] Difference between $K$ and its conjugate:
\be\label{KmKb}
K - \overline K  = 4\arctan\left(\dfrac{i\alpha\sin(\alpha x_1)}{\alpha\cosh(\beta(x+x_2))}\right).
\ee
\item[(ii)] The following identities are satisfied:
%\begin{align*}
%K_t=\dfrac{4i\alpha e^{\beta(x+x_2)+i\alpha x_1}}{1+e^{2\beta(x+x_2)+2i\alpha x_1}}, \quad B_t=\frac{4\alpha^2\beta\cos(\alpha x_1)\cosh(\beta(x+x_2))}{\alpha^2\cosh^2(\beta(x+x_2))+\beta^2\sin^2(\alpha x_1)},
%\end{align*}
\be\label{AAA}
\begin{aligned}
\sec^2\left(\frac{B}{4}\right)= 1+\left(\dfrac{\beta\sin(\alpha x_1)}{\alpha\cosh(\beta(x+x_2))}\right)^2,\\ \tan^2\left(\frac{B}{4}\right)=\left(\dfrac{\beta\sin(\alpha x_1)}{\alpha\cosh(\beta(x+x_2))}\right)^2 ,
\end{aligned}
\ee
and
\be\label{BBB}
\dfrac{B_t\sec^2\left(\frac{B}{4}\right)}{1+\ell^2\tan^2\left(\frac{B}{4}\right)}= \dfrac{4\alpha^2\beta\cos(\alpha x_1)\cosh(\beta(x+x_2))}{\alpha^2\cosh^2(\beta(x+x_2))+\ell^2\beta^2\sin^2(\alpha x_1)}.
\ee
\een
\end{lem}

\begin{proof}
See Appendix \ref{Lema_aux_10}.
\end{proof}

\begin{proof}[Proof of Theorem \ref{teorema_permutabilidad}]  We divide the proof in several steps.

\medskip

{\bf Step 1. Preliminaries.} For the sake of notation we define
\[
\begin{aligned}
(\phi^{0,1},\phi^{0,2}):=(y_0,v_0) ,\quad & \quad (\phi^{1,1},\phi^{1,2}):=(K+u_0,K_t+s_0) , \\
 (\phi^{2,1},\phi^{2,2}):= & ~ (\overline K+\tilde{u}_0,\overline K_t+\tilde{s}_0).
\end{aligned}
\]
Also,
\[
\begin{aligned}
\varphi^{1}= &~ (\varphi^{1,1},\varphi^{1,2}):=(B+z_0,B_t+w_0) , \\
\varphi^{2}= &~ (\varphi^{2,1},\varphi^{2,2}):=(B+\tilde{z}_0,B_t+\tilde{w}_0),
\end{aligned}
\]
and 
\[
a_1:=\beta+i\alpha+\delta, \quad a_2:=\beta-i\alpha+\tilde \delta.
\]
Finally, let $\ell$ and $\tilde{\ell}$ denote  
\be\label{ell_ell_tilde}
\ell:=\dfrac{a_1-a_2}{a_1+a_2}, \quad \tilde{\ell}:= \dfrac{a_1+a_2}{a_1-a_2}.
\ee
Note that both values $\ell$ and $\tilde \ell$ are well-defined, since $\delta,\tilde \delta$ are small. We want to prove $\varphi^1\equiv \varphi^2$. In order to prove this, let us define the auxiliary function $(\phi^{3,1},\phi^{3,2})$ via the identities 
\begin{align}\label{phi3_perm}
\phi^{3,1}-\phi^{1,1}=-4\arctan\left(\ell \tan\left(\dfrac{\varphi^{1,1}-\phi^{0,1}}{4}\right)\right),
\end{align}
and
\begin{align}\label{phi3_t_perm}
\phi^{3,2}-\phi^{1,2}=\dfrac{-\ell(\varphi^{1,2}-\phi^{0,2})\sec^2\left(\frac{\varphi^{1,1}-\phi^{0,1}}{4}\right)}{1+\ell^2\tan^2\left(\frac{\varphi^{1,1}-\phi^{0,1}}{4}\right)}.
\end{align}

\medskip

{\bf Step 2. First identities.} Note that if
\be\label{Datos_ini}
\begin{aligned}
(\phi^{0,1},\phi^{0,2})=(0,0), \quad  (\phi^{1,1},\phi^{1,2})=& ~ (K,K_t), \quad  (\varphi^{1,1},\varphi^{1,2})=(B,B_t), \\
a_1=\beta+i\alpha \quad \hbox{ y } & ~ \quad a_2=\beta-i\alpha,
\end{aligned}
\ee
then from \eqref{KmKb} we have 
\[
\phi^{3,1}=K-4\arctan\left(\dfrac{2i\alpha}{2\beta}\dfrac{\beta\sin(\alpha x_1)}{\alpha\cosh(\beta(x+x_2))}\right) = \overline K.
\]
Similarly, replacing \eqref{Datos_ini} in \eqref{phi3_t_perm}, we obtain
\begin{align}\label{phi_3_2}
\phi^{3,2}&=K_t-\dfrac{\ell B_t\sec^2\left(\frac{B}{4}\right)}{1+\ell^2\tan^2\left(\frac{B}{4}\right)}.
\end{align}
Therefore, using \eqref{AAA} and \eqref{BBB}, we obtain that \eqref{phi_3_2} is reduced to simplifying the RHS of the identity
\[
\phi^{3,2}  =  \dfrac{4i\alpha e^{\beta(x+x_2)+i\alpha x_1}}{1+e^{2\beta(x+x_2)+2i\alpha x_1}}-\dfrac{4\ell\alpha^2\beta\cos(\alpha x_1)\cosh(\beta(x+x_2))}{\alpha^2\cosh^2(\beta(x+x_2))+\ell^2\beta^2\sin^2(\alpha x_1)}.
\]
Let us consider the notation
\be\label{thetas}
\theta_1:=\alpha x_1, \quad \theta_2:=\beta(x+x_2), \quad  \theta:=\beta(x+x_2)+i\alpha x_1. 
\ee
We have,
\begin{align}
& \phi^{3,2}=\nonu \\
& = \dfrac{4i\alpha^3 e^{\theta}(\cosh^2(\theta_2)-\sin^2(\theta_1))-4i\alpha^3(1+e^{2\theta})\cos(\theta_1)\cosh(\theta_2)}{(1+e^{2\theta})(\alpha^2\cosh^2(\theta_2)-\alpha^2\sin^2(\theta_1))}\nonumber
\\ & = \dfrac{4i\alpha e^{\theta}(\cosh^2(\theta_2)-\sin^2(\theta_1))-4i\alpha(1+e^{2\theta})\cos(\theta_1)\cosh(\theta_2)}{(1+e^{2\theta})(\cosh^2(\theta_2)-\sin^2(\theta_1))}\nonumber
\\ & = \dfrac{i\alpha\big(e^{\theta+2\theta_2}+e^{\theta-2\theta_2}+e^{\theta+2i\theta_1}+e^{\theta-2i\theta_1}-(1+e^{2\theta})(e^{i\theta_1}+e^{-i\theta_1})(e^{\theta_2}+e^{-\theta_2})\big)}{\big(1+e^{2\theta}\big)\big(\cosh^2(\theta_2)-\sin^2(\theta_1)\big)}\nonumber
\\ & = \dfrac{-i\alpha\big(e^{3\theta}+e^{-\theta}+2e^{\theta}\big)}{\big(1+e^{2\theta}\big)\big(\cosh^2(\theta_2)-\sin^2(\theta_1)\big)}
= \dfrac{-i\alpha e^{-\theta}\big(1+e^{2\theta}\big)^2}{\big(1+e^{2\theta}\big)\big(\cosh^2(\theta_2)-\sin^2(\theta_1)\big)}\nonumber
\\ & = \dfrac{-i\alpha e^{-\theta}\big(1+e^{2\theta}\big)}{\cosh^2(\theta_2)-\sin^2(\theta_1)}
= \dfrac{-4i\alpha e^{-\theta}\big(1+e^{2\theta}\big)}{\big(1+e^{2\theta}\big)\big(e^{-2i\theta_1}+e^{-2\theta_2}\big)} \nonumber
\\ &= \dfrac{-4i\alpha e^{-\theta}}{e^{-2i\theta_1}+e^{-2\theta_2}}
 = \dfrac{-4i\alpha e^{\theta_2-i\theta_1}}{1+e^{2(\theta_2-i\theta_1)}} \ = \ \overline{K}_t. \nonu
\end{align}
Then, if \eqref{Datos_ini} holds, necessarily
\be\label{permutabilidad_k}
\phi^3 =(\phi^{3,1},\phi^{3,2}) = (\overline{K},\overline{K}_t).
\ee

\medskip

{\bf Step 3. ODEs satisfied by $\phi^3$.} Let us consider now general values of $\phi^0$, $\phi^1$, $\varphi^1$ and $a_1,a_2$, as before. We shall prove that $\phi^3=(\phi^{3,1},\phi^{3,2})$ defined in \eqref{phi3_perm}-\eqref{phi3_t_perm} satisfy the identities
\begin{align}
\phi^{3,1}_x-\phi^{0,2}&=\dfrac{1}{a_1}\sin\left(\dfrac{\phi^{3,1}+\phi^{0,1}}{2}\right)+a_1\sin\left(\dfrac{\phi^{3,1}-\phi^{0,1}}{2}\right),\label{permutabilidad_back_phi3_1}
\\ \phi^{3,2}-\phi^{0,1}_x&=\dfrac{1}{a_1}\sin\left(\dfrac{\phi^{3,1}+\phi^{0,1}}{2}\right)-a_1\sin\left(\dfrac{\phi^{3,1}-\phi^{0,1}}{2}\right).\label{permutabilidad_back_phi3_2}
\end{align}
Hence, from \eqref{permutabilidad_k} we conclude that $(\phi^{3,1},\phi^{3,2})\equiv (\phi^{2,1},\phi^{2,2})$. Similarly, denoting $\phi^4 :=(\phi^{4,1},\phi^{4,2})$ the solution to 
\begin{align*}
\phi^{2,1}-\phi^{4,1}&=-4\arctan\left(\dfrac{a_1+a_2}{a_1-a_2}\tan\left(\dfrac{\varphi^{2,1}-\phi^{0,1}}{4}\right)\right),
\\ \phi^{2,2}-\phi^{4,2}&=-\dfrac{\tilde{\ell} \big(\varphi^{2,2}-\phi^{0,2}\big)\sec^2\left(\frac{\varphi^{2,1}-\phi^{0,1}}{4}\right)}{1+\tilde{\ell}^{2}\tan^2\left(\frac{\varphi^{2,1}-\phi^{0,1}}{4}\right)},
\end{align*}
and proving that $ (\phi^{4,1},\phi^{4,2})$ satisfy 
\begin{align*}
\phi^4_x-\phi^0_t&=\dfrac{1}{a_1}\sin\left(\dfrac{\phi^4+\phi^0}{2}\right)+a_1\sin\left(\dfrac{\phi^4-\phi^0}{2}\right),
\\ \phi^4_t-\phi^0_x&=\dfrac{1}{a_1}\sin\left(\dfrac{\phi^4+\phi^0}{2}\right)-a_1\sin\left(\dfrac{\phi^4-\phi^0}{2}\right),
\end{align*}
then we have $(\phi^{4,1},\phi^{4,2})\equiv(\phi^{1,1},\phi^{1,2})$. From here we conclude that  $(\varphi^{1,1},\varphi^{1,2})\equiv(\varphi^{2,1},\varphi^{2,2})$. Moreover,  
\begin{align}\label{tan_tan_permutabilidad}
\tan\left(\dfrac{\varphi^1-\phi^0}{4}\right)=-\dfrac{a_1+a_2}{a_1-a_2}\tan\left(\dfrac{\phi^2-\phi^1}{4}\right).
\end{align}
This identity will be used a posteriori. Let us now show \eqref{permutabilidad_back_phi3_1} and \eqref{permutabilidad_back_phi3_2}.

\medskip

{\bf Step 4. Proof of \eqref{permutabilidad_back_phi3_1}.} In fact, from \eqref{phi3_perm} we have
\begin{align}\label{ec_varphi1}
\varphi^{1,1}-\phi^{0,1}=-4\arctan\left(\ell^{-1}\tan\left(\dfrac{\phi^{3,1}-\phi^{1,1}}{4}\right)\right).
\end{align}
Then, taking derivative wrt $x$,  
\begin{align}\label{ec_varphi2}
\varphi^{1,1}_x-\phi^{0,1}_x=\dfrac{-\ell^{-1}(\phi^{3,1}_x-\phi^{1,1}_x)\sec^2\left(\frac{\phi^{3,1}-\phi^{1,1}}{4}\right)}{1+\ell^{-2}\tan^2\left(\frac{\phi^{3,1}-\phi^{1,1}}{4}\right)},
\end{align}
or
\begin{align}
-\frac1{\ell}(\phi^{3,1}_x-\phi^{1,1}_x)\sec^2\left(\frac{\phi^{3,1}-\phi^{1,1}}{4}\right)  = \left(1+\frac1{\ell^2}\tan\left(\frac{\phi^{3,1}-\phi^{1,1}}{4}\right)\right)\big(\varphi^{1,1}_x-\phi^{0,1}_x\big).\label{rhs_permutabilidad}
\end{align}
On the other hand, from \eqref{ec_varphi1} it is not difficult to show that  
\be\label{sin_permutabilidad}
\begin{aligned}
\sin\left(\dfrac{\varphi^{1,1}-\phi^{0,1}}{2}\right)=& ~ \dfrac{-2\ell^{-1}\tan\left(\frac{\phi^{3,1}-\phi^{1,1}}{4}\right)}{1+\ell^{-2}\tan^2\left(\frac{\phi^{3,1}-\phi^{1,1}}{4}\right)},\\
 \cos\left(\dfrac{\varphi^{1,1}-\phi^{0,1}}{2}\right)= & ~ \dfrac{1-\ell^{-2}\tan^2\left(\frac{\phi^{3,1}-\phi^{1,1}}{4}\right)}{1+\ell^{-2}\tan^2\left(\frac{\phi^{3,1}-\phi^{1,1}}{4}\right)}.
\end{aligned}
%\label{cos_permutabilidad}
\ee
Since from Proposition \ref{Descenso_global} we have the connections
\[
\mathbb{B}(\phi^{0,1},\phi^{0,2})\xrightarrow{\,a_2\,}{(\phi^{1,1},\phi^{1,2})}, \qquad \mathbb{B}(\phi^{1,1},\phi^{1,2})\xrightarrow{\,a_1\,}{(\varphi^{1,1},\varphi^{1,2})},
\]
which in particular imply
\[
\begin{aligned}
\varphi^{1,1}_x-\phi^{1,2}= &~{} \dfrac{1}{a_1}\sin\left(\dfrac{\varphi^{1,1}+\phi^{1,1}}{2}\right)+a_1\sin\left(\dfrac{\varphi^{1,1}-\phi^{1,1}}{2}\right)\\
\phi^{1,2}-\phi^{0,1}_x = &~ {} \dfrac{1}{a_2}\sin\left(\dfrac{\phi^{1,1}+\phi^{0,1}}{2}\right)-a_2\sin\left(\dfrac{\phi^{1,1}-\phi^{0,1}}{2}\right),
\end{aligned} 
\]
we can rewrite the LHS of \eqref{ec_varphi2} as follows:
\begin{align}
& \varphi^{1,1}_x-\phi^{0,1}_x \nonu \\
& \quad = \varphi^{1,1}_x-\phi^{1,2}+\phi^{1,2}-\phi^{0,1}_x \nonumber\\ 
& \quad = \dfrac{1}{a_1}\sin\left(\dfrac{\varphi^{1,1}+\phi^{1,1}}{2}\right)+a_1\sin\left(\dfrac{\varphi^{1,1}-\phi^{1,1}}{2}\right) \nonu\\
& \qquad +\dfrac{1}{a_2}\sin\left(\dfrac{\phi^{1,1}+\phi^{0,1}}{2}\right)-a_2\sin\left(\dfrac{\phi^{1,1}-\phi^{0,1}}{2}\right)\nonumber
\\ 
& \quad = \dfrac{1}{a_1}\sin\left(\dfrac{\varphi^{1,1}-\phi^{0,1}+\phi^{0,1}+\phi^{1,1}}{2}\right)+a_1\sin\left(\dfrac{\varphi^{1,1}-\phi^{0,1}+\phi^{0,1}-\phi^{1,1}}{2}\right)\nonumber
\\ & \qquad +\dfrac{1}{a_2}\sin\left(\dfrac{\phi^{1,1}+\phi^{0,1}}{2}\right)-a_2\sin\left(\dfrac{\phi^{1,1}-\phi^{0,1}}{2}\right).\nonumber
\end{align}
Expanding terms,
\begin{align}
 \varphi^{1,1}_x-\phi^{0,1}_x  =&~ \dfrac{1}{a_1}\sin\left(\dfrac{\varphi^{1,1}-\phi^{0,1}}{2}\right)\cos\left(\dfrac{\phi^{0,1}+\phi^{1,1}}{2}\right) \nonu\\
 & ~ {} +\dfrac{1}{a_1}\cos\left(\dfrac{\varphi^{1,1}-\phi^{0,1}}{2}\right)\sin\left(\dfrac{\phi^{0,1}+\phi^{1,1}}{2}\right)\nonumber
\\
 & ~ {} +a_1\sin\left(\dfrac{\varphi^{1,1}-\phi^{0,1}}{2}\right)\cos\left(\dfrac{\phi^{1,1}-\phi^{0,1}}{2}\right)\nonu \\
 &~ {}  -a_1\cos\left(\dfrac{\varphi^{1,1}-\phi^{0,1}}{2}\right)\sin\left(\dfrac{\phi^{1,1}-\phi^{0,1}}{2}\right) \nonumber
\\ 
& ~ {} +\dfrac{1}{a_2}\sin\left(\dfrac{\phi^{1,1}+\phi^{0,1}}{2}\right)-a_2\sin\left(\dfrac{\phi^{1,1}-\phi^{0,1}}{2}\right).\nonumber
\end{align}
Replacing this last identity in the RHS of \eqref{rhs_permutabilidad}, and using the identities found in \eqref{sin_permutabilidad}, we have 
\begin{align}
& -\frac{1}{\ell}(\phi^{3,1}_x-\phi^{1,1}_x)\sec^2\left(\dfrac{\phi^{3,1}-\phi^{1,1}}{4}\right) \nonumber
\\  
& =\left(1+\frac{1}{\ell^{2}}\tan^2\left(\dfrac{\phi^{3,1}-\phi^{1,1}}{4}\right)\right)\left(\dfrac{1}{a_2}\sin\left(\dfrac{\phi^{1,1}+\phi^{0,1}}{2}\right)-a_2\sin\left(\dfrac{\phi^{1,1}-\phi^{0,1}}{2}\right)\right)\nonumber
\\ 
& \; +\left(1-\frac{1}{\ell^{2}}\tan^2\left(\dfrac{\phi^{3,1}-\phi^{1,1}}{4}\right)\right)\left(\dfrac{1}{a_1}\sin\left(\dfrac{\phi^{1,1}+\phi^{0,1}}{2}\right)-a_1\sin\left(\dfrac{\phi^{1,1}-\phi^{0,1}}{2}\right)\right)\nonumber
\\
 & \; - \frac{2}{\ell} \tan\left(\dfrac{\phi^{3,1}-\phi^{1,1}}{4}\right)\left(\dfrac{1}{a_1}\cos\left(\dfrac{\phi^{1,1}+\phi^{0,1}}{2}\right)+a_1\cos\left(\dfrac{\phi^{1,1}-\phi^{0,1}}{2}\right)\right).\label{ec_final1_permutabilidad}
\end{align}
Then, using that the LHS of \eqref{ec_varphi2} can be rewritten as 
\[
\phi^{3,1}_x-\phi^{1,1}_x=\phi^{3,1}_x-\phi^{0,2}+\phi^{0,2}-\phi^{1,1}_x,
\]
recalling that $\mathbb{B}(\phi^{0,1},\phi^{0,2})\xrightarrow{\, a_2\,}{(\phi^{1,1},\phi^{1,2})}$, i.e.,
\[
\begin{aligned}
\phi^{1,1}_x-\phi^{0,2}= &~{} \dfrac{1}{a_2}\sin\left(\dfrac{\phi^{1,1}+\phi^{0,1}}{2}\right)+a_2\sin\left(\dfrac{\phi^{1,1}-\phi^{0,1}}{2}\right)\\
\phi^{1,2}-\phi^{0,1}_x = &~ {} \dfrac{1}{a_2}\sin\left(\dfrac{\phi^{1,1}+\phi^{0,1}}{2}\right)-a_2\sin\left(\dfrac{\phi^{1,1}-\phi^{0,1}}{2}\right),
\end{aligned} 
\]
we can replace \eqref{ec_final1_permutabilidad} in \eqref{ec_varphi2} to get 
\begin{align}
 \phi^{3,1}_x-\phi^{0,2} & = \phi^{3,1}_x-\phi^{1,1}_x + (\phi^{1,1}_x -\phi^{0,2}) \nonu\\
&=  - \cos^2\left(\dfrac{\phi^{3,1}-\phi^{1,1}}{4}\right) \left(\ell+\ell^{-1}\tan^2\left(\dfrac{\phi^{3,1}-\phi^{1,1}}{4}\right)\right) \nonu\\
& \qquad \left(\dfrac{1}{a_2}\sin\left(\dfrac{\phi^{1,1}+\phi^{0,1}}{2}\right)-a_2\sin\left(\dfrac{\phi^{1,1}-\phi^{0,1}}{2}\right)\right)\nonumber
\\ 
& \quad- \cos^2\left(\dfrac{\phi^{3,1}-\phi^{1,1}}{4}\right) \left(\ell-\ell^{-1}\tan^2\left(\dfrac{\phi^{3,1}-\phi^{1,1}}{4}\right)\right)  \nonu\\
& \qquad \left(\dfrac{1}{a_1}\sin\left(\dfrac{\phi^{1,1}+\phi^{0,1}}{2}\right)-a_1\sin\left(\dfrac{\phi^{1,1}-\phi^{0,1}}{2}\right)\right)\nonumber
\\ 
& \quad + 2\sin\left(\dfrac{\phi^{3,1}-\phi^{1,1}}{4}\right)\cos\left(\dfrac{\phi^{3,1}-\phi^{1,1}}{4}\right)  \nonu\\
& \qquad \left(\dfrac{1}{a_1}\cos\left(\dfrac{\phi^{1,1}+\phi^{0,1}}{2}\right)+a_1\cos\left(\dfrac{\phi^{1,1}-\phi^{0,1}}{2}\right)\right)   \nonu \\
& \quad + \dfrac{1}{a_2}\sin\left(\dfrac{\phi^{1,1}+\phi^{0,1}}{2}\right)+a_2\sin\left(\dfrac{\phi^{1,1}-\phi^{0,1}}{2}\right).\nonu
\end{align}
A further simplification gives
\begin{align}
 \phi^{3,1}_x-\phi^{0,2} & =\dfrac{1}{a_2}\sin\left(\dfrac{\phi^{1,1}+\phi^{0,1}}{2}\right)+a_2\sin\left(\dfrac{\phi^{1,1}-\phi^{0,1}}{2}\right)\nonumber
\\ 
& \quad -\cos^2\left(\dfrac{\phi^{3,1}-\phi^{1,1}}{4}\right) \sin\left(\dfrac{\phi^{1,1}+\phi^{0,1}}{2}\right) \nonu\\
& \qquad \bigg[\left(\dfrac{1}{a_2}+\dfrac{1}{a_1}\right)\ell + \ell^{-1}\left(\dfrac{1}{a_2}-\dfrac{1}{a_1}\right)\tan^2\left(\dfrac{\phi^{3,1}-\phi^{1,1}}{4}\right)\bigg] \nonumber
 \\ 
 & \quad  +\cos^2\left(\dfrac{\phi^{3,1}-\phi^{1,1}}{4}\right) \sin\left(\dfrac{\phi^{1,1}- \phi^{0,1}}{2}\right) \nonu\\
& \qquad  \bigg[   (a_1 +a_2)\ell +  (a_2-a_1)\ell^{-1}\tan^2\left(\dfrac{\phi^{3,1}-\phi^{1,1}}{4}\right) \bigg] \nonumber
\\ 
& \quad + \sin\left(\dfrac{\phi^{3,1}-\phi^{1,1}}{2}\right)\left(\dfrac{1}{a_1}\cos\left(\dfrac{\phi^{1,1}+\phi^{0,1}}{2}\right)+a_1\cos\left(\dfrac{\phi^{1,1}-\phi^{0,1}}{2}\right)\right).   \nonu 
\end{align}
Thanks to \eqref{ell_ell_tilde}, we have
\begin{align}
&\phi^{3,1}_x-\phi^{0,2} \nonu\\
&=\dfrac{1}{a_2}\sin\left(\dfrac{\phi^{1,1}+\phi^{0,1}}{2}\right)+a_2\sin\left(\dfrac{\phi^{1,1}-\phi^{0,1}}{2}\right)\nonumber
\\ 
& \quad - \sin\left(\dfrac{\phi^{1,1}+\phi^{0,1}}{2}\right) \nonu\\
& \qquad \bigg[\left(\dfrac{1}{a_2}- \dfrac{1}{a_1}\right)\cos^2\left(\dfrac{\phi^{3,1}-\phi^{1,1}}{4}\right)  + \left(\dfrac{1}{a_2}+\dfrac{1}{a_1}\right)\sin^2\left(\dfrac{\phi^{3,1}-\phi^{1,1}}{4}\right)\bigg] \nonumber
 \\
  & \quad  + \sin\left(\dfrac{\phi^{1,1}- \phi^{0,1}}{2}\right)\nonu\\
& \qquad   \bigg[   (a_1-a_2) \cos^2\left(\dfrac{\phi^{3,1}-\phi^{1,1}}{4}\right) -  (a_1+a_2)\sin^2\left(\dfrac{\phi^{3,1}-\phi^{1,1}}{4}\right) \bigg] \nonumber
\\
 & \quad + \sin\left(\dfrac{\phi^{3,1}-\phi^{1,1}}{2}\right)\left(\dfrac{1}{a_1}\cos\left(\dfrac{\phi^{1,1}+\phi^{0,1}}{2}\right)+a_1\cos\left(\dfrac{\phi^{1,1}-\phi^{0,1}}{2}\right)\right).   \nonu 
\end{align}
Simplifying,
\begin{align}
\phi^{3,1}_x-\phi^{0,2} &= - \dfrac{1}{a_1} \sin\left(\dfrac{\phi^{1,1}+\phi^{0,1}}{2}\right) \bigg[  \sin^2\left(\dfrac{\phi^{3,1}-\phi^{1,1}}{4}\right)-\cos^2\left(\dfrac{\phi^{3,1}-\phi^{1,1}}{4}\right) \bigg] \nonumber
 \\
 & \quad  + a_1 \sin\left(\dfrac{\phi^{1,1}- \phi^{0,1}}{2}\right) \bigg[    \cos^2\left(\dfrac{\phi^{3,1}-\phi^{1,1}}{4}\right) - \sin^2\left(\dfrac{\phi^{3,1}-\phi^{1,1}}{4}\right) \bigg] \nonumber
\\
 & \quad + \sin\left(\dfrac{\phi^{3,1}-\phi^{1,1}}{2}\right)\left(\dfrac{1}{a_1}\cos\left(\dfrac{\phi^{1,1}+\phi^{0,1}}{2}\right)+a_1\cos\left(\dfrac{\phi^{1,1}-\phi^{0,1}}{2}\right)\right)  \nonu \\
&=  \dfrac{1}{a_1} \sin\left(\dfrac{\phi^{1,1}+\phi^{0,1}}{2}\right)\cos\left(\dfrac{\phi^{3,1}-\phi^{1,1}}{2}\right) \nonu \\
& \quad  + a_1 \sin\left(\dfrac{\phi^{1,1}- \phi^{0,1}}{2}\right)   \cos\left(\dfrac{\phi^{3,1}-\phi^{1,1}}{2}\right) \nonu\\
  & \quad + \sin\left(\dfrac{\phi^{3,1}-\phi^{1,1}}{2}\right)\left(\dfrac{1}{a_1}\cos\left(\dfrac{\phi^{1,1}+\phi^{0,1}}{2}\right)+a_1\cos\left(\dfrac{\phi^{1,1}-\phi^{0,1}}{2}\right)\right)  \nonu .
\end{align}
Finally,
\[
\phi^{3,1}_x-\phi^{0,2} =  \dfrac{1}{a_1} \sin\left(\dfrac{\phi^{3,1}+\phi^{0,1}}{2}\right)  + a_1 \sin\left(\dfrac{\phi^{3,1}- \phi^{0,1}}{2}\right) .
\]
%
%\begin{align}
%\phi^{3,1}_x-\phi^{0,2}  \nonu
%\\ & =\dfrac{1}{a_2}\sin\left(\dfrac{\phi^{1,1}+\phi^{0,1}}{2}\right)+a_2\sin\left(\dfrac{\phi^{1,1}-\phi^{0,1}}{2}\right)\nonumber
%\\ & \quad -\left(\dfrac{1}{a_2}-\dfrac{1}{a_1}\cos\left(\dfrac{\phi^{3,1}-\phi^{1,1}}{2}\right)\right)\sin\left(\dfrac{\phi^{1,1}+\phi^{0,1}}{2}\right)\nonumber
%\\ & \quad -\left(a_2-a_1\cos\left(\dfrac{\phi^{3,1}-\phi^{1,1}}{2}\right)\right)\sin\left(\dfrac{\phi^{1,1}-\phi^{0,1}}{2}\right)\nonumber
%\\ & \quad +\sin\left(\dfrac{\phi^{3,1}-\phi^{1,1}}{2}\right)\left(\dfrac{1}{a_1}\cos\left(\dfrac{\phi^{1,1}+\phi^{0,1}}{2}\right)+a_1\cos\left(\dfrac{\phi^{1,1}-\phi^{0,1}}{2}\right)\right)\nonumber
%\\ & = \dfrac{1}{a_1}\cos\left(\dfrac{\phi^{3,1}-\phi^{1,1}}{2}\right)\sin\left(\dfrac{\phi^{1,1}+\phi^{0,1}}{2}\right)+\dfrac{1}{a_1}\sin\left(\dfrac{\phi^{3,1}-\phi^{1,1}}{2}\right)\cos\left(\dfrac{\phi^{1,1}+\phi^{0,1}}{2}\right) \nonumber
%\\ & \quad +a_1\cos\left(\dfrac{\phi^{3,1}-\phi^{1,1}}{2}\right)\sin\left(\dfrac{\phi^{1,1}-\phi^{0,1}}{2}\right)+a_1\sin\left(\dfrac{\phi^{3,1}-\phi^{1,1}}{2}\right)\cos\left(\dfrac{\phi^{1,1}-\phi^{0,1}}{2}\right)\nonumber
%\\ & = \dfrac{1}{a_1}\sin\left(\dfrac{\phi^{3,1}+\phi^{0,1}}{2}\right)+a_1\sin\left(\dfrac{\phi^{3,1}-\phi^{0,1}}{2}\right), \nonumber
%\end{align}
This ends the proof of the case \eqref{permutabilidad_back_phi3_1}.

\medskip

{\bf Step 4. Proof of \eqref{permutabilidad_back_phi3_2}.} We proceed as before. First, we write the LHS of \eqref{phi3_t_perm} as follows:
 \[
\phi^{3,2}-\phi^{1,2}=\phi^{3,2}-\phi^{0,1}_x+\phi^{0,1}_x-\phi^{1,2}.
\]
%Luego, \eqref{phi3_t_perm} puede reescribirse como \[
%\phi^{3,2}-\phi^{0,1}_x= \phi^{1,2}-\phi^{0,1}_x+\dfrac{-\ell(\varphi^{1,2}-\phi^{0,2})\sec^2\left(\frac{\varphi^{1,1}-\phi^{0,1}}{4}\right)}{1+\ell^2\tan^2\left(\frac{\varphi^{1,1}-\phi^{0,1}}{4}\right)}.
%\] 
%Notando que \[
%\dfrac{1}{1+\ell^2\tan^2\left(\frac{\varphi^{1,1}-\phi^{0,1}}{4}\right)}=\cos^2\left(\dfrac{\phi^{3,1}-\phi^{1,1}}{4}\right),
%\]
%y que \[
%\sec^2\left(\dfrac{\varphi^{1,1}-\phi^{0,1}}{4}\right)=1+\ell^{-2}\tan^2\left(\dfrac{\phi^{3,1}-\phi^{1,1}}{4}\right)
%\] 
Similarly, we have $\varphi^{1,2}-\phi^{0,2}=\varphi^{1,2}-\phi^{1,1}_x+\phi^{1,1}_x-\phi^{0,2}$. Thanks to \eqref{phi3_perm}, we have that \eqref{phi3_t_perm} reads now 
\begin{align}
\phi^{3,2}-\phi^{0,1}_x &=\phi^{1,2}-\phi^{0,1}_x \nonu
\\ & \ \ -\ell\big(\varphi^{1,2}-\phi^{1,1}_x+\phi^{1,1}_x-\phi^{0,2}\big)\left(1+\ell^{-2}\tan^2\left(\dfrac{\phi^{3,1}-\phi^{1,1}}{4}\right)\right)\nonu\\
& \qquad \cos^2\left(\dfrac{\phi^{3,1}-\phi^{1,1}}{4}\right). \label{rhs_permutabilidad_2}
\end{align}
On the other hand, recall that 
\begin{align}
\phi^{1,1}_x-\phi^{0,2}=\dfrac{1}{a_2}\sin\left(\dfrac{\phi^{1,1}+\phi^{0,1}}{2}\right)+a_2\sin\left(\dfrac{\phi^{1,1}-\phi^{0,1}}{2}\right).\label{conexion1_permutabilidad}
\end{align}
Similarly, we have 
\begin{align*}
& \varphi^{1,2}-\phi^{1,1}_x\\
&=\dfrac{1}{a_1}\sin\left(\dfrac{\varphi^{1,1}+\phi^{1,1}}{2}\right)-a_1\sin\left(\dfrac{\varphi^{1,1}-\phi^{1,1}}{2}\right)
\\ & =\dfrac{1}{a_1}\sin\left(\dfrac{\varphi^{1,1}-\phi^{0,1}+\phi^{0,1}+\phi^{1,1}}{2}\right)-a_1\sin\left(\dfrac{\varphi^{1,1}-\phi^{0,1}+\phi^{0,1}-\phi^{1,1}}{2}\right)
\\ &=\dfrac{1}{a_1}\left(\sin\left(\dfrac{\varphi^{1,1}-\phi^{0,1}}{2}\right)\cos\left(\dfrac{\phi^{1,1}+\phi^{0,1}}{2}\right)+\cos\left(\dfrac{\varphi^{1,1}-\phi^{0,1}}{2}\right)\sin\left(\dfrac{\phi^{1,1}+\phi^{0,1}}{2}\right)\right)
\\ & \quad -a_1\left(\sin\left(\dfrac{\varphi^{1,1}-\phi^{0,1}}{2}\right)\cos\left(\dfrac{\phi^{1,1}-\phi^{0,1}}{2}\right)-\cos\left(\dfrac{\varphi^{1,1}-\phi^{0,1}}{2}\right)\sin\left(\dfrac{\phi^{1,1}-\phi^{0,1}}{2}\right)\right).
\end{align*}
Therefore, \eqref{sin_permutabilidad} implies
\begin{align}
& \varphi^{1,2}-\phi^{1,1}_x \nonu\\
&= \dfrac{-2\ell^{-1}\tan\left(\frac{\phi^{3,1}-\phi^{1,1}}{4}\right)}{1+\ell^{-2}\tan^2\left(\frac{\phi^{3,1}-\phi^{1,1}}{4}\right)}\left(\dfrac{1}{a_1}\cos\left(\dfrac{\phi^{1,1}+\phi^{0,1}}{2}\right)-a_1\cos\left(\dfrac{\phi^{1,1}-\phi^{0,1}}{2}\right)\right)\nonumber
\\ & \quad + \dfrac{1-\ell^{-2}\tan^2\left(\frac{\phi^{3,1}-\phi^{1,1}}{4}\right)}{1+\ell^{-2}\tan^2\left(\frac{\phi^{3,1}-\phi^{1,1}}{4}\right)} \left(\dfrac{1}{a_1}\sin\left(\dfrac{\phi^{1,1}+\phi^{0,1}}{2}\right)-a_1\sin\left(\dfrac{\phi^{1,1}-\phi^{0,1}}{2}\right)\right)\label{conexion2_permutabilidad}
\end{align}
Therefore, replacing \eqref{conexion1_permutabilidad} and \eqref{conexion2_permutabilidad} in \eqref{rhs_permutabilidad_2} we get 
\begin{align*}
& \phi^{3,2}-\phi^{0,1}_x\\
&=\dfrac{1}{a_2}\sin\left(\dfrac{\phi^{1,1}+\phi^{0,1}}{2}\right)-a_2\sin\left(\dfrac{\phi^{1,1}-\phi^{0,1}}{2}\right)
\\ 
& \quad +\sin\left(\frac{\phi^{3,1}-\phi^{1,1}}{2}\right)\left(\dfrac{1}{a_1}\cos\left(\dfrac{\phi^{1,1}+\phi^{0,1}}{2}\right)-a_1\cos\left(\dfrac{\phi^{1,1}-\phi^{0,1}}{2}\right)\right)
\\ 
& \quad -\left(\ell\cos^2\left(\frac{\phi^{3,1}-\phi^{1,1}}{4}\right)-\ell^{-1}\sin^2\left(\frac{\phi^{3,1}-\phi^{1,1}}{4}\right)\right) \nonu\\
& \qquad \quad \left(\dfrac{1}{a_1}\sin\left(\dfrac{\phi^{1,1}+\phi^{0,1}}{2}\right)-a_1\sin\left(\dfrac{\phi^{1,1}-\phi^{0,1}}{2}\right)\right)
\\ 
& \quad -\ell \left(\dfrac{1}{a_2}\sin\left(\dfrac{\phi^{1,1}+\phi^{0,1}}{2}\right)+a_2\sin\left(\dfrac{\phi^{1,1}-\phi^{0,1}}{2}\right)\right)   \nonu\\
& \qquad \quad \left(1+\ell^{-2}\tan^2\left(\dfrac{\phi^{3,1}-\phi^{1,1}}{4}\right)\right)\cos^2\left(\dfrac{\phi^{3,1}-\phi^{1,1}}{4}\right).
\end{align*}
Finally, gathering terms and using the value of $\ell$ we obtain
\begin{align*}
&\phi^{3,2}-\phi^{0,1}_x\\
&=\dfrac{1}{a_2}\sin\left(\dfrac{\phi^{1,1}+\phi^{0,1}}{2}\right)-a_2\sin\left(\dfrac{\phi^{1,1}-\phi^{0,1}}{2}\right)
\\ 
& \quad -\dfrac{1}{a_1^2-a_2^2}\left((a_1^2+a_2^2)\cos\left(\dfrac{\phi^{3,1}-\phi^{1,1}}{2}\right)-2a_1a_2\right) \\
& \qquad \quad \left(\dfrac{1}{a_1}\sin\left(\dfrac{\phi^{1,1}+\phi^{0,1}}{2}\right)+a_1\sin\left(\dfrac{\phi^{1,1}-\phi^{0,1}}{2}\right)\right)
\\ 
& \quad -\dfrac{1}{a_1^2-a_2^2}\left(a_1^2+a_2^2-2a_1a_2\cos\left(\dfrac{\phi^{3,1}-\phi^{1,1}}{2}\right)\right)   \\
& \qquad \quad \left(\dfrac{1}{a_2}\sin\left(\dfrac{\phi^{1,1}+\phi^{0,1}}{2}\right)+a_2\sin\left(\dfrac{\phi^{1,1}-\phi^{0,1}}{2}\right)\right)
\\ 
& = \dfrac{1}{a_2}\sin\left(\dfrac{\phi^{1,1}+\phi^{0,1}}{2}\right)-a_2\sin\left(\dfrac{\phi^{1,1}-\phi^{0,1}}{2}\right) +\dfrac{1}{a_1}\sin\left(\dfrac{\phi^{3,1}+\phi^{0,1}}{2}\right) 
\\ 
& \quad -a_1\sin\left(\dfrac{\phi^{3,1}-\phi^{0,1}}{2}\right) -\dfrac{1}{a_2}\sin\left(\dfrac{\phi^{1,1}-\phi^{0,1}}{2}\right)+a_2\sin\left(\dfrac{\phi^{1,1}-\phi^{0,1}}{2}\right)
\\ 
& = \dfrac{1}{a_1}\sin\left(\dfrac{\phi^{3,1}+\phi^{0,1}}{2}\right)-a_1\sin\left(\dfrac{\phi^{3,1}-\phi^{0,1}}{2}\right),
\end{align*}
which finally proves \eqref{permutabilidad_back_phi3_2}. 
\end{proof}

\begin{cor}\label{Coro_permutabilidad}
Under the assumptions of Theorem \ref{teorema_permutabilidad} we have 
\[
(u_0,s_0)=(\overline{\tilde{u}}_0,\overline{\tilde{s}}_0), \quad \delta=\overline{\tilde \delta}.
\]
\end{cor}

\begin{proof}
Theorem \ref{teorema_permutabilidad} implies $(z_0,w_0)\equiv(\tilde{z}_0,\tilde{w}_0)$. Then, after conjugation of \eqref{perm_ztilde_1} and \eqref{perm_ztilde_2} we have
\begin{align*}
B_x+z_{0,x}-K_t-\bar{\tilde{s}}_0 =&~ \dfrac{1}{\beta+i\alpha+\overline{\tilde\delta}}\sin\left(\dfrac{B+z_0+K+\bar{\tilde{u}}_0}{2}\right) \nonumber
\\ &\quad +(\beta+i\alpha+\overline{\tilde \delta})\sin\left(\dfrac{B+z_0-K-\bar{\tilde{u}}_0}{2}\right),
    \\ B_t+w_0-K_x-\bar{\tilde{u}}_{0,x} =&~\dfrac{1}{\beta+i\alpha+\overline{\tilde\delta}}\sin\left(\dfrac{B+z_0+K+\bar{\tilde{u}}_0}{2}\right) \nonumber
    \\ & \quad -(\beta+i\alpha+\overline{\tilde\delta})\sin\left(\dfrac{B+z_0-K-\bar{\tilde{u}}_0}{2}\right).
\end{align*}
Therefore, thanks to the uniqueness of perturbations (via Implicit Function Theorem), and using \eqref{perm_z_1} and \eqref{perm_z_2}, we conclude the result.
\end{proof}

The following result will be essential in the rest of the proof. 

\begin{cor}[Real-valued character of the double BT]\label{coro_reales_permutabilidad}
Let $(z_0,w_0)$ be satisfying the hypotheses of Theorem \ref{teorema_permutabilidad}. Then $y_0,v_0$ are real-valued.
\end{cor}

\begin{rem}
This last result finally proves Theorem \ref{MT3}.
\end{rem}

\begin{proof}
Note that Corollary \ref{Coro_permutabilidad} implies $\delta=\overline{\tilde \delta}$. Then, from \eqref{tan_tan_permutabilidad} \begin{align*}
\tan\left(\dfrac{B+z_0-y_0}{4}\right)=\dfrac{2\beta+\delta+\tilde\delta}{2i\alpha+\delta-\tilde \delta}\tan\left(\dfrac{K+u_0-\overline{K}-\bar{u}_0}{4}\right).
\end{align*}
Simplifying, we get
\begin{align*}
\tan\left(\dfrac{B+z_0-y_0}{4}\right)=\dfrac{\beta+\mathrm{Re}\,\delta}{\alpha+\mathrm{Im}\,\delta}\tanh\left(\dfrac{\mathrm{Im}\,(K+u_0)}{2}\right),
\end{align*}
so that $y_0(x)$ is real-valued.
\end{proof}

\section{2-kinks and kink-antikink perturbations}\label{9}

In this section we will assume that $\K=\R$ in Definition \ref{fperturbacion}. Consider  $(D,D_t)=(R,R_t)$  or $(A,A_t)$, $2$-kink or kink-antikink profiles respectively, with shifts $x_1,x_2\in\mathbb{R}$ and speed $\beta\in(-1,1)$, $\beta\neq 0$. Also, we will consider $(Q,Q_t)$ a real-valued kink profile with speed $-\beta$ and shift $x_1+x_2$, see \eqref{Q_util} for more details. 

\medskip

In what follows, we denote by $d$ the parameter of the BT associated to $(D,D_t)$: if $(D,D_t)=(R,R_t)$, then $d:=a_3(\beta) = -a(\beta)$; and if $(D,D_t)=(A,A_t)$, then $d:=a(\beta)$. See Fig. \ref{Fig_flechas_2} for more details.

%\medskip
%
%Similar to Section \ref{Perturbaciones_B}, in this paragraph we will sketch the proof of the following result.

\begin{prop}[Connection to the zero solution]\label{Descenso_global_2k_kak}
Let $(D,D_t)$ be a kink-antikink or $2$-kink profile, as in Definitions \ref{Perfil_2K} and \ref{Perfil_KaK}, with speed $\beta\in(-1,1)\setminus\{0\}$ and shifts $x_1,x_2\in\mathbb{R}$. Let also $ (Q,Q_t)(\cdot; -\beta, x_1 + x_2)$ be a real-valued kink profile associated to $(D,D_t)$, with BT parameter $d$. Then, there exist constants $\eta_0>0$ and $C>0$ such that, for all $0<\eta<\eta_0$ and for all $(z_0,w_0)\in H^1\left(\mathbb{R}\right)\times L^2\left(\mathbb{R}\right)$ such that
\begin{align*}
\Vert (z_0,w_0)\Vert_{H^1(\mathbb{R})\times L^2(\mathbb{R})} < \eta, 
\end{align*}
the following holds:
\ben
\item There are unique $(u_0,s_0,b)$ defined in an open subset of $H^1\left(\mathbb{R}\right)\times L^2\left(\mathbb{R}\right)\times\mathbb{R}$ such that  
\be\label{ABA}
\mathcal F(D+z_0,D_t+w_0,Q +u_0,Q_t+s_0,d+b)  = (0,0),
\ee
and where
\be\label{AAB}
\Vert(u_0,s_0)\Vert_{H^1\times L^2}+\vert b\vert < C\eta.
\ee
%¿Mejor a'un, si $(z_0,w_0)$ son impares, entonces $(u_0,s_0)$ poseen parte real par y parte imaginaria impar?.
\item Making $\eta_0$ smaller if necessary, there are unique $(y_0,v_0,\tilde b)$, defined in an open subset  of $H^1\left(\mathbb{R}\right)\times L^2\left(\mathbb{R}\right)\times\mathbb{R}$, and such that
\be\label{BAB}
\mathcal F(Q +u_0,Q_t+s_0,y_0,v_0,a^{-1}(\beta) +\tilde b) =(0,0), 
\ee
and moreover, 
\be\label{BBA}
 \Vert (y_0,v_0)\Vert_{H^1\times L^2}+\vert \tilde b\vert < C\eta.
\ee
%M'as a'un, si $(u_0,s_0)$ son ¿par e impar? respectivamente, entonces $(y_0,\beta_0)$ son impares.
\een
\end{prop}

The proof of this result is very similar to the one of Proposition \ref{Descenso_global}, so that we only indicate the main differences. First of all, we need the following integrant factor lemma. For the proofs, see Appendix \ref{demostraciones_fi}.

\begin{lem}[Integrant factor for the 2-kink]\label{bajada_mu_2k} Let $(R,R_t)$ and $(Q,Q_t)$ be 2-kink and real-valued kink profiles as in Proposition \ref{Descenso_global_2k_kak}. Let us consider
\begin{align*}
    \mu_R(x)  := \dfrac{\cosh(\ga (x+x_1+x_2))}{\cosh^2(\ga x_1)+\beta^2\sinh^2(\ga(x+x_2))}=\dfrac{1}{4\ga}R_x-\dfrac{1}{4\beta\ga}R_t.
\end{align*}
Then, $\mu_R(x)$ is smooth and solves the ODE: 
\begin{align}\label{mu_edo_2k}
 \mu_x- \frac12\left( \dfrac{1}{d}\cos\left(\dfrac{R+Q}{2}\right)+ d\cos\left(\dfrac{R-Q}{2}\right)  \right) \mu = 0,
\end{align}
where $d=a_3=-a(\beta)$. Moreover, we have the nondegeneracy condiciton 
\begin{align*}
\int_\mathbb{R} \mu_R \cdot\left(R_x-Q_t\right) =\dfrac{4}{\beta} \neq 0.
\end{align*}
\end{lem}

\begin{lem}[Integrant factor for the kink-antikink]\label{bajada_mu_kak} Let $(A,A_t)$ and $(Q,Q_t)$ be kink-antikink and real-valued kink profiles, respectively exactly as in Proposition \ref{Descenso_global_2k_kak}. Let us consider
\begin{align*}
    \mu_A(x)  := \dfrac{\cosh(\ga (x+x_1+x_2))}{\beta^2\cosh^2(\ga (x+x_2))+\sinh^2(\ga x_1)}=\dfrac{1}{4\beta^2\ga}A_t-\dfrac{1}{4\beta\ga}A_x.
\end{align*}
Then, $\mu_A(x)$ is smooth and solves the ODE: 
\begin{align}
 \mu_x- \frac12\left( \dfrac{1}{d}\cos\left(\dfrac{A+Q}{2}\right)+ d\cos\left(\dfrac{A-Q}{2}\right)  \right) \mu = 0,
\end{align}
where $d=a=a(\beta)$. Moreover, we have
\begin{align}\label{integral_mu_kak}
\int_\mathbb{R} \mu_A \cdot\left(A_x-Q_t\right) =-\dfrac{4}{\beta}  \neq 0.
\end{align}
\end{lem}

In order to show \eqref{ABA}-\eqref{AAB}, first item in Proposition \ref{Descenso_global_2k_kak}, we follow the proof in Lemma \ref{bbajada}. After linearizing the BT, we must study whether or not the ODE 
\begin{align*}
    u_{0,x}+\left(\dfrac{1}{2d}\cos\left(\dfrac{D+Q}{2}\right)+\dfrac{d}{2}\cos\left(\dfrac{D-Q}{2}\right)\right)& u_0
    \\  =f+\dfrac{b}{d^2}\sin\left(\dfrac{D+Q}{2}\right)&+b\sin\left(\dfrac{D-Q}{2}\right), 
\end{align*}
has a unique solution $(u_0,\,b)$ such that $u_0\in H^1(\mathbb{R})$, for each $f\in H^1\left(\mathbb{R}\right)$. Using $\mu$ as in Lemmas \ref{bajada_mu_2k} or \ref{bajada_mu_kak} depending on the cases $D=A,R$, we have
\begin{align*}
u_0 = \dfrac{1}{\mu}\int_{-\infty}^x\mu \left(f + \dfrac{b}{d}(D_x-Q_t)\right).
\end{align*}
Additionally, Lemmas \ref{bajada_mu_2k}-\ref{bajada_mu_kak} imply that we can choose $b\in\mathbb{R}$ such that 
\[
\int_{\mathbb{R}}\mu \left(f + \dfrac{b}{d}(D_x-Q_t)\right) =0.
\]
The rest of the proof is similar to the one in Lemma \ref{bbajada}.

\medskip

Finally, \eqref{BAB} and \eqref{BBA}, part of the second item in Proposition \ref{Descenso_global_2k_kak}, are consequence of a new application of the Implicit Function Theorem. In fact, we must study whether or not the equation
\begin{align} 
&-y_{0,x}+\dfrac{\tilde b}{a_2^2}\sin\left(\dfrac{Q}{2}\right)-\dfrac{y_0}{2a_2}\cos\left(\dfrac{Q}{2}\right)+\tilde b \sin\left(\dfrac{Q}{2}\right)-\dfrac{a_2 y_0}{2}\cos\left(\dfrac{Q}{2}\right) \  = \ f,\label{eq_qkink}
\end{align}
possesses a unique solution $(y_0,\,\tilde b)$ such that $y_0\in H^1(\mathbb{R})$, for each $f\in H^1\left(\mathbb{R}\right)$. Simplifying \eqref{eq_qkink} and recalling that $\gamma=(1-\beta^2)^{-1/2}$, we get 
\begin{align*}
    y_{0,x}+\gamma \cos\left(\dfrac{Q}{2}\right) y_0 =  f + \dfrac{2\tilde b}{1-\beta} \sin\left(\dfrac{Q}{2}\right).
\end{align*}
We define now the integrant factor $\mu_Q(x) := \sech(\gamma(x+x_0))$.  Since $\mu_Q$ decays exponentially fast, we have 
\[
y_0 \ = \ \dfrac{1}{\mu_Q}\int_{-\infty}^x\mu_Q \left(f + \frac{2\tilde b}{1-\beta} \sin\left(\frac{Q}{2}\right)\right).
\]
Note that $\int_\mathbb{R} \mu_Q\sin (\frac{Q}{2})=\int_\mathbb{R} \sech^2(\ga(x+x_0))  =\frac{2}{\ga}$. Then, we can choose $\tilde b\in\mathbb{R}$ such that    
\[%\begin{align}\label{cerokq}
\int_{\mathbb{R}}\mu_Q \left(f + \frac{2\tilde b}{1-\beta}\sin\left(\frac{Q}{2}\right)\right) =0.
\]%\end{align}
The rest of the proof is similar to the one in Lemma \ref{kbajada}. 

\section{2-kink and kink-antikink perturbations: inverse dynamics}\label{10}

In this section we still assume $\K=\R$ in Definition \ref{fperturbacion}. Our objective will be to show the following result, in the vein of Proposition \ref{Ascenso_global}.

\begin{prop}[Connection with 2-soliton solutions]\label{Ascenso_global_2k_kak}
Let $(D,D_t)$ be a $2$-kink or kink-antikink profile, as in Definitions \ref{Perfil_2K}-\ref{Perfil_KaK}, with speed $\beta\in(-1,1)\setminus\{0\}$ and shifts $x_1,x_2\in\mathbb{R}$. Let $(Q,Q_t)=(Q,Q_t)(\cdot; -\beta, x_1+x_2)$ be the real-valued kink profile associated to $(D,D_t)$. Then, there are constants $\eta_1>0$ and $C>0$ such that, for all $0<\eta<\eta_1$ and for all $(y,v,\tilde b)\in H^1\left(\mathbb{R}\right)\times L^2\left(\mathbb{R}\right)\times\mathbb{R}$, if
\begin{align*}
\Vert (y,v)\Vert_{H^1(\mathbb{R})\times L^2(\mathbb{R})}+\vert \tilde b\vert <\eta, 
\end{align*}
then the following holds:
\ben
\item There are unique $(u,s)$ defined in $H^1\left(\mathbb{R}\right)\times L^2\left(\mathbb{R}\right)$ such that  
\begin{align*}
\mathcal F(Q +u,Q_t+s,y,v, a(\beta)^{-1}+\tilde b) & = (0,0),
\end{align*}
and for some $\widetilde D_0$ in the Schwartz class and $z$ given by the modulation \eqref{Modulacion_temporal}, 
\be\label{Orto_seguro}
\int_\mathbb{R}(u,s)\cdot (\widetilde D_0,D)=\mathcal N_D(\tilde b,u,z), \quad \int \widetilde D_0 Q_x \neq 0,
\ee
and where $\mathcal N_D(\tilde b,u,z)$ is a nonlinear term in $u$, and where additionally
\begin{align*}
\Vert(u,s)\Vert_{H^1\times L^2}< C\eta.
\end{align*}
%¿Mejor a'un, si $(y,v)$ son impares, entonces $(u,s)$ poseen parte real par y parte imaginaria impar.?
\item If $|b|<\eta$, and making $\eta_1$ smaller if necessary, there are unique $(z,w)$, defined in a subset of $H^1\left(\mathbb{R}\right)\times L^2\left(\mathbb{R}\right)$, and such that 
\begin{align*}
\mathcal F(D+z,D_t+w,Q +u,Q_t+s,d+ b) &=(0,0), 
\end{align*}
\be\label{Axe_Bahia}
\int_\mathbb{R}(z,w)\cdot (D_1,(D_1)_t)=0,
\ee
%\int_\mathbb{R}z_{x}\left(\dfrac{1}{\mu}\right)_x=0,
and finally, $\Vert (z,w)\Vert_{H^1\times L^2} <C\eta.$
%¿M'as a'un, si $(u,s)$ son de parte real par y parte imaginaria impar, entonces $(z,w)$ son impares.?
\een
\end{prop}

Since the proof of this result is similar to the proof of Proposition \ref{Ascenso_global}, we only sketch the main ideas. The first part of Proposition \ref{Ascenso_global_2k_kak} requires to understand if the ODE
\begin{align}\label{ks}
    u_{x}-\gamma\cos\left(\dfrac{Q}{2}\right)u  =  f,
\end{align}    
possesses a unique solution $u\in H^1\left(\mathbb{R}\right)$ for all $f\in H^1\left(\mathbb{R}\right)$. The associated integrating factor here is  $\mu_Q(x)  :=  \cosh(\gamma(x+x_0))$, and the solution $u$ is given by 
 \begin{align*}
    u \ = \ \dfrac{1}{\mu_Q}\mu_Q(0)\,u(0)+\dfrac{1}{\mu_Q}\int_0^x \mu_Q f.
\end{align*}
Precisely, condition \eqref{Orto_seguro} allows us to choose $u$ in a unique form. The value of $\widetilde D_0$, obtained in the same form as $\widetilde B_0$ was obtained in \eqref{B0}, is given by
\[
\widetilde D_0:=  D_{xxt}+ \dfrac{1}{2d}(D-D_{t,x}) \cos\left(\dfrac{D+Q}{2}\right)  - \frac12d (D+D_{t,x})\cos\left(\dfrac{D-Q}{2}\right).
\]
The rest of the proof is the same as before. For the second part, we will need to integrating factors:
\begin{align*}
    \mu^A(x)=:\dfrac{1}{\mu_A}(x)  = \dfrac{\beta^2\cosh^2(\ga(x+x_2))+\sinh^2(\ga x_1)}{\cosh(\ga(x+x_1+x_2))},
\end{align*}
and
\begin{align*}
    \mu^R(x):=\dfrac{1}{\mu_R}(x)  = \dfrac{\beta^2\sinh^2(\ga(x+x_2))+\cosh^2(\ga x_1)}{\cosh(\ga(x+x_1+x_2))},
\end{align*}
which are smooth and solve the ODE 
\begin{align*}
 \mu_x+\left(\dfrac{1}{2d}\cos\left(\dfrac{D+Q}{2}\right)+\dfrac{d}{2}\cos\left(\dfrac{D-Q}{2}\right)\right) \mu = 0,
\end{align*}
with $D=A,R$, $d= a$ and $d=a_3=-a$ respectively.
%que satisface ($a_3=-a(\beta)$)
%\begin{align*}
% \mu_x+\left(\dfrac{1}{2a_3}\cos\left(\dfrac{R+Q}{2}\right)+\dfrac{a_3}{2}\cos\left(\dfrac{R-Q}{2}\right)\right) \mu = 0.
%\end{align*}
Both integrant factors are exponentially increasing in space. With these functions on hand, we plan to conclude the proof. Indeed, the second part requires the study of the ODE 
\begin{align*}
   z_{x}-\left(\dfrac{1}{2d}\cos\left(\dfrac{D+Q}{2}\right)-\dfrac{d}{2}\cos\left(\dfrac{D-Q}{2}\right)\right)z=f.
\end{align*}
Simplifying, and using the integrant factors before proposed, we have
\begin{align}\label{z_cero_subida}
    z \ = \ \dfrac{1}{\mu}\mu(0)\,z(0)+\dfrac{1}{\mu}\int_0^x \mu f, \quad \mu=\mu^R,~\mu^A.
\end{align}
Once again, the uniqueness is obtained by imposing \eqref{Axe_Bahia}. The rest of the proof is well-known.

\section{Stability of $2$-solitons. Proof of Theorem \ref{MT1}}\label{11}

In this Section we prove Theorem \ref{MT1}. Let us consider $(\phi_0,\phi_1)$ satisfying \eqref{initial_data0} for some  $\eta<\eta_0$ small. Let also $(\phi(t),\phi_t(t))$ be the unique solution of \eqref{sg1} with initial condition $(\phi,\phi_t)(0)=(\phi_0,\phi_1)$. Note that $(\phi(t),\phi_t(t)) - (D,D_t)(t) \in H^1\times L^2.$

\begin{proof}[Proof of Theorem \ref{MT1}] Let $\varepsilon_0>0$ be a fixed parameter. Let $(D,D_t)$ be a profile defined as in \ref{mod_estatica}. Consider the tubular neighborhood \eqref{Tubular}, for $t \leq T^* <+\infty$.
%\[%\be\label{Tubular_2}
%\begin{aligned}
%& \mathcal{U}(A_0,\eta)\\
%&\quad  := \left\{(\varphi,\varphi_t) \ :\ \inf_{\tilde{x}_1,\tilde{x}_2\in\mathbb{R}}\Vert (\varphi,\varphi_t)-(D,D_t)(\cdot;\beta,\tilde{x}_1,\tilde{x}_2)\Vert_{H^1\times L^2}\leq A_0\eta\right\}.
%\end{aligned}
%\]%\ee
Note that in order to recover the $2$-soliton solutions of Remarks \ref{2k_solucion} and \ref{kak_solucion}, it is enough to redefine 
\[
(D,D_t)(t,x;\beta,x_1,x_2):=(D,D_t)(x;\beta,x_1+t,x_2).
\]
At this point we split the proof into two cases: $(i)$ breather, and $(ii)$ $2$-kink and kink-antikink.

\subsection*{Breather case}\label{Caso_malo} In what follow we split the proof in two cases: $t$ is uniformly far from all  $t_k$, and the case $t$ close to some $t_k$.
\medskip

1. Let us assume then that $(\phi,\phi_t)(t)$ satisfies \eqref{Tubular} with $T^*$ obeying 
\[
\vert T^*-t_k\vert \geq \varepsilon_0,
\]
for all $k\in\mathbb{Z}$. We plan to show that \eqref{Tubular} is satisfied with $C^*$ replaced by $C^*/2$, proving Theorem \ref{MT1} for all times $t$ far from $t_k$. Indeed, taking $\eta_0>0$ small and $\eta\in(0,\eta_0)$, thanks to Corollary \ref{Mod_Dinamica} we have unique functions $x_1(t),x_2(t)\in\mathbb{R}$, defined in $[0,T^*]$, and such that $(z,w)(t,x)$, defined in \eqref{z_y_w}, satisfy the orthogonality conditions \eqref{Modulacion_temporal}. Note also that we have \eqref{Derivadas_0}.  WLOG, we can assume  \eqref{x1_cond} not satisfied and $x_1(0)=x_2(0)=0$. We define $(z_0,w_0):=(z,w)(0)$. From Proposition \ref{Descenso_global} we obtain functions $(y_0,v_0)$, $(u_0,s_0)$ and parameters $\delta,\tilde{\delta}$. Moreover, Corollary \ref{coro_reales_permutabilidad}, implies that $(y_0,v_0)\in H^1(\mathbb{R})\times L^2(\mathbb{R})$ are real-valued. Recall that the constants from Proposition \ref{Descenso_global} do not depend on $C^*$. Now, we evolve SG to a time $t>0$, with initial data $(y_0,v_0)$. Thanks to Theorem \ref{GWP0} we have \eqref{gwp_estimacion} for $(y(t),v(t))$, and Proposition \ref{Ascenso_global} is valid for all $t\in\mathbb{R}$ far from $t_k$. On the other hand, from Corollary \ref{Mod_Dinamica} we have
\[
\vert x_1'(t)\vert+\vert x_2'(t)\vert \lesssim C^*\eta,
\] 
so that the set of times $\tilde t_k$ where \eqref{x1_cond} is satisfied is still a countable set of points with no accumulation points. Invoking Proposition \ref{Ascenso_global}, starting at $(y,v)(t)$,  and considering for all time $t\in\mathbb{R}$ the $2$-soliton and $1$-soliton profiles 
\[ 
\begin{aligned}
(B^*,B_t^*):= & ~(B,B_t)(x;\beta,x_1(t),x_2(t)), \\
 (\overline{K}^*,\overline{K}_t^*):=& ~(\overline{K},\overline{K}_t)(x;\beta,x_1(t),x_2(t)),
\end{aligned}
\]
and parameters $\beta-i\alpha+\tilde{\delta}, \ \beta+i\alpha+\delta\in\mathbb{C}$, we obtain a function $(B^*,B^*_t)(t)+(z,w)(t)$. This form constructed coincides with the solution $(\phi,\phi_t)(t)$. Indeed, note that at time $t=0$,
%ambas soluciones satisfacen las ecuaciones: 
%\begin{align*}
%\mathcal{F}\big(\phi_0,\phi_1,\bar{K}^*(0)+u_0,\bar{K}_t^*(0)+s_0,\beta-i\alpha+\tilde{\delta}\big)&=0,
%\\ \mathcal{F}\big(B^*(0)+z_0,B_t^*(0)+w_0,\bar{K}^*(0)+u_0,\bar{K}_t^*(0)+s_0,\beta-i\alpha+\tilde{\delta}\big)&=0,
%\end{align*} 
both initial data coincide, so that, thanks to the uniqueness of the solutions associated to the Cauchy problem \eqref{sg1} (see also Theorems \ref{GWP0} and \ref{GWP2}), we conclude that $(B^*+z,B_t^*+w)(t)$ obtained via BT is actually $(\phi,\phi_t)(t)$. Finally, we also have 
\begin{align}\label{estimacion_caso1_MT}
\sup_{\vert t-t_k\vert \geq \varepsilon_0}\Vert (\phi,\phi_t)(t)-(B^*,B^*_t)(t)\Vert_{H^1\times L^2}\leq C_0\eta,
\end{align}
so that, considering $C^*$ large such that $C_0\leq \frac{1}{2}C^*$, we conclude that $T^*$ must be infinite (see \eqref{Mar}). This idea is schematically represented in Fig. \ref{Fig:final}.
\begin{figure}[h!]
\begin{tikzpicture}
  \matrix (m) [matrix of math nodes,row sep=3em,column sep=3em,minimum width=3em]
  {
    (B^*, B^*_t)(0)+(z_0,w_0) & & & (B^*,B^*_t)(t)+(z,w)(t) \\
     (K^*,K^*_t)(0)+(u_0,s_0) &  & & (\bar{K}^*,\bar{K}^*_t)(t)+(\bar{u},\bar{s})(t) \\
    (y_0,v_0) & & & (y,v)(t) \\};
  \path[-stealth]
    
    (m-1-1) 
            edge node [above] {$t$} 
    (m-1-4)
    
    (m-2-4) 
            edge node [right] {$\ \beta-i\alpha+\tilde{\delta} $} 
    (m-1-4)
        
    (m-2-1) 
            edge node [left] {$ \beta-i\alpha+\tilde{\delta}\ $}
    (m-3-1)
    
    (m-3-1) 
            edge node [above] {$t$} 
    (m-3-4)        
    
    (m-3-4) 
            edge node [right] {$\ \beta+i\alpha+\delta$} 
    (m-2-4)
  
    (m-1-1) 
            edge node [left] {$ \beta+i\alpha+\delta \ $} 
    (m-2-1);
\end{tikzpicture}
\caption{Diagram for the proof of Theorem \ref{MT1} in the case where $x_1(t)$ does not follow \eqref{x1_cond}.}\label{Fig:final}
\end{figure}

\medskip
2. Let us consider now the case $|T^* - t_k|< \ve_0$ for some $k \in\N$ fixed. We shall prove that for $\varepsilon_0$ sufficiently small, but independent of $k$,)
\begin{align}\label{estimacion_tk_MT}
\sup_{\vert t-t_k\vert < \varepsilon_0} \Vert (\phi,\phi_t)-(B^*,B^*_t)\Vert_{H^1\times L^2}\leq  \frac34 C^*\eta.
\end{align}
Since $C^*$ grows as $\varepsilon_0$ tends to zero, we must choose $\eta_0$ sufficiently small such that each step above holds properly. Let $I_k:= (t_k-\varepsilon_0,t_k+\varepsilon_0].$ Let us consider
\begin{align}\label{cota_hipotesis_MT}
 T_*:= \sup\Big\{T\in I_k: \  \forall\, t \in (t_k-\varepsilon_0, T], \ \Vert (z,w)(t)\Vert_{H^1\times L^2}\leq \frac34C^*\eta \Big\}.
\end{align}
It is enough to show $T_*=t_k+\varepsilon_0$. Let us assume that $T_*<t_k+\varepsilon_0$. Note that, by the same argument as the previous step, using BT we have 
\[ 
\Vert (z,w)(t_k-\ve_0)\Vert_{H^1\times L^2}=\frac12 C^*\eta.
\] 
Now, we use a bootstrap argument. Let $t\in[t_k-\varepsilon_0,T_*]$ and consider 
\begin{align*}
\Delta:= \frac{d}{dt}\left(\dfrac{1}{2}\int_\mathbb{R} (z_x^2 +z^2+w^2 )(t,x)dx \right).
\end{align*}
We claim that $\Delta $ is bounded by $C(C^*)^2\eta^2$, a contradiction to the definition of $T_*$. First, we will need \eqref{SG} in terms of $(z,w)$, using \eqref{z_y_w} with $D=B$. In fact, 
\[
\begin{cases}
\partial_t B^* + z_t =B_t^* + w\\
\partial_t B_t^* + w_t = B_{xx}^* + z_{xx}-\sin(B^* +z) .
\end{cases}
\]
Simplifying, we get
\[
\begin{cases}
 z_t =  w-x_1' B_t-x_2'B_x \\
  w_t =   z_{xx}-\sin(B^* +z) +\sin B^* -x_1' B_{tt}^*-x_2' B_{tx}^*.
\end{cases}
\]
Now, computing directly,
\begin{align}
\Delta  &= \int_\R (zz_t - z_{xx} z_{t} +ww_t  )\nonu\\
&= \int_\R  (z - z_{xx}) (w-x_1' B_t^*-x_2'B_x^*)  \nonu \\
& \qquad + \int_\R w (  z_{xx}-\sin(B^* +z) +\sin B^* -x_1' B_{tt}^*-x_2' B_{tx}^*)  \nonu
\\
&= \int_\R   z (w-x_1' B_t^*-x_2'B_x^*)+ \int_\R   z_{xx} (x_1' B_t^* + x_2'B_x^*)  \nonu \\
& \qquad + \int_\R w \left(\sin B^* (\cos z-1) + \cos(B^*)\sin z -x_1' B_{tt}^*-x_2' B_{tx}^* \right).\nonu
\end{align}
Clearly if $(z,w)$ are small,
\[
\abs{\Delta} \lesssim \int_\R (z_x^2 + z^2 + w^2)   + |x_1'(t)|^2 +|x_2'(t)|^2.
\]
Therefore, using  \eqref{cota_hipotesis_MT} and \eqref{Derivadas} we obtain that for $t\in(t_k-\varepsilon_0,T_*]$ it holds  
\begin{align*}
\left\vert \frac{d}{dt} \dfrac{1}{2}\int_{\mathbb{R}} (z^2+z_x^2+w^2)\right\vert  =\abs{\Delta} \leq C(C^*)^2\eta^2.
\end{align*}
Consequently, integrating we have that for $\varepsilon_0$ sufficiently small (but fixed) 
\[
\begin{aligned}
& \int_\mathbb{R}z^2(T^*)+z_x^2(T^*)+w^2(T^*)\\
&~ \qquad \leq  \int_\mathbb{R}z^2(t_k-\varepsilon_0)+z_x^2(t_k-\varepsilon_0)+w^2(t_k-\varepsilon_0)+C\varepsilon_0(C^*)^2\eta^2  \leq \frac34(C^*)^2\eta^2.
\end{aligned}
\]
Then, \eqref{cota_hipotesis_MT} has been improved, and $T_*=t_k+\varepsilon_0$. This estimate does not depend on $k\in\mathbb{Z}$, but only on the length of the interval $\sim\varepsilon_0$. Therefore, $T^*$ in \eqref{Tubular} is infinite for all $C^*$ large enough. This proves \eqref{final_data0} and the proof of Theorem \ref{MT1} in the case of the breather solution.

\subsection*{$2$-kink or kink-antikink case} Here we can repeat the previous scheme but with no problem on the time $t$ chosen. Since proofs are similar, we only sketch the main steps.
\medskip

Let $(z,w)(t)$ be the functions defined in \eqref{z_y_w} and $x_1(t),\, x_2(t)$ modulations from Corollary \ref{Mod_Dinamica}. Hence, applying Proposition \ref{Descenso_global_2k_kak} with perturbation $(z_0,w_0)=(z,w)(0)$ we obtain functions with real values $(y_0,v_0)$. Then, we evolve SG with initial data $(y_0,v_0)\in H^1(\mathbb{R}) \times L^2 (\mathbb{R})$. Finally, we consider functions  $(Q,Q_t)(x;-\beta,x_1+x_2)$ and a parameter of BT $d\in\mathbb{R}$ given as follows: 
\begin{enumerate}
\item If $(D,D_t)=(A,A_t)$, then we have $d:=a(\beta)$.
\smallskip
\item If $(D,D_t)=(R,R_t)$ then $d:=-a(\beta)$.
\end{enumerate} 
Now we invoke Proposition \ref{Ascenso_global_2k_kak} for each time $t$ fixed, and with $2$-soliton and $1$-soliton profiles given by 
\[
\begin{aligned}
(D^*,D^*_t):=&~(D,D_t)(x;\beta,x_1(t),x_2(t)), \\
 (Q^*,Q^*_t):=& ~(Q,Q_t)(x;-\beta,x_1(t)+ x_2(t)).
\end{aligned}
\]
Thanks to the uniqueness of the solution to the Cauchy problem \eqref{sg1}, we have coincidence between $(\phi,\phi_t)(t)$ and the functions returned via BT. Lastly, noticing that from Theorem \ref{GWP0} we have 
\[
\sup_{t\in\mathbb{R}}\Vert (y,v)(t)\Vert_{H^1\times L^2}\lesssim \Vert (y_0,v_0)\Vert_{H^1\times L^2},
\]
we conclude from Proposition \ref{Ascenso_global_2k_kak} that
\[
\sup_{t\in\mathbb{R}}\Vert (\phi,\phi_t)(t)-(D^*,D^*_t)(t)\Vert_{H^1\times L^2}\leq C_0\eta.
\]
The proof of Theorem \ref{MT1} in these cases is complete.
\end{proof}

%\subsection{Demostraci'on del Corolario \ref{MT2}}
%
%Sea $(K,K_t)(t)$ un perfil de kink complejo tal que en tiempo $t=0$ no explota. Sea $(u_0,s_0)\in (H^1\times L^2)(\R; \Com)$ suficientemente peque\~no.  Siguiendo \cite{DG}, es posible definir una teor'ia local de existencia para \eqref{sg1} con dato inicial $(K,K_t)(t=0)+(u_0,s_0)$. Despu'es de modular apropiadamente los par'ametros $x_1(t)$ y $x_2(t)$, podemos aplicar el argumento del Theorem \ref{MT1} de transformaci'on de Backlund para , 
%
%\begin{cor}[Estabilidad de la explosi'on para $(K,K_t)$]\label{MT2}
% es posible definir una 'unica soluci'on de SG de la forma 
%\[
%(\phi,\phi_t)(t) =(\widetilde K, \widetilde K_t)(t) + (u,s)(t), \quad (u,s)(t)\in (H^1\times L^2)(\R; \Com),
%\]
%donde $(\widetilde K, \widetilde K_t)(t) $ es un perfil de kink complejo $(K, K_t)(t) $ razonablemente modificado v'ia modulaciones en tiempo. Esta soluci'on est'a bien definida para cada $t\neq \tilde t_k$, una sucesi'on de tiempos no acotada y sin punto de acumulaci'on, cercanos a cada $t_k$. Asimismo, esta soluci'on explota en cada instante de tiempo $\tilde t_k$.
%\end{cor}

\subsection{Proof of Corollary \ref{MT4}} We will show the breather case only, the other cases are very similar. Thanks to Lemma  \ref{Energia_Back} and \eqref{CosBK}, it is enough to compute
\[
\begin{aligned}
\ell^{+,1}_\pm(t)=& \lim_{x\to \pm \infty} \left(1-\cos\left(\dfrac{B+z+K+u}{2}\right)\right) \\
=& \lim_{x\to \pm \infty} \left(1-\cos\left(\dfrac{B+K}{2}\right)\right) = \begin{cases} 2, & x\to +\infty \\ 0 & x\to -\infty \end{cases}, \\
 \ell^{-,1}_\pm(t) =& \lim_{x\to \pm \infty} \left(1-\cos\left(\dfrac{B+z-(K+u)}{2}\right)\right) =\begin{cases} 2, & x\to +\infty \\ 0 & x\to -\infty\end{cases}.%\lim_{x\to \pm \infty} \left(1-\cos\left(\dfrac{-K}{2}\right)\right),
\end{aligned}
\]
and
\[
\begin{aligned}
\ell^{+,2}_\pm(t)=& \lim_{x\to \pm \infty} \left(1-\cos\left(\dfrac{K+u+y}{2}\right)\right) = \begin{cases} 2, & x\to +\infty \\ 0 & x\to -\infty \end{cases}, \\
 \ell^{-,2}_\pm(t) =& \lim_{x\to \pm \infty} \left(1-\cos\left(\dfrac{K+u-y}{2}\right)\right)= \begin{cases} 2, & x\to +\infty \\ 0 & x\to -\infty \end{cases}.
\end{aligned}
\]
Hence, using these values, and Proposition \ref{Descenso_global} and \eqref{energia},
\begin{align*}
E[B+z,B_t+w]&=E[K+u,K_t+s]+ \dfrac{4}{\bt+i\al +\delta}+ 4 (\bt+i\al +\delta) , 
\\ E[K+u,K_t+s]&=E[y,v]  + \dfrac{4}{\bt-i\al +\tilde \delta} + 4(\bt-i\al +\tilde\delta).
\end{align*}
Since $\tilde \delta =\overline{\delta}$ (see Corollary \ref{Coro_permutabilidad}), we obtain
\[
E[B+z,B_t+w] =E[y,v]  +  \dfrac{8(\bt +\re\delta)}{(\bt + \re\delta)^2 +(\al +\ima\delta)^2}+ 8 (\bt + \re \delta),
\]
from which we obtain \eqref{EE}, since $\al^2 +\bt^2=1$. For the momentum part, we proceed in the same fashion, obtaining \eqref{MM}.
%
%
%\section{Estabilidad asint'otica}
%

\appendix

%
%\section{Demostraci'on del Lema \ref{prkk1}}\label{A1}
%
%\medskip
%
%Procederemos mediante un cálculo directo. Derivando respecto a $x$ en \eqref{Q} obtenemos que 
%\be\label{Qx_back}
%Q_x =\dfrac{4\gamma e^{\gamma(x+x_0)}}{1+e^{2\gamma(x+x_0)}} = \frac{2\ga}{\cosh(\ga(x+x_0))}.
%\ee
% 
%Por otro lado, reemplazando directamente \eqref{Q} en el lado derecho y aplicando identidades trigonométricas básicas obtenemos 
%\begin{align*}
%\sin\left(\frac{Q}{2}\right) = \dfrac{2e^{\gamma(x+x_0)}}{1+e^{2\gamma(x+x_0)}} = \frac{1}{\cosh(\ga(x+x_0))}.
%\end{align*}
%Comparando ambas expresiones se concluye la condición sobre $a$. En concreto, tenemos que debe satisfacerse que  $$\dfrac{1}{a} \ = \ \dfrac{1-\beta}{\sqrt{1-\beta^2}} \ \ \Longrightarrow \ \  a \ = \  \dfrac{(1+\beta)^{1/2}}{(1-\beta)^{1/2}},$$
%lo que concluye la demostración. La paridad de $Q_x$, $Q_t$ y $\sin (\frac Q2)$ con respecto al eje $x=-x_0$ es directa de los c'alculos de arriba y de \eqref{Qt}. 
%
%\bigskip
%

\section{Proof of Proposition \ref{back_kak}}\label{ABC}

We start proving that \eqref{kakbk1} is satisfied. We follow the same scheme of Proposition \ref{back_breather}. Taking derivative of $A$ wrt $x$ we get
\begin{align}\label{A_x_back}
A_x  \ =& \ \dfrac{4\beta^2\cosh^2(\ga(x+x_2))}{\beta^2\cosh^2(\ga(x+x_2))+\sinh^2(\ga x_1)}\cdot\dfrac{-\sinh(\ga x_1)}{\beta\cosh^2(\ga(x+x_2))}\cdot\ga\sinh(\ga(x+x_2)) \nonu
\\  = &\ -\dfrac{4\beta\gamma\sinh(\gamma x_1)\sinh(\gamma(x+x_2))}{\sinh^2(\gamma x_1)+\beta^2\cosh^2(\gamma(x+x_2)}.
\end{align}
For the sake of simplicity we define $\theta := \gamma(x-x_1+x_2)$. Using basic trigonometric identities we have
\begin{align}
    \sin\left(\dfrac{A\pm Q}{2}\right) =  \dfrac{2\tan\left(\arctan\left(\frac{\sinh (\gamma x_1)}{\beta\cosh (\gamma(x+x_2))}\right) \pm \arctan\left(e^{\theta}\right)\right)}{1+\tan^2\left(\arctan\left(\frac{\sinh (\gamma x_1)}{\beta\cosh (\gamma(x+x_2))}\right) \pm \arctan\left(e^{\theta}\right)\right)}. \label{rhskak1}
\end{align}
Since $\tan(a\pm b)=\frac{\tan a\pm \tan b}{1\mp\tan a\tan b}$, \eqref{rhskak1} reads now
\[
\sin\left(\dfrac{A\pm Q}{2}\right)=\dfrac{2\left(\dfrac{\sinh(\ga x_1)\pm \beta\cosh(\ga(x+x_2))e^\theta}{\beta\cosh(\ga(x+x_2))\mp\sinh(\ga x_1)e^\theta}\right)}{1+\left(\dfrac{\sinh(\ga x_1)\pm \beta\cosh(\ga(x+x_2))e^\theta}{\beta\cosh(\ga(x+x_2))\mp\sinh(\ga x_1)e^\theta}\right)^2},
\]
and simplifying,
\begin{align}\label{rhs_s_kak}
    &\sin\left(\dfrac{A\pm Q}{2}\right) = \dfrac{2f_2(x)}{\big(1+e^{2\theta}\big)\big(\sinh^2(\gamma x_1)+\beta^2\cosh^2(\gamma(x+x_2))\big)}, 
\end{align}
where $f_2(x) = f_2(x;\beta,x_1,x_2)$ is such that
\begin{align}
 f_2(x)  := & ~ \beta\sinh(\gamma x_1)\cosh(\gamma(x+x_2))\mp e^\theta\sinh^2(\gamma x_1) \nonumber\\ 
& \quad \pm \beta^2e^\theta\cosh^2(\gamma(x+x_2))-\beta e^{2\theta}\sinh(\gamma x_1)\cosh(\gamma(x+x_2)).\label{rhs_kak}
\end{align} 
We are now ready to show that \eqref{kakbk1} is satisfied. Subtracting \eqref{Qt} from \eqref{A_x_back} we obtain
\[
A_x-Q_t=-\dfrac{4\beta\gamma\sinh(\gamma x_1)\sinh(\gamma(x+x_2))}{\sinh^2(\gamma x_1)+\beta^2\cosh^2(\gamma(x+x_2))}-\frac{4\beta \ga e^{\theta}}{1+ e^{2\theta}}=\dfrac{F_2}{F_3},
\]
where
\begin{align}\label{c_kak_back}
F_3 := \big(1+e^{2\theta}\big)\,\big(\sinh^2(\gamma x_1)+\beta^2\cosh^2(\gamma(x+x_2)\big), \quad \hbox{ y}
\end{align}
\[
\begin{aligned}
 F_2  :=  & -4\beta\gamma\bigg[e^\theta\big(\beta^2\cosh^2(\gamma(x+x_2))+\sinh^2(\gamma x_1)\big)  \\
 & \qquad \qquad +(1+e^{2\theta})\sinh(\gamma(x+x_2))\sinh(\gamma x_1) \bigg].
\end{aligned}
\]
On the other hand, recalling that  $a:=a(\beta)$ and $\ga=1/\sqrt{1-\beta^2}$,  from \eqref{rhs_kak} we conclude 
\begin{align}\label{rhs_kak_3}
\dfrac{1}{a}\sin\left(\dfrac{A+Q}{2}\right)+a\,\sin\left(\dfrac{A-Q}{2}\right) = \dfrac{F_4}{F_3}
\end{align}
where $F_3$ is given by \eqref{c_kak_back} and
\[
\begin{aligned}
&F_4:= 4\beta\gamma\big[(1-e^{2\theta})\sinh(\gamma x_1)\cosh(\gamma(x+x_2))+e^\theta\sinh^2(\gamma x_1)-\beta^2e^\theta\cosh^2(\gamma(x+x_2))\big].
\end{aligned}
\]
Therefore, \eqref{kakbk1} is reduced to show that $F_2-F_4\equiv 0$. Indeed, 
\begin{align*}
F_2-F_4\ =& \ -4\beta\gamma\big[2e^\theta\sinh^2(\gamma x_1)+(1+e^{2\theta})\sinh(\gamma(x+x_2))\sinh(\gamma x_1) \big]
    \\ & \ \quad - 4\beta\gamma (1-e^{2\theta})\sinh(\gamma x_1)\cosh(\gamma(x+x_2))=  0.
\end{align*}
This proves \eqref{kakbk1}. We only need to show \eqref{kakbk2} now. We follow the same scheme as before: 
%Derivando $Q$ respecto a $x$ tenemos que \begin{align}\label{Qx_back}
%Q_x = \dfrac{4\ga e^{\theta}}{1+e^{2\theta}}.
%\end{align} Con esto estamos listos para mostrar \eqref{kakbk2}. 
form \eqref{perfil_dt_KaK} and \eqref{Q} we have
\[
A_t-Q_x=\dfrac{4\beta^2\gamma \cosh(\gamma(x+x_2))\cosh(\gamma x_1)}{\beta^2\cosh^2(\gamma(x+x_2))+\sinh^2(\gamma x_1)}-\dfrac{4\ga e^{\theta}}{1+e^{2\theta}} = \dfrac{\widetilde{F}_2}{F_3},
\] 
where $F_3$ is given in \eqref{c_kak_back} and 
\begin{align}\label{bkrhs}
    \widetilde{F}_2 := & 4\gamma\big[\beta^2\cosh(\ga(x+x_2))\cosh(\ga x_1)(1+e^{2\theta}) \nonu\\
    & {} \qquad -(\beta^2\cosh^2(\ga(x+x_2))+\sinh^2(\ga x_1))e^\theta\big].
 \end{align}
On the other hand, since $a=a(\beta)$ and $\ga=1/\sqrt{1-\beta^2}$, and following the same ideas as in the proof of \eqref{rhs_kak_3}, we have
\[
   \dfrac{1}{a}\sin\left(\dfrac{A+Q}{2}\right)-a\,\sin\left(\dfrac{A-Q}{2}\right)\ = \ \dfrac{\widetilde{F}_4}{F_3},
\]
where $F_3$ came from \eqref{c_kak_back} and $\widehat{F}_4$ denotes the quantity
\begin{align}
   \widetilde{F}_4\ :=\ & -4\gamma\big[ \beta^2\sinh(\ga x_1)\cosh(\ga(x+x_2))(1-e^{2\theta}) \nonu\\
   & \qquad -e^\theta(\beta^2\cosh^2(\ga(x+x_2))-\sinh^2(\ga x_1))\big] 
   . \label{bkrhs2}
 \end{align}
Therefore, \eqref{kakbk2} has been reduced to show that $\widetilde{F}_2-\widetilde{F}_4\equiv 0$. Indeed, from \eqref{bkrhs} and \eqref{bkrhs2}
\begin{align*}
    \widetilde{F}_2-\widetilde{F}_4     &  =  4\ga\beta^2\cosh(\ga x_1)\cosh(\ga (x+x_2))(1+e^{2\theta})-8\ga \beta^2\cosh^2(\ga(x+x_2))e^\theta
    \\ & \ \quad + 4\ga \beta^2(1-e^{2\theta})\sinh(\ga x_1)\cosh(\ga(x+x_2))  =  0,
\end{align*} 
which ends the proof.

\section{Description of derivatives and orthogonality}\label{ap_derivadas}

\subsection{Orthogonality for breather type functions} We start with the following result.

\begin{lem}
Let $(B,B_t)$ be a SG breather profile with scaling parameter $\beta\in(-1,1)\setminus\{0\}$ and shifts $x_1,x_2\in\mathbb{R}$. Let us suppose that  $x_2=0$. Then, $B_t$ and $B_x$ are even and odd respectively.
\end{lem}
\begin{proof} It is enough to see that from \eqref{perfil_Breather}, \eqref{B_x} and \eqref{perfil_dtBreather},  %{\color{red}(Colocar parentesis en los valores de las funciones abajo)}
\be\label{B1B2}
\begin{aligned}
 B_t =B_1 =&~  \dfrac{4\alpha^2\beta \cos(\alpha x_1) \cosh (\beta (x+x_2))}{\alpha^2\cosh^2(\beta (x+x_2))+\beta^2\sin^2(\alpha x_1)}, \\
  B_x =B_2  =&~  \dfrac{-4\beta^2\alpha\sin(\alpha x_1)  \sinh (\beta (x+x_2))}{\alpha^2\cosh^2(\beta (x+x_2))+\beta^2\sin^2(\alpha x_1)},
\end{aligned}
\ee
so that if $x_2=0$ we get
\[
 B_t  = \dfrac{4\alpha^2\beta \cos(\alpha x_1) \cosh (\beta x)}{\alpha^2\cosh^2(\beta x)+\beta^2\sin^2(\alpha x_1)}, \ \quad \ B_x   = \dfrac{-4\beta^2\alpha\sin(\alpha x_1)  \sinh (\beta x)}{\alpha^2\cosh^2(\beta x)+\beta^2\sin^2(\alpha x_1)},
\]
which readily gives the respective parity properties.
\end{proof}

\begin{cor}\label{paridad}
Let $(B,B_t)$ be a SG breather with scaling parameter $\beta\in(-1,1)\setminus\{0\}$ and shifts $x_1,x_2\in\mathbb{R}$. Then, 
\begin{align*} \int_\mathbb{R} B_tB_xdx =0.
\end{align*}
\end{cor}

\begin{proof} A consequence of the previous lemma and the invariance under translations of the integral on $\R$. %Basta notar que, gracias a que $B_t B_x$ es integrable, y del hecho que bajo un cambio de variables siempre podemos asumir que $x_2=0$, se deduce el resultado en virtud del lema anterior.
%Basta notar que bajo un cambio de variables siempre podemos asumir que $x_2=0$, de donde se deduce el resultado en virtud del lema anterior.
\end{proof}

\begin{lem}
Let $(B,B_t)$ be a SG breather profile with scaling parameter $\beta\in(-1,1)$, $\beta\neq 0$, and shifts $x_1,x_2\in\mathbb{R}$. Consider $(B_i,B_{t,i})$ the derivatives of $B$ y $B_t$ wrt the variables $x_i$, $i=1,2$. Let us additionally suppose that $x_2=0$. Then, $B_{t,1}$ and $B_{t,2}$ are functions in the Schwartz class, even and odd in $x$ respectively.
\end{lem}

\begin{proof} For the sake of brevity we define $\theta_1:=\gamma x_1$ y $\theta_2:=\ga(x+x_2)=\ga x$.  Since $B_t$ in \eqref{perfil_dtBreather} is smooth, we have after differentiation
 \begin{align*}
 B_{t,1} \!\! &=-4\alpha^3\beta\dfrac{\big(\sin\theta_1\cosh\theta_2(\alpha^2\cosh^2\theta_2+\beta^2\sin^2\theta_1)+\beta^2\sin(2\theta_1)\cos\theta_1\cosh\theta_2\big)}{\big(\alpha^2\cosh^2\theta_2+\beta^2\sin^2\theta_1\big)^2} , 
 \\ B_{t,2} \!\! &= 4\alpha^2\beta^2 \dfrac{\big(\cos\theta_1\sinh\theta_2(\alpha^2\cosh^2\theta_2+\beta^2\sin^2\theta_1)-\alpha^2\sinh(2\theta_2)\cos\theta_1\cosh\theta_2\big)}{\big(\alpha^2\cosh^2\theta_2+\beta^2\sin^2\theta_1\big)^2}.
\end{align*}
The desired parity properties are then direct.
\end{proof}

\begin{cor}\label{paridad2}
Let $(B,B_t)$ be a SG breather profile with scaling parameter $\beta\in(-1,1)\setminus\{0\}$ and shifts $x_1,x_2\in\mathbb{R}$. Then, 
\begin{align*} 
\int_\mathbb{R} B_{t,1}B_{t,2}dx =0.
\end{align*}
\end{cor}
\begin{proof} Direct from previous lemma. %Basta notar que, gracias a que $B_{t,x_1}\cdot B_{t,x_2}$ es integrable, y del hecho que bajo un cambio de variables siempre podemos asumir que $x_2=0$, se deduce el resultado en virtud del lema anterior.
%Basta notar que bajo un cambio de variables siempre podemos asumir que $x_2=0$, de donde se deduce el resultado en virtud del lema anterior.
\end{proof}

\subsection{Orthogonality of 2-kink or kink-antikink type functions}\label{Orto_KK} In this subsection, we treat the case of 2-kink $R$ and kink-antikink $A$. Since proofs are similar to the breather case, we only sketch the main ideas.

\begin{lem}\label{AAt_derivadas}
Let $(A,A_t)$ be a SG kink-antikink profile with speed $\beta\in(-1,1)\setminus\{0\}$ and shifts $x_1,x_2\in\mathbb{R}$. Consider $(A_{i},A_{t,i})$ the derivatives of $A$ and $A_t$ wrt the directions $x_i$, $i=1,2$. Suppose again that $x_2=0$. Then, $A_t$ and $A_{t,1}$ are even, and $A_x$ and $A_{t,2}$ are odd. Each function above is in the Schwartz class.
\end{lem}

\begin{proof} We define $\theta_1:=\gamma x_1$ y $\theta_2:=\ga(x+x_2)=\ga x$.  Thanks to  \eqref{A_x_back}, \eqref{perfil_dt_KaK} and direct computations, we have %Dada la suavidad de $B_t$ podemos diferenciar directamente, de donde obtenemos que
 \begin{align*}
 A_t &= \dfrac{4\beta^2\ga\cosh\theta_1\cosh\theta_2}{\beta^2\cosh^2\theta_2+\sinh^2\theta_1}, \ \quad A_x = \dfrac{-4\beta\ga\sinh\theta_1\sinh\theta_2}{\beta^2\cosh^2\theta_2+\sinh^2\theta_1},
\\ 
A_{t,1} &=\dfrac{4\beta^2\ga^2\big(\sinh\theta_1\cosh\theta_2(\beta^2\cosh^2\theta_2+\sinh^2\theta_1)-\sinh(2\theta_1)\cosh\theta_1\cosh\theta_2 \big)}{\big(\beta^2\cosh^2\theta_2+\sinh^2\theta_1\big)^2},
\\ 
A_{t,2} &= \dfrac{4\beta^2\ga^2\big(\cosh\theta_1\sinh\theta_2 (\beta^2\cosh^2\theta_2+\sinh^2\theta_1)-\beta^2\sinh(2\theta_2)\cosh\theta_1\cosh\theta_2\big)}{\big(\beta^2\cosh^2\theta_2+\sinh^2\theta_1\big)^2}.
\end{align*}
Here parities are concluded directly since $x_2=0$.
\end{proof}

\begin{cor}\label{paridad3}
Let $(A,A_t)$ be a kink-antikink profile with speed $\beta\in(-1,1)$ and shifts $x_1,x_2\in\mathbb{R}$. Then, 
\begin{align*} 
\int_\mathbb{R} A_tA_x dx =0, \ \quad \int_{\mathbb{R}}A_{t,1}A_{t,2}=0.
\end{align*}
\end{cor}

\begin{proof} Direct form the previous lemma. %Basta notar que, gracias a que $B_{t,x_1}\cdot B_{t,x_2}$ es integrable, y del hecho que bajo un cambio de variables siempre podemos asumir que $x_2=0$, se deduce el resultado en virtud del lema anterior.
%Basta notar que bajo un cambio de variables siempre podemos asumir que $x_2=0$, de donde se deduce el resultado en virtud del lema anterior.
\end{proof}

\begin{lem}
Let $(R,R_t)$ be a SG 2-kink profile with speed $\beta\in(-1,1)$ and shifts $x_1,x_2\in\mathbb{R}$. Let us consider $(R_{i},R_{t,i})$ the derivatives of $R$ y $R_t$ in the directions $x_i$, $i=1,2$. Let us assume additionally that $x_2=0$. Then, $R_t$ and $R_{t,1}$ are odd, and $R_x$ and $R_{t,2}$ are even. Each of the last four last functions is in the Schwartz class.\footnote{Note that $R$ is not in the Schwartz class.}
\end{lem}

\begin{proof} Using the same notation as in the proof of Lemma \ref{AAt_derivadas}, we have
\[
 R_t = \dfrac{-4\beta^2\ga\sinh\theta_1\sinh\theta_2}{\cosh^2\theta_1+\beta^2\sinh^2\theta_2}, \ \quad R_x = \dfrac{4\beta\ga\cosh\theta_1\cosh\theta_2}{\cosh^2\theta_1+\beta^2\sinh^2\theta_2},
 \]
and
\begin{align*}
& \frac{R_{t,1}}{4\beta^2\ga^2} =-\dfrac{\cosh\theta_1\sinh\theta_2(\cosh^2\theta_1+\beta^2\sinh^2\theta_2)-\beta^2\sinh(2\theta_1)\sinh\theta_2\sinh\theta_1}{\big(\cosh^2\theta_1+\beta^2\sinh^2\theta_2\big)^2},
\\
& \frac{R_{t,2}}{4\beta^2\ga^2}= - \dfrac{\sinh\theta_1\cosh\theta_2 (\cosh^2\theta_1+\beta^2\sinh^2\theta_2)-\beta^2\sinh(2\theta_2)\sinh\theta_1\sinh\theta_2}{\big(\cosh^2\theta_1+\beta^2\sinh^2\theta_2\big)^2}.
\end{align*}
\end{proof}
Finally, the following result is direct:

\begin{cor}\label{paridad4}
Let $(R,R_t)$ be a $2$-kink SG profile with speed $\beta\in(-1,1)$, $\beta\neq 0$, and shifts $x_1,x_2\in\mathbb{R}$. Then, 
\begin{align*} 
\int_\mathbb{R} R_tR_x dx =0, \ \quad \int_{\mathbb{R}}R_{t,1}R_{t,2}=0.
\end{align*}
\end{cor}
%\begin{proof} Basta notar que, gracias a que $R_{t,x_1}\cdot R_{t,x_2}$ es integrable, y del hecho que bajo un cambio de variables siempre podemos asumir que $x_2=0$, se deduce el resultado en virtud del lema anterior.
%%Basta notar que bajo un cambio de variables siempre podemos asumir que $x_2=0$, de donde se deduce el resultado en virtud del lema anterior.
%\end{proof}

\section{Proof of Lemma \ref{Mod_esta}}\label{Modula}

The proof of this result is standard, we only sketch the main ideas. Let us define $H:\mathbb{R}^2\times \mathcal{U}(\eta)\to\mathbb{R}^2$, given by 
\[
 \big(H(x_1,x_2,\phi,\phi_t)\big)_j: = \Big( \langle \phi-D ,\,D_{j} \rangle_{H^1} , \langle \phi_t -D_t ,(D_t)_{j} \rangle_{L^2} \Big), \ \ j=1,2,
\]
where $D$, $D_t$, $D_j$ y $D_{t,j}$ are evaluated at the point $(\cdot;\bt,x_1,x_2)$. Clearly we have $H(x_1,x_2,D,D_t)=(0,0)$.
%Then, por el Theorem de la Función Implícita, existe un único par de funciones $r_i :\mathcal{U}(\eta)\to\mathbb{R},$ $i=1,2$, tales que 
%$$ \big\Vert (\phi,\phi_t) -(B,B_t)\big( \cdot;r_1( \phi,\phi_t),r_2( \phi,\phi_t)\big)\big\Vert_{H^1\times L^2}= \inf_{x_1,x_2\in\mathbb{R}}\big\Vert (\phi,\phi_t) -(B,B_t)( \cdot;x_1,x_2)\big\Vert_{H^1\times L^2} \leq \eta . $$
%En efecto, notando que para $\bar{x}_1,\bar{x}_2\in\mathbb{R}$ arbitrarios pero fijos, basta considerar las funciones $(\phi^*,\phi_t^*)( x)=(B,B_t)( x;\bar{x}_1,\bar{x}_2)$, con lo que $ F( \bar{x}_1,\bar{x}_2,\phi^*,\phi_t^*)=0.$ 
Moreover,  $H\in\mathcal{C}^1$ in a vicinity of  $(x_1,x_2,D,D_t)$.  Differentiating, we get 
\[
\big(H_{x_i}(x_1,x_2,D,D_t)\big)_j= - \Big( \langle D_i,D_j\rangle , \big\langle D_{t,i},D_{t,j}\big\rangle  \Big), \ \ i,j\in \{1,2\}.
\]
Let us show that $H'(x_1,x_2,D,D_t)$ is invertible. In what follows, we proceed by cases, depending on $D=A,B$ or $R$. 

\begin{enumerate}
\item Case $D=B$. Thanks to Lemmas \ref{paridad} and \ref{paridad2}, we have $H'$ diagonal and invertible.  %$\partial_{q_2}\rho^1(M_0)=\partial_{q_1}\rho^2(M_0)=0$.

\smallskip

\item Case $D=A,R$.  Thanks to Lemmas \ref{paridad3} and \ref{paridad4}, we have the same situation as before. 
%tenemos que los términos $\partial_{q_2}\rho^1(M_0)=\partial_{q_1}\rho^2(M_0)=0$.
\end{enumerate}
From the last statements we conclude that the matrix  $H'(x_1,x_2,D,D_t)$ is always invertible. Hence, the Implicit Function Theorem says that, if $\nu_0$ is sufficiently small, and $\nu\in (0,\nu_0)$, we will have unique functions $(\tilde x_1,\tilde x_2)$ in $\mathcal{C}^1$, depending on $(\phi,\phi_t)\in\mathcal{U}(\nu)$, and such that $H(\tilde x_1(\phi,\phi_t), \tilde x_2(\phi,\phi_t), (\phi,\phi_t))=(0,0)$. %Finalmente, escogemos $\tilde{x}_i(\phi,\phi_t)=r_i(\phi,\phi_t)+q_i(\phi,\phi_t)$ en cada caso.

\section{Proof of Lemmas \ref{mu_b_bajada}, \ref{bajada_mu_2k} and  \ref{bajada_mu_kak}}\label{demostraciones_fi}

\subsection{Proof of Lemma \ref{mu_b_bajada}}  First of all, note that from \eqref{Kx_deco} we have that $\mu_K$ in \eqref{mu_kink} satisfies
\[
\mu_K= \frac{\cosh(\beta (x+x_2)) \cos(\alpha x_1) - i\sinh(\bt(x+x_2)) \sin(\al x_1) }{\cosh^2(\beta (x+x_2)) \cos^2(\alpha x_1) + \sinh^2(\bt(x+x_2)) \sin^2(\al x_1)} = \frac1{2\bt} K_x.
\]
Therefore, it is necessary that  $x_1$ do not satisfy \eqref{x1_cond} in order to get $\mu_K$ well-defined for any $x$. In this case, $\mu_K$ is smooth and decays to zero exponentially in space.

\medskip

Proving \eqref{nonzero_K}, notice that since $x_1$ does not satisfy \eqref{x1_cond}, we can use \eqref{mu_kink} and  \eqref{sinK_cosK}:
\[
\begin{aligned}
\int_\mathbb{R} \mu_K \sin\left(\dfrac{K}{2}\right)= & ~\int_\mathbb{R} \dfrac{dx}{\cosh^2(\beta (x+x_2)+i\alpha x_1)} =\dfrac{2}{\beta}.
%\\
%=&~ \dfrac{1}{\beta}\tanh(\beta (x+x_2)+i\alpha x_1)\bigg\vert_{-\infty}^{\infty} =\dfrac{2}{\beta}.
\end{aligned}
\]
Now we prove \eqref{integral}. It is enough to notice that
\begin{align*}
\partial_x\left(\dfrac{\beta^2\sin(2\alpha x_1)-i\alpha^2\sinh(2\beta (x+x_2))}{\alpha\beta(\alpha^2\cosh(\beta (x+x_2))^2+\beta^2\sin(\alpha x_1)^2)}\right)= \Phi_1-\Phi_2,
\end{align*}
where
\begin{align*}
\Phi_1&=  -\dfrac{8\alpha\beta^2\cosh(\beta (x+x_2)+i\alpha x_1)\sinh(\beta (x+x_2))\sin(\alpha x_1)}{\big(\alpha^2\cosh(\beta (x+x_2))^2+\beta^2\sin(\alpha x_1)^2\big)^2}=\mu(x)B_x, \\
\Phi_2&= \dfrac{2i\alpha}{\alpha^2\cosh(\beta (x+x_2))^2+\beta^2\sin(\alpha x_1)^2} =\mu(x)K_t.
\end{align*}
Integrating on $\mathbb{R}$ we obtain $\int_\mathbb{R}\mu\cdot\big(B_x-K_t\big)=-\frac{4i}{\alpha\beta}$, i.e. \eqref{integral}.

\medskip

Let us show \eqref{edomuK}. We have from \eqref{sinK_cosK}  $\beta\cos\left(\frac{K}{2}\right)=-\beta\tanh(\beta (x+x_2)+i\alpha x_1)$, 
hence, from \eqref{mu_kink},
\[
(\mu_K)_x= -\frac{\bt\sinh(\beta (x+x_2)+i\alpha x_1)}{\cosh^2(\beta (x+x_2)+i\alpha x_1)} = \beta\cos\left(\dfrac{K}{2}\right) \mu_K,
% = -\beta\tanh(\beta (x+x_2)+i\alpha x_1) =\beta\cos\left(\dfrac{K}{2}\right),
\] 
which proves \eqref{edomuK}.

 \medskip

In order to finish, we only need to prove \eqref{edomu}. Recall the notation in \eqref{thetas}.
%\be\label{thetas}
%\theta_1:=\alpha x_1, \quad \theta_2:=\beta(x+x_2), \quad  \theta:=\beta(x+x_2)+i\alpha x_1. 
%\ee
First we have
\begin{align*}
\mu_x(x)&=\dfrac{\beta\sinh(\theta)\big(\alpha^2\cosh^2(\theta_2)+\beta^2\sin^2(\theta_1)\big)-\alpha^2\beta\sinh(2\theta_2)\cosh(\theta)}{\big(\alpha^2\cosh^2(\theta_2)+\beta^2\sin^2(\theta_1)\big)^2}
\\ & = \left(\dfrac{\beta\tanh(\theta)\big(\alpha^2\cosh^2(\theta_2)+\beta^2\sin^2(\theta_1)\big)-\alpha^2\beta\sinh(2\theta_2)}{\alpha^2\cosh^2(\theta_2)+\beta^2\sin^2(\theta_1)}\right)\mu(x)
\\ & = \left(\beta \tanh(\theta)-\dfrac{\alpha^2\beta\sinh(2\theta_2)}{\alpha^2\cosh^2(\theta_1)+\beta^2\sin^2(\theta_1)}\right)\mu(x).
\end{align*} 
Consequently, our problem now is to show that 
%\begin{align*}
%\Psi(x)=\dfrac{(\beta-i\alpha)}{2}\cos\left(\dfrac{B+K}{2}\right)+ \dfrac{(\beta+i\alpha)}{2}\cos\left(\dfrac{B-K}{2}\right).
%\end{align*}
%En concreto, debemos probar que 
\begin{align}\label{rhs_mu_breather}
& \beta \tanh(\theta)-\dfrac{\alpha^2\beta\sinh(2\theta_2)}{\alpha^2\cosh^2(\theta_1)+\beta^2\sin^2(\theta_1)} \nonu\\
& \qquad =\dfrac{(\beta-i\alpha)}{2}\cos\left(\dfrac{B+K}{2}\right)+ \dfrac{(\beta+i\alpha)}{2}\cos\left(\dfrac{B-K}{2}\right).
\end{align}

Let us compute explicitly the RHS of the last equation. Using basic trigonometric identities 
\begin{align*}
\cos\left(\dfrac{B\pm K}{2}\right)&=\left(1-\tan^2\left(\dfrac{B\pm K}{4}\right)\right)\left(1+\tan^2\left(\dfrac{B\pm K}{4}\right)\right)^{-1}
\\ & = \dfrac{(1-e^{2\theta})(\alpha^2\cosh^2(\theta_2)-\beta^2\sin^2(\theta_1))\mp4\alpha\beta\cosh(\theta_2)\sin(\theta_1)e^{\theta}}{(1+e^{2\theta})(\alpha^2\cosh^2(\theta_2)+\beta^2\sin^2(\theta_1))}.
\end{align*}
Then, using this the RHS of \eqref{rhs_mu_breather} reads now 
\begin{align}
&\mathrm{RHS}\eqref{rhs_mu_breather} \nonu\\
&=\dfrac{\beta(1-e^{2\theta})(\alpha^2\cosh^2(\theta_2)-\beta^2\sin^2(\theta_1))+4i\alpha^2\beta\cosh(\theta_2)\sin(\theta_1)e^{\theta}}{(1+e^{2\theta})(\alpha^2\cosh^2(\theta_2)+\beta^2\sin^2(\theta_1))}\nonumber
\\ & = \dfrac{\beta\tanh(\theta)(\beta^2\sin^2(\theta_1)-\alpha^2\cosh^2(\theta_2))}{\alpha^2\cosh^2(\theta_2)+\beta^2\sin^2(\theta_1)}+\dfrac{2i\alpha^2\beta\cosh(\theta_2)\sin(\theta_1)}{\cosh(\theta)(\alpha^2\cosh^2(\theta_2)+\beta^2\sin^2(\theta_1))}\nonumber
\\ & =\beta\tanh(\theta)- \dfrac{2\alpha^2\beta\tanh(\theta)\cosh^2(\theta_2)}{\alpha^2\cosh^2(\theta_2)+\beta^2\sin^2(\theta_1)}+\dfrac{2i\alpha^2\beta\cosh(\theta_2)\sin(\theta_1)}{\cosh(\theta)(\alpha^2\cosh^2(\theta_2)+\beta^2\sin^2(\theta_1))}\nonumber
\\ &=\beta\tanh(\theta)-\dfrac{2\alpha^2\beta\cosh(\theta_2)\left(\sinh(\theta)\cosh(\theta_2)-\sinh(i\theta_1)\right)}{\cosh(\theta)(\alpha^2\cosh^2(\theta_2)+\beta^2\sin^2(\theta_1))} \nonumber
\\ & = \beta\tanh(\theta)-\dfrac{2\alpha^2\beta\sinh(\theta_2)\cosh(\theta_2)}{\alpha^2\cosh^2(\theta_2)+\beta^2\sin^2(\theta_1)}\label{truco}
\\ & =\beta\tanh(\theta)-\dfrac{\alpha^2\beta\sinh(2\theta_2)}{\alpha^2\cosh^2(\theta_2)+\beta^2\sin^2(\theta_1)}; \nonumber
\end{align}
where in \eqref{truco} we used 
\begin{align*}
& \sinh(\theta)\cosh(\theta_2)-\sinh(i\theta_1)=\sinh(\theta)\cosh(\theta_2)-\sinh(\theta-\theta_2)
\\ & \quad=\sinh(\theta)\cosh(\theta_2)-\sinh(\theta)\cosh(\theta_2)+\cosh(\theta)\sinh(\theta_2) = \cosh(\theta)\sinh(\theta_2),
\end{align*}
which ends the proof.

\subsection{Proof of Lemma \ref{bajada_mu_2k}}
First we prove \eqref{integral}. Indeed, note that
\begin{align*}
\partial_x\left(\dfrac{\beta^2\sinh^2(2\ga(x+x_2))-\sinh(2\ga x_1)}{\beta\big(\beta^2\sinh^2(\ga(x+x_2))+\cosh^2(\ga x_1)\big)} \right)= \Phi_1-\Phi_2,
\end{align*}
where
\begin{align*}
\Phi_1&=  \dfrac{4\beta\ga \cosh(\ga(x+x_1+x_2))\cosh(\ga x_1)\cosh(\ga(x+x_2))}{\big(\beta^2\sinh^2(\ga(x+x_2))+\cosh^2(\ga x_1)\big)^2}=\mu(x)R_x, \\
\Phi_2&= \dfrac{2\beta\ga }{\beta^2\sinh^2(\ga(x+x_2))+\cosh^2(\ga x_1)} =\mu(x)Q_t.
\end{align*}
Integrating on $\mathbb{R}$ we obtain \eqref{integral}.
%\begin{align*}
%\int_\mathbb{R}\mu(x)\cdot\big(R_x-Q_t\big)=\dfrac{4}{\beta}.
%\end{align*}
We prove now \eqref{mu_edo_2k}. We will compute each term involved in the equation. For the sake of simplicity, we denote 
\[
\theta_1:=\ga x_1 ,\quad \theta_2 :=\ga(x+x_2) , \quad \theta=\ga(x+x_1+x_2).
\]
First we have
\begin{align*}
\mu_x(x)&=\dfrac{\ga\sinh(\theta)\big(\cosh^2(\theta_1)+\beta^2\sinh^2(\theta_2)\big)-\beta^2\ga\sinh(2\theta_2)\cosh(\theta)}{\big(\cosh^2(\theta_1)+\beta^2\sinh^2(\theta_2)\big)^2}
\\ & = \left(\ga\tanh(\theta)-\dfrac{\beta^2\ga\sinh(2\theta_2)}{\cosh^2(\theta_1)+\beta^2\sinh^2(\theta_2)}\right)\mu(x).
\end{align*}
Consequently, our problem now is to show that 
%\[
%\Psi(x)=\dfrac{1}{2a_3}\cos\left(\dfrac{R+Q}{2}\right)+\dfrac{a_3}{2}\cos\left(\dfrac{R-Q}{2}\right).
%\]
%En concreto, debemos probar que 
\[
\ga\tanh(\theta)-\dfrac{\beta^2\ga\sinh(2\theta_2)}{\cosh^2(\theta_1)+\beta^2\sinh^2(\theta_2)}=\dfrac{1}{2a_3}\cos\left(\dfrac{R+Q}{2}\right)+\dfrac{a_3}{2}\cos\left(\dfrac{R-Q}{2}\right).
\]
Let us compute the RHS of the last equation. For this, we use trigonometric identities: 
\begin{align*}
\cos\left(\dfrac{R\pm Q}{2}\right)&=\left(1-\tan^2\left(\dfrac{R\pm Q}{4}\right)\right)\left(1+\tan^2\left(\dfrac{R\pm Q}{4}\right)\right)^{-1}
\\ & = \dfrac{(1-e^{2\theta})(\cosh^2(\theta_1)-\beta^2\sinh^2(\theta_2))\mp 4\beta\cosh(\theta_1)\sinh(\theta_2)e^\theta}{(1+e^{2\theta})(\cosh^2(\theta_1)+\beta^2\sinh^2(\theta_2))}
\end{align*}
Hence, using this last identity, the RHS of \eqref{mu_edo_2k} is reduced to 
\begin{align*}
\mathrm{RHS} &= \dfrac{-\ga(1-e^{2\theta})(\cosh^2(\theta_1)-\beta^2\sinh^2(\theta_2))-4\beta^2\ga\cosh(\theta_1)\sinh(\theta_2)e^\theta}{(1+e^{2\theta})(\cosh^2(\theta_1)+\beta^2\sinh^2(\theta_2))}
\\ & = \dfrac{\ga\tanh(\theta)(\cosh^2(\theta_1)-\beta^2\sinh^2(\theta_2))}{(\cosh^2(\theta_1)+\beta^2\sinh^2(\theta_2))}-\dfrac{2\beta^2\ga\cosh(\theta_1)\sinh(\theta_2)}{\cosh(\theta)(\cosh^2(\theta_1)+\beta^2\sinh^2(\theta_2))}
\\ & = \ga\tanh(\theta)-\dfrac{2 \beta^2\ga\tanh(\theta)\sinh^2(\theta_2)}{\cosh^2(\theta_1)+\beta^2\sinh^2(\theta_2)}-\dfrac{2 \beta^2\ga\cosh(\theta_1)\sinh(\theta_2)}{\cosh(\theta)(\cosh^2(\theta_1)+\beta^2\sinh^2(\theta_2))}
\\ & =\ga\tanh(\theta)-\dfrac{2\beta^2\ga\sinh(\theta_2)\big(\sinh(\theta)\sinh(\theta_2)+\cosh(\theta_1)\big)}{\cosh(\theta)(\cosh^2(\theta_1)+\beta^2\sinh^2(\theta_2))}
\\ & = \ga\tanh(\theta)-\dfrac{2\beta^2\ga\sinh(\theta_2)\cosh(\theta)\cosh(\theta_2)}{\cosh(\theta)(\cosh^2(\theta_1)+\beta^2\sinh^2(\theta_2))}
\\ & = \ga\tanh(\theta)-\dfrac{\beta^2\ga\sinh(2\theta_2)}{\cosh^2(\theta_1)+\beta^2\sinh^2(\theta_2)}.
\end{align*}
The proof is complete.

\subsection{Proof of Lemma \ref{bajada_mu_kak}} Same as the proof of Lemma \ref{bajada_mu_2k}.
%Nos proponemos primero a probar \eqref{integral_mu_kak}. Para ello basta notar que
%\begin{align*}
%\partial_x\left(\dfrac{\sinh(2\ga x_1)-v^2\sinh(2\ga (x+x_2))}{v\big(v^2\cosh^2(\ga(x+x_2))+\sinh^2(\ga x_1)\big)}\right)= \Phi_1-\Phi_2,
%\end{align*}
%donde 
%\begin{align*}
%\Phi_1&=  -\dfrac{4v\gamma\cosh(\gamma(x+x_1+x_2))\sinh(\ga x_1)\sinh(\ga (x+x_2))}{\big(v^2\cosh^2(\ga(x+x_2))+\sinh^2(\ga x_1)\big)^2}=\mu(x)A_x, \\
%\Phi_2&= \dfrac{2v \gamma}{v^2\cosh^2(\ga(x+x_2))+\sinh^2(\ga x_1)} =\mu(x)Q^+_t.
%\end{align*}
%Integrando sobre todo $\mathbb{R}$ obtenemos que \begin{align*}
%\int_\mathbb{R}\mu(x)\cdot\big(A_x-Q^+_t\big)=-\dfrac{4}{v}.
%\end{align*}
%Para concluir sólo resta probar \eqref{edomu}. Para ello, recordando que $A(t,x)$ y $Q^+(t,x)$ satisfacen el sistema de EDP asociados a Bäcklund, obtenemos fácilmente que 

\section{Proof of Lemma \ref{Aux_10}}\label{Lema_aux_10}

\noindent
{\it Proof of (i).} We use the same notation as in \eqref{thetas}. We have
\begin{align*}
 K-4\arctan\left(\dfrac{2i\alpha}{2\beta}\dfrac{\beta\sin(\theta_1)}{\alpha\cosh(\theta_2)}\right) & = ~ 4\arctan\big(e^{\theta}\big)-4\arctan\left(\dfrac{i\sin(\theta_1)}{\cosh(\theta_2)}\right)
\\  &  = ~ 4\arctan\big(e^{\theta}\big)-4\arctan\left(\dfrac{e^{i\theta_1}-e^{-i\theta_1}}{e^{\theta_2}+e^{-\theta_2}}\right).
\end{align*}
Therefore, using that $\arctan(u)-\arctan(v)=\arctan(\frac{u-v}{1+uv}),$ we obtain 
\begin{align}
\phi^{3,1}& = 4\arctan\big(e^{\theta}\big)-4\arctan\left(\dfrac{e^{i\alpha x_1}-e^{-i\theta_1}}{e^{\theta_2}+e^{-\theta_2}}\right)= 4\arctan\left(\dfrac{e^{\theta}-\dfrac{e^{i\theta_1}-e^{-i\theta_1}}{e^{\theta_2}+e^{-\theta_2}}}{1+e^{\theta}\dfrac{e^{i\theta_1}-e^{-i\theta_1}}{e^{\theta_2}+e^{-\theta_2}}}\right)\nonumber
\\ & = 4\arctan\left(\dfrac{e^{\theta}\big(e^{\theta_2}+e^{-\theta_2}\big)-e^{i\theta_1}+e^{-i\theta_1}}{e^{\theta_2}+e^{-\theta_2}+e^{\theta}\big(e^{i\theta_1}-e^{-i\theta_1}\big)}\right)\nonumber
\\ &= 4\arctan\left(\dfrac{e^{2\beta(x+x_2)+i\alpha x_1}+e^{-i\theta_1}}{e^{-\theta_2}+e^{\beta(x+x_2)+2i\alpha x_1}}\right)\nonumber
\\ &= 4\arctan\left(e^{\bar \theta}\cdot\dfrac{e^{\beta(x+x_2)+2i\alpha x_1}+e^{-\theta_2}}{e^{-\theta_2}+e^{\beta(x+x_2)+2i\alpha x_1}}\right)= 4\arctan\big(e^{\bar\theta}\big) = \overline K.\nonu
\end{align}

\noindent
{\it Proof (ii).} The identities in \eqref{AAA} are straightforward. In order to show \eqref{BBB}, we have 
\begin{align*}
\dfrac{B_t\sec^2\left(\frac{B}{4}\right)}{1+\ell^2\tan^2\left(\frac{B}{4}\right)}&= \dfrac{ \left(\dfrac{4\alpha^2\beta\cos(\theta_1)\cosh(\theta_2)}{\alpha^2\cosh^2(\theta_2)+\beta^2\sin^2(\theta_1)}\right)\left(1+\left(\dfrac{\beta\sin(\theta_1)}{\alpha\cosh(\theta_2)}\right)^2\right)}{1+\ell^2\left(\dfrac{\beta\sin(\theta_1)}{\alpha\cosh(\theta_2)}\right)^2}
\\
 & = \dfrac{4\alpha^2\beta\cos(\theta_1)\cosh(\theta_2)}{\alpha^2\cosh^2(\theta_2)+\ell^2\beta^2\sin^2(\theta_1)}.
\end{align*}
%como se ped\'ia.

\section{Proof of \eqref{No_deg}}\label{F}
By the analysis made in Section $3$, it is direct that $\mathrm{Re}\, \big(\widetilde{B}_0K_x\big)$ is even and $\mathrm{Im}\,\big(\widetilde{B}_0K_x\big)$ is odd. Then, from the fact that $\widetilde{B}_0K_x$ belongs to the Schwartz class, whenever $x_1$ does not satisfy \eqref{x1_cond}, we conclude  
\[
\int_\mathbb{R}\widetilde{B}_0K_x \in \mathbb{R}.
\]
Now, to show that the integral is different from zero, we will separate the analysis in two cases. Since for $\beta=0$ we have that $\widetilde{B}_0\equiv 0$, we will first analyze the behavior of the integral \eqref{No_deg} as a function of $\beta$ in the cases for which $\beta\sim 0$. Recall that since the integrand is periodic in time, it is enough to analyze its value between one period. From now on we will denote $I(x_1,\beta)$ to the integral in \eqref{No_deg}. 
\\

By numeric simulations performed in \texttt{Mathematica} we obtain that, for any value of $x_1$ fixed, the function $\tilde I(\beta):=I(x_1,\beta)$ is always different from zero for any small $\beta\neq 0$. Fig. \ref{beta_nd} shows the behavior of $\tilde I(\beta)$ for different values $x_1\in \{0,...,6\}$ fixed.
\begin{figure}[h!] \centering
{\includegraphics[scale=0.2]{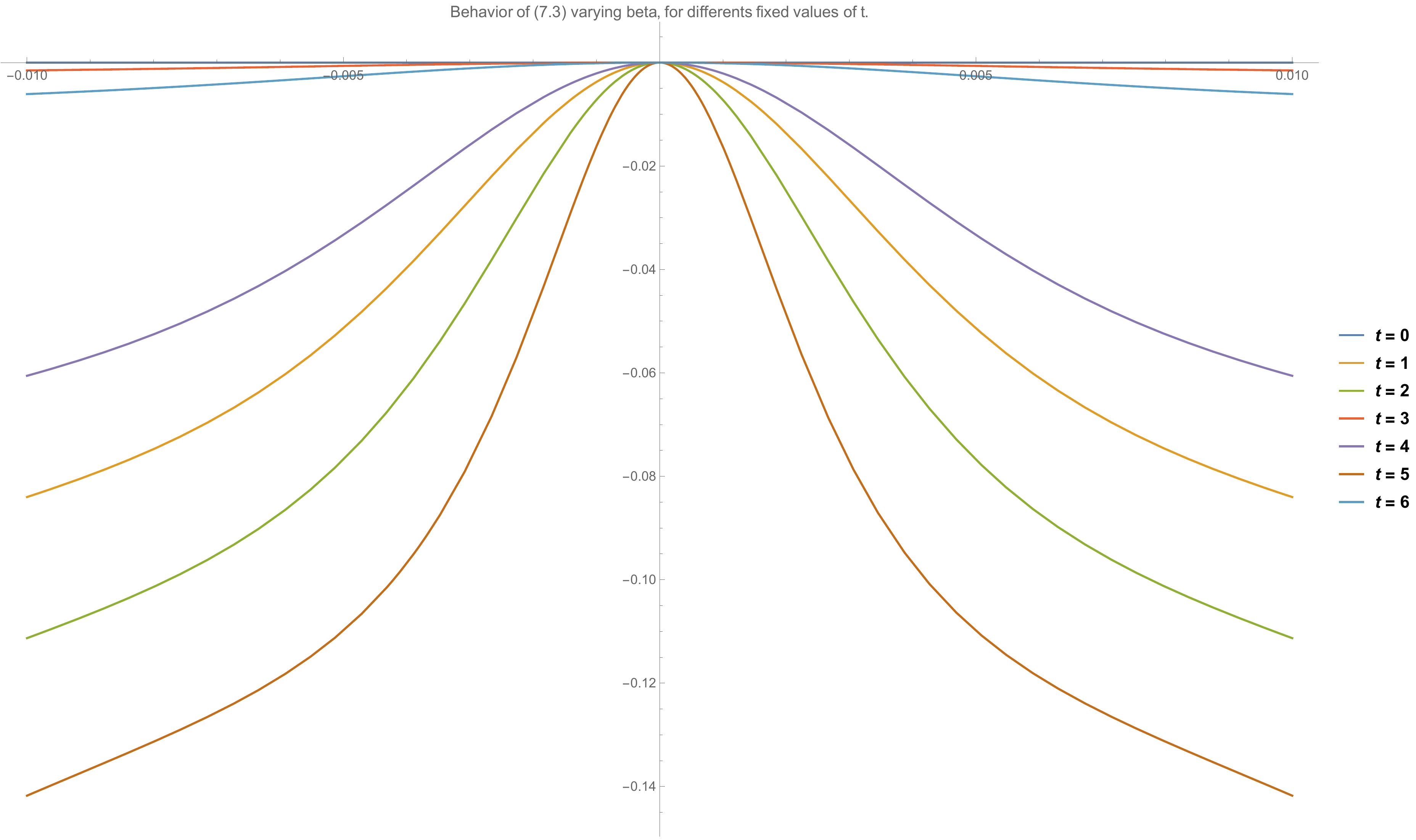}} 
%\hspace{0.15cm} %Espacio horizontal
\caption{Behavior of $I(x_1,\beta)$ for $\beta\sim 0$ and fixed time values $x_1$.}\label{beta_nd}
\end{figure}
Figs. \ref{int_tiempo1} and \ref{int_tiempo3} show the behavior of $I(x_1,\beta)$ as a function of $x_1$, for different values of $\beta>0.1$ fixed. In each of them, the horizontal axis represents the time variable $x_1$ and the vertical axis represents the values of $I(x_1,\beta)$. It is clear that in each figure $\tilde I(x_1):=I(x_1,\beta)$ keeps uniformly far from zero. The vertical lines in Fig. \ref{int_tiempo1} correspond to the values of $x_1$ on which condition \ref{x1_cond} is satisfied. For other values of $\beta$ we found exactly the same figure.

\begin{figure}[h!] 
\includegraphics[scale=0.18]{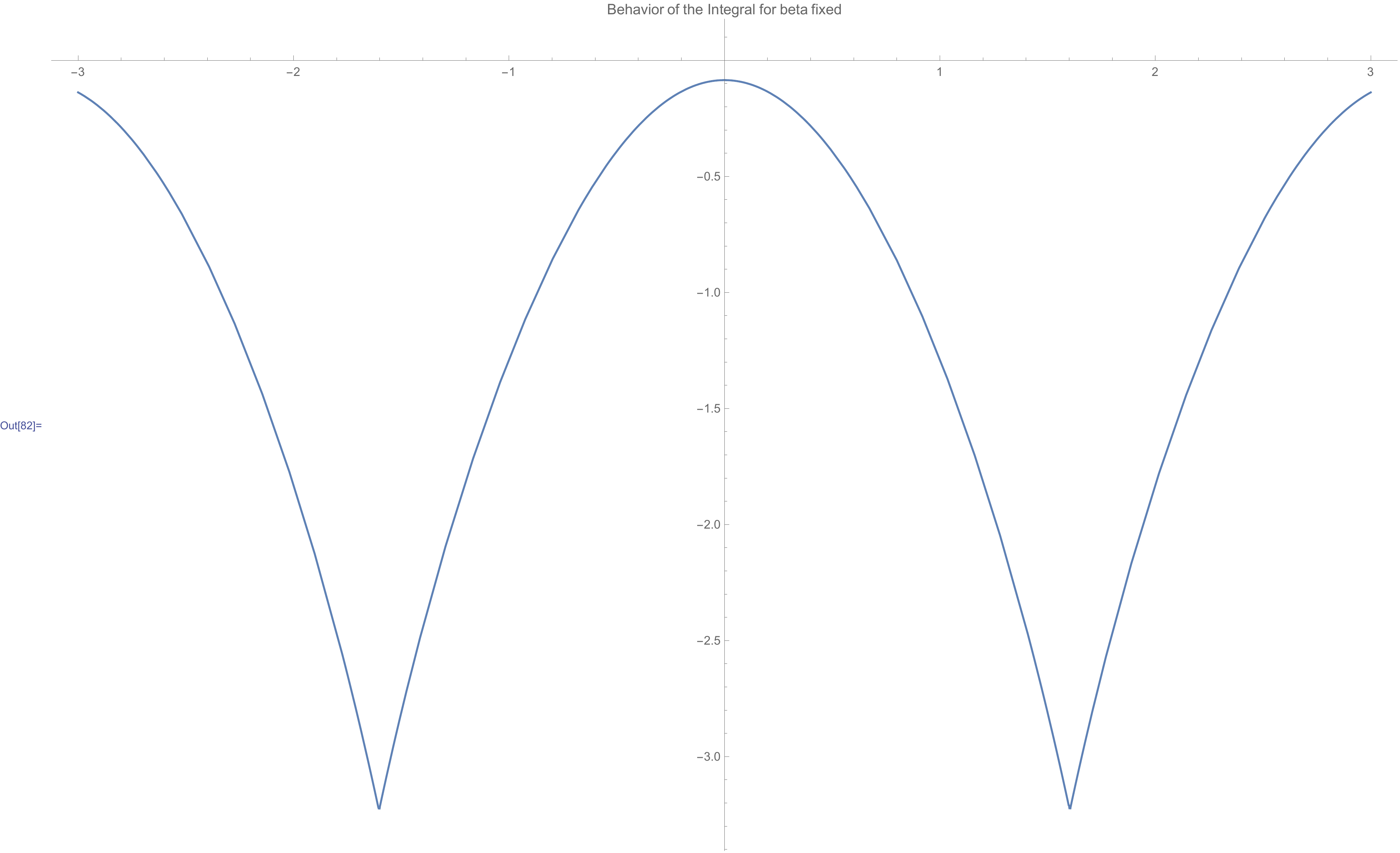}
%\hspace{0.15cm} %Espacio horizontal
\includegraphics[scale=0.18]{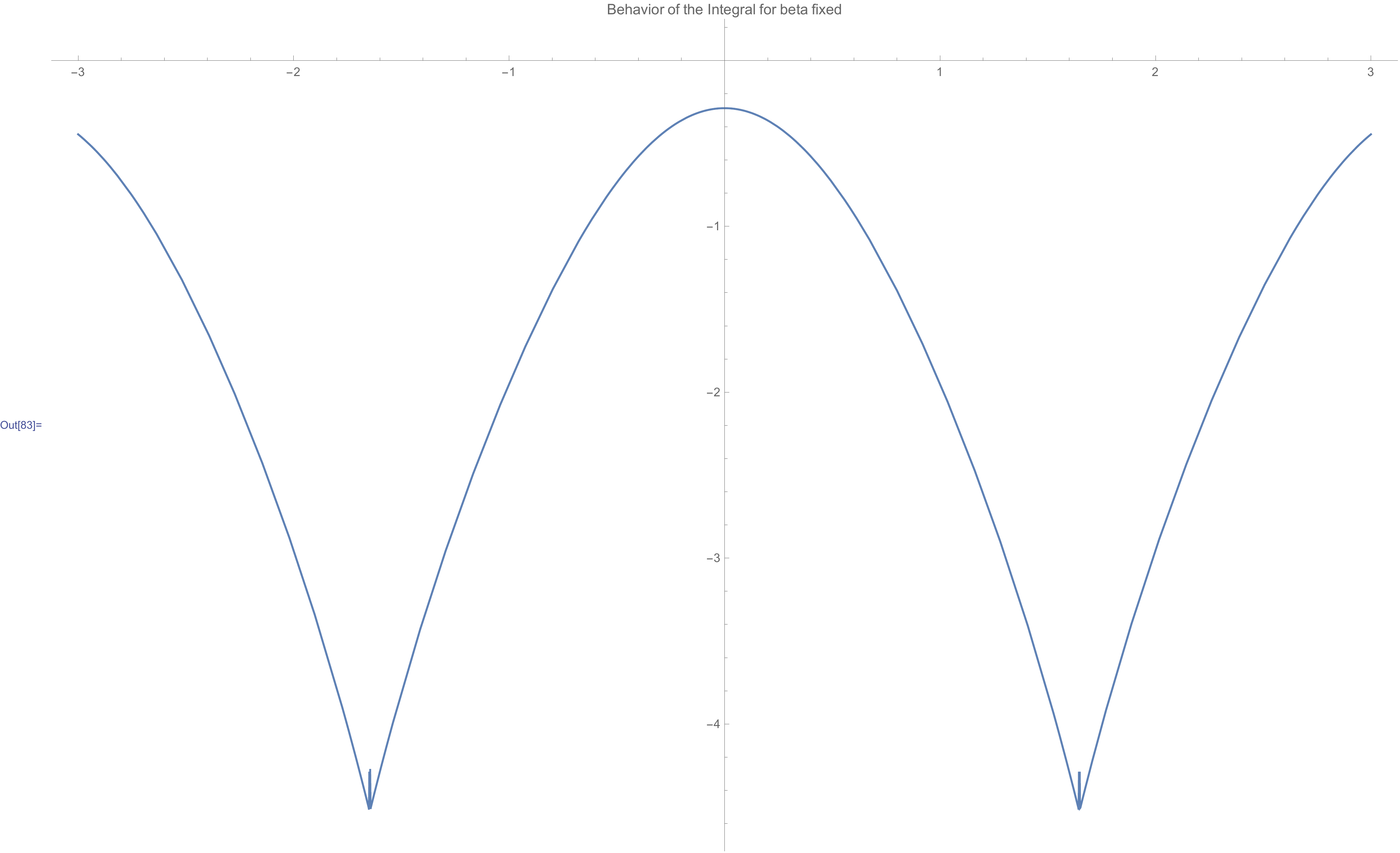}
\caption{Behavior of $I(x_1,\beta)$ in time for $\beta=0.2$ and $\beta=0.3$.}\label{int_tiempo1}
%\caption{Behavior of $I(x_1,\beta)$ in time for $\beta=0.3$.} \label{int_tiempo2}
\end{figure}

\begin{figure}[h!] \centering
{\includegraphics[scale=0.2]{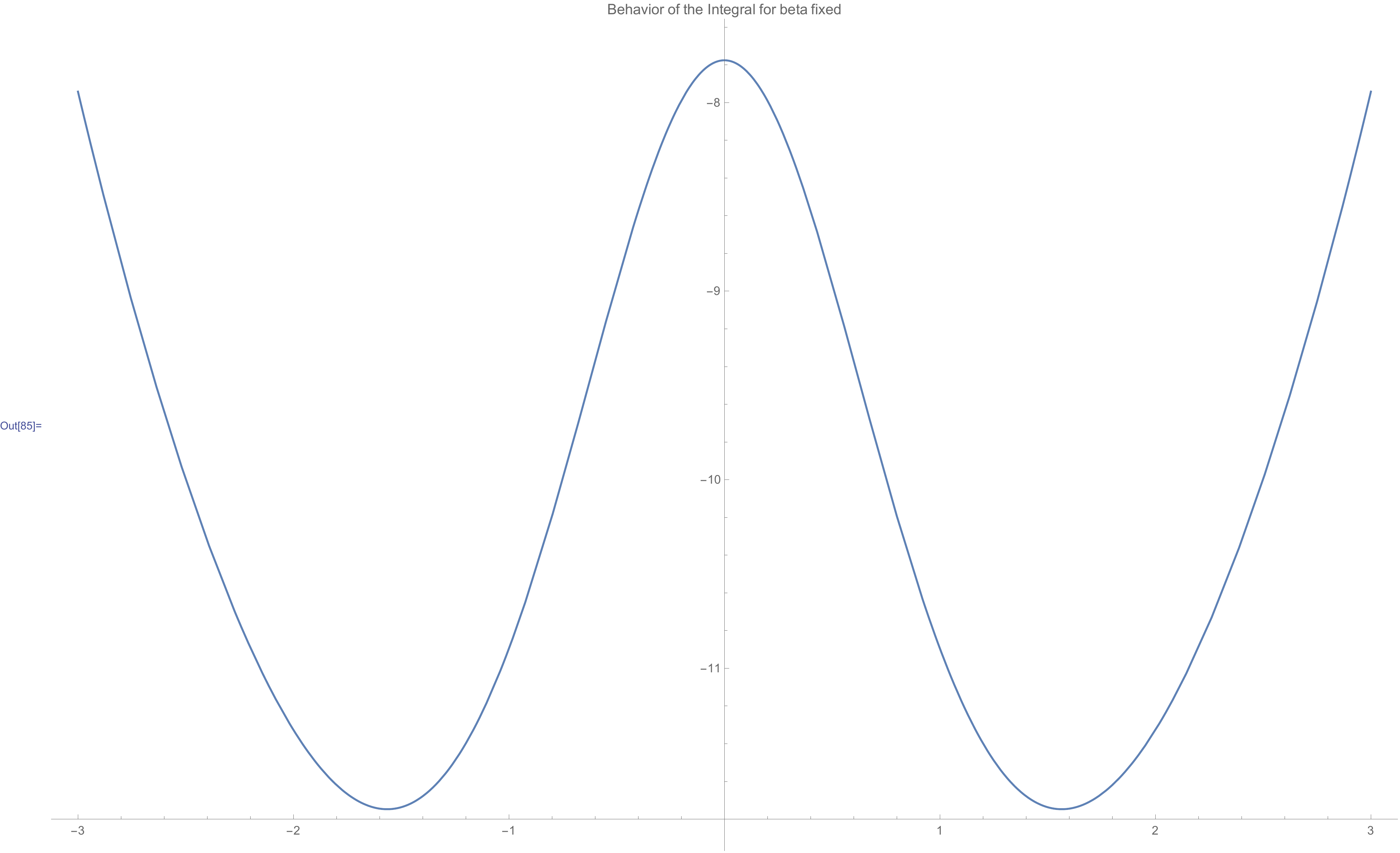}} 
\caption{Behavior of $I(x_1,\beta)$ in time for $\beta=0.9$.} \label{int_tiempo3}
\end{figure}
% ============= BIBLIOGRAFIA ==============
%\newpage

%%%%%%%%%%%%%%%%%%%%%%
%%%%%%%%%%%%%%%%%%%%%%
%%%%%%%%%%%%%%%%%%%%%%
%%%%%%%%%%%%%%%%%%%%%%

% ============= FIN DE DOCUMENTO ==============
\end{document}